\documentclass[11pt,leqno]{article}
\usepackage{amsthm,amsfonts,amssymb,amsmath}
\usepackage{epsfig}
\numberwithin{equation}{section}
\usepackage[thinlines]{easybmat}
\usepackage{epstopdf}

\newcommand{\DDD}{D3}

\renewcommand\d{\partial}

\def\eps{\varepsilon }

\renewcommand\d{\partial}

\newcommand\R{\mathbb R}
\newcommand\C{\mathbb C}

\def\eps{\varepsilon}

\newcommand\errfn{\textrm{errfn}}
\newcommand\br{\begin{remark}}
\newcommand\er{\end{remark}}
\newcommand\bp{\begin{pmatrix}}
\newcommand\ep{\end{pmatrix}}
\newcommand\be{\begin{equation}}
\newcommand\ee{\end{equation}}
\newcommand\ba{\begin{equation}\begin{aligned}}
\newcommand\ea{\end{aligned}\end{equation}}

\newcommand{\bap}{\begin{app}}
\newcommand{\eap}{\end{app}}
\newcommand{\begs}{\begin{exams}}
\newcommand{\eegs}{\end{exams}}
\newcommand{\beg}{\begin{example}}
\newcommand{\eeg}{\end{exaplem}}
\newcommand{\bpr}{\begin{proposition}}
\newcommand{\epr}{\end{proposition}}
\newcommand{\bt}{\begin{theorem}}
\newcommand{\et}{\end{theorem}}
\newcommand{\bc}{\begin{corollary}}
\newcommand{\ec}{\end{corollary}}
\newcommand{\bl}{\begin{lemma}}
\newcommand{\el}{\end{lemma}}
\newcommand{\bd}{\begin{definition}}
\newcommand{\ed}{\end{definition}}
\newcommand{\brs}{\begin{remarks}}
\newcommand{\ers}{\end{remarks}}

\newtheorem{theo}{Theorem}[section]
\newtheorem{prop}[theo]{Proposition}
\newtheorem{cor}[theo]{Corollary}
\newtheorem{lem}[theo]{Lemma}

\newtheorem{exams}[theo]{Examples}

\numberwithin{equation}{section}

\newcommand{\RR}{{\mathbb R}}

\newcommand{\const}{\text{\rm constant}}
\newcommand{\Id}{{\rm Id }}

\newtheorem{theorem}{Theorem}[section]
\newtheorem{proposition}[theorem]{Proposition}
\newtheorem{corollary}[theorem]{Corollary}
\newtheorem{lemma}[theorem]{Lemma}
\newtheorem{definition}[theorem]{Definition}

\newtheorem{example}[theorem]{Example}
\newtheorem{remark}[theorem]{Remark}

\newcommand\cR{{\mathcal  R}}

\newcommand\cN{{\mathcal  N}}

\newcommand\cQ{{\mathcal Q}}

\newcommand\cT{{\mathcal T}}
\newcommand\cS{{\mathcal S}}

\newtheorem{corr}{Corollary}

\newcommand{\RM}{\mathbb{R}}
\newcommand{\ZM}{\mathbb{Z}}
\newcommand{\QM}{\mathbb{Q}}
\newcommand{\CM}{\mathbb{C}}
\newcommand{\NM}{\mathbb{N}}

\newcommand{\tr}{\,\mbox{\rm tr}}

\newcommand{\cn}{\operatorname{cn}}

\newcommand{\qut}[1]{\textquotedblleft #1\textquotedblright}

\renewcommand{\Re}{\hbox{\rm Re}}
\renewcommand{\Im}{\hbox{\rm Im}}

\title{
Nonlinear modulational stability of periodic traveling-wave solutions of
the generalized Kuramoto--Sivashinsky equation}

\author{\sc \small
Blake Barker\thanks{Indiana University, Bloomington, IN 47405;
bhbarker@indiana.edu: Research of B.B. was partially supported
under NSF grants no. DMS-0300487 and DMS-0801745.}
~~~
Mathew A. Johnson\thanks{Department of Mathematics, University of Kansas, 1460 Jayhawk Boulevard, 
Lawrence, KS 66045; matjohn@math.ku.edu: Research of M.J. was partially supported by an NSF Postdoctoral
Fellowship under DMS-0902192.}
~~~
Pascal Noble\thanks{Universit\'e de Lyon, Universit\'e Lyon I, Institut Camille Jordan, UMR CNRS 5208, 43 bd du 11 novembre 1918, F - 69622 Villeurbanne Cedex, France; noble@math.univ-lyon1.fr:
Research of P.N. was partially supported by the French ANR Project no.
ANR-09-JCJC-0103-01.}
~~~
\\
\sc \small
L. Miguel Rodrigues\thanks{Universit\'e de Lyon, Universit\'e Lyon 1,
Institut Camille Jordan, UMR CNRS 5208, 43 bd du 11 novembre 1918,
F - 69622 Villeurbanne Cedex, France; rodrigues@math.univ-lyon1.fr: Stay of M.R. in Bloomington was supported by 
the French ANR Project no. ANR-09-JCJC-0103-01.}
~~~
Kevin Zumbrun\thanks{Indiana University, Bloomington, IN 47405;
kzumbrun@indiana.edu:
Research of K.Z. was partially supported
under NSF grant no. DMS-0300487.}
}
\begin{document}

\maketitle


\begin{center}
\end{center}

\begin{center}
{\bf 2010 MR Subject Classification}: 35B35, 35B10.
\end{center}


\begin{abstract}
In this paper we consider the spectral and nonlinear stability of periodic traveling wave solutions of a generalized Kuramoto-Sivashinsky equation. In particular, we resolve the long-standing question of nonlinear modulational stability by demonstrating that 
spectrally stable waves are 
nonlinearly stable when subject to small localized (integrable) perturbations.  Our
analysis is based upon detailed estimates of the linearized solution operator, 
which are complicated by the fact that the
(necessarily essential) spectrum of the associated linearization intersects the imaginary axis at the origin. 
We carry out a numerical Evans function
study of the spectral problem and find bands of 
spectrally stable periodic traveling waves, 
in close agreement with previous numerical studies of 
Frisch--She--Thual, 
Bar--Nepomnyashchy, Chang--Demekhin--Kopelevich, and others carried
out by other techniques.
We also compare predictions of the associated Whitham modulation equations, which formally describe the dynamics
of weak large scale perturbations of a periodic wave train, with numerical time evolution studies,
demonstrating their effectiveness at a practical level.
For the reader's convenience,  we include in an appendix 
the corresponding treatment of the Swift-Hohenberg equation,
a nonconservative counterpart of the generalized Kuramoto-Sivashinsky
equation for which the nonlinear stability analysis is considerably simpler,
together with numerical Evans function analyses extending
spectral stability analyses of Mielke and Schneider.
\end{abstract}

\tableofcontents

\bigbreak

\section{Introduction}\label{intro}
Localized coherent structures such as 
solitary waves play an essential role as elementary processes in 
nonlinear phenomena. 
Examples of this are multi-bump solutions in reaction 
diffusion equations, which are constructed by piecing together 
well-separated solitary waves \cite{E}, or the limiting case of infinite,
periodic wave trains.
A similar situation occurs in nonlinear dispersive media described by a KdV equation where exact multi-bump 
and periodic
solutions exist. In this paper, we consider 
{\it periodic solutions} of an {\it unstable 
dissipative-
dispersive} nonlinear equation, namely a generalized Kuramoto-Sivashinsky (gKS) equation 
\be\label{e:KS}
u_t+\gamma \partial_x^4u+\eps \partial_x^3 u +
\delta \partial_x^2u+\partial_x f(u) =0,
\quad
\gamma,\delta >0,
\ee
where $f(u)$ is an appropriate nonlinearity and $\eps,\gamma\in\RM$ are arbitrary constants with $\gamma> 0$.  In the case $f(u)=\frac{u^2}{2}$, equation \eqref{e:KS} is a canonical model for pattern formation that has been used to describe, variously, plasma instabilities, flame front propagation, turbulence in 
reaction-diffusion systems and
nonlinear waves in fluid mechanics \cite{S1,S2,K,KT,CD,PSU}. 

Equation (\ref{e:KS}) may be derived formally either from shallow water 
equations \cite{YY} 
or from the full Navier-Stokes system \cite{W} for $0<\gamma=\delta\ll 1$. Here $\delta$ measures the deviation of the Reynolds number from the critical Reynolds number above which large scale weak perturbations are spectrally unstable. For this latter application, what we have in mind is the description of chaotic motions in thin film flows down an incline \cite{CD}. Indeed, periodic traveling waves are some of the few simple solutions in the attractor for 
the classic ($\eps=0$) 
Kuramoto-Sivashinsky equation, 
a generic equation for chaotic 
dynamics, and there is now a huge literature on these solutions (and their bifurcations, in particular period doubling cascades) and their stability;  see \cite{FST, KE, K, CD}.
As $\eps$ increases, the set of stable periodic waves, and presumably
also their basin of attraction appears (numerically) to enlarge \cite{CDK,BaN},
until, in the 
$|(\gamma, \delta)|\to 0$ 
limit, they and 
other approximate superpositions of solitary waves
appear to dominate asymptotic behavior \cite{CD,PSU,BJRZ,BJNRZ2}.

Since $\delta>0$ it is easily seen via Fourier analysis that all constant solutions of \eqref{e:KS} are unstable, from which it follows
that all asymptotically constant solutions (such as the solitary waves) are also unstable.  Nevertheless,
one can still construct multi-bump solutions to (gKS) on asymptotically large time $O(\delta^{-1})$ by 
gluing together 
solitary waves,
provided that the distance between them is not too large \cite{PSU}. One possible interpretation of this is that there exist stable periodic wave trains nearby the 
solitary wave.
Indeed, it has been known, almost since the introduction of the classical Kuramoto-Sivashinsky equation \eqref{e:KS} ($\eps=0$) in 
1975,
that there exists a spectrally stable band of periodic solutions in parameter space; see for example the numerical studies in 
\cite{CKTR,FST}. 
These stable periodic wave trains may be heuristically viewed as a superposition of infinitely many well separated
solitary waves.  In \cite{EMR}, the existence of such a band of stable periodic traveling waves was justified for 
the equation \eqref{e:KS} with {\it periodic boundary conditions} and in the 
singular KdV limit 
$|(\gamma, \delta)|\to 0$.
%

Although numerical time-evolution experiments suggest 
that
these spectrally stable waves are 
nonlinearly stable as well (see \cite{CD}), up to now this conjecture had not been rigorously verified.  In this paper, relying heavily on the recent infusion of new tools in \cite{JZ1,JZN,BJRZ,BJNRZ1,BJNRZ2} in the context of 
general conservation laws and the St. Venant (shallow water) equations,
we prove the 
result, previously announced in \cite{BJNRZ3}, 
that 
{\it spectral modulational stability of periodic solutions of
\eqref{e:KS}}, defined in the standard sense, 
{\it implies linear and nonlinear modulational stability to small localized} (integrable) 
{\it perturbations}; that is, a localized perturbation of a periodic traveling wave converges to 
a periodic traveling wave that is modulated in phase. 
The first such nonlinear result for any version of \eqref{e:KS},
this closes in particular the 35-year old open question of
nonlinear stability of spectrally stable periodic waves of
the classical Kuramoto-Sivashinsky equation ($\eps=0$) found
numerically
 in \cite{FST}.

With 
these
improvements in nonlinear theory, we find this 
also 
an opportune moment to make a definitive discussion of 
the
generalized Kuramoto-Sivashnisky equation (and Swift-Hohenberg equations) in terms of existence, nonlinear theory, and numerical spectral stability studies, all three, across all parameters, 
both connecting to and greatly generalizing the 
variety of prior works \cite{FST, KE,K,CD,M1,Sc,CDK,BaN}. 
We thus carry out also a
numerical analysis of the spectrum in order to check the spectral assumptions made in our main theorem. Since the spectrum always contains the origin, 
due to both translational invariance and the 
conservative form of the equations, 
this stability is at best marginal, 
a circumstance that greatly complicates both 
numerical and analytical
stability analyses. 
%

Our numerical  approach is based on complementary tools; 
namely Hills method and the Evans function. On the one hand, we use 
SpectrUW numerical software \cite{CDKK} based on Hills method, 
which is a Galerkin method of approximation, in order to obtain a good overview of location of the spectrum: the periodic coefficients and eigenvectors  are expanded using Fourier  series, and then a frequency cutoff is used to reduce the problem to finding 
eigenvalues of a finite dimensional matrix.  
It is known 
that Hills method converges 
faster than any algebraic order \cite{JZ5};
moreover, in practice, it gives quickly a reliable global
qualitative picture.
However, the associated error bounds are of abstract nature,
with coefficients whose size is not a priori guaranteed.
Further, near the critical zone around $\lambda=0$,
the resolution of this method is not in practice sufficient to guess
at stability, let alone obtain satisfactory numerical verification.

Thus,
in order to get more reliable pictures near the origin and guarantee the spectral stability of periodic wave, we use
on the other hand
 an Evans function approach, computing a winding number to 
prove that there is or is not 
unstable spectrum in the part of the unstable (positive real part)
complex half plane excluding a small neighborhood of the origin,
then using Cauchy's integral formula to 
determine the Taylor expansions 
of the spectral curves passing through the origin. 
This method, though cumbersome for approximating global spectrum,
is excellent for excluding the existence of spectra on a given region,
and comes with error bounds that can in a practical sense
be prescribed via the tolerance of the Runga--Kutta 4-5 scheme 
used to evaluate the Evans function by solution of an appropriate ODE;
see \cite{Br,Z1,BJNRZ1} for general discussion of convergence
of Evans function methods.  Furthermore, under generic assumptions, the numerical protocol introduced in Section \ref{sec:lowfreqanal} below
detects sideband stability and instability of the underlying periodic wave train without 
the need of lengthy spectral perturbation expansion calculations, thus adding
what we believe is a valuable new method to the numerical toolbox for analyzing the spectral
stability of periodic wave trains.
For relations between Hill's method, Fredholm determinants, and
the periodic Evans function of Gardner, see \cite{JZ5,Z2}.

In order to validate our 
numerical
method, we check benchmarks from the Kuramoto-Sivashinsky equation ($\eps=0, f(u)=u^2/2$) and the Swift Hohenberg equation. We obtain very good agreement with existing numerical works.
In particular, we recover and extend 
stability boundaries found numerically for KS in \cite{FST}
and analytically for Swift-Hohenberg in \cite{M1,Sc}. Then we perform an all-parameters study linking in particular the Kuramoto-Sivashinsky equation $\eps\to 0$ $\gamma=\delta=1$ to KdV equation $\eps=0, \gamma=\delta\to 0$  by 
homotopy, again obtaining excellent agreement with existing results of
\cite{CDK,BaN}.
A rigorous proof of spectral stability of periodic traveling wave solutions of \eqref{e:KS}, in the spirit of \cite{M1,Sc}, or a numerical proof (interval arithmetic) would be a very interesting direction for future work, 
particularly in an interesting
parameter regime such as the canonical KDV limit 
$|(\gamma, \delta)|\to 0$ 
studied in \cite{PSU}.
In this regard, we point to the recent singular perturbation
analysis of this limit in \cite{BaN} and \cite{JNRZ4}, the second 
relying on and completing the first, reducing the problem of rigorous
validation of stability/instability as 
$|(\gamma, \delta)|\to 0$ 
to computer-verified evaluation of the signs of certain elliptic integrals
associated with KdV.

 
As we will see, our main theorem indicates that the long time dynamics of a localized perturbation of a given
periodic traveling wave of \eqref{e:KS} is governed by a space-time dependent 
phase 
modulation $\psi(x,t)$, which in turn
satisfies in some sense a set of modulation equations.  Moreover, it is well known that, in the neighborhood 
of the origin, the spectrum of the linearization about a given periodic wave train is strongly related to some local well-posedness
properties of the associated linearized Whitham system; this set of modulation equations may be derived via a nonlinear optics
expansion (WKB), and formally govern the evolution
of weak large scale perturbations of the underlying wave train.  

In \cite{FST}, the authors derive for the classic
Kuramoto-Sivashinsky equation ($\eps=0$) a second order modulation equation in the phase $\psi$ of the form
$\psi_{tt}-a\psi_{xx}=b\psi_{txx}$, which is indeed an alternative formulation of the linearized Whitham equations;
more generally, the Whitham equations for \eqref{e:KS}
consist of a system of two first-order conservation laws \cite{NR2}.
The signs of the coefficients $a$ and $b$ depend implicitly on the underlying wave train, and 
determine
the spectral stability in the low frequency/low wave number regime; see again \cite{NR2}.
The fact that Whitham's equations determines (a part of) the spectral stability of periodic wave trains in the low frequency regime 
has also been established for viscous conservation laws in \cite{Se}, for generalized KdV equations in \cite{JZ4} and \cite{BrJZ}, and for
the shallow water equations with a source term in \cite{NR1}. 

Here, in terms of the modulation equations
we have two requirements for low-frequency modulational stability: 
reality of the characteristics of the first order 
part (in \cite{FST}, positivity of sound speed squared; more generally, positivity of a certain 
discriminant); and positivity of the various diffusion parameters
associated with different characteristic modes (which are equal 
in the classic case \cite{FST},
but in general distinct \cite{NR1,NR2}).  

We point out that the justification of such Whitham equations has become an important direction
of research in the last decade, with important and fundamental results being given, for example, in the context
of the reaction diffusion equations \cite{DSSS} and the Shallow water equations \cite{NR1}.
In this paper, we carry out numerical 
time evolution studies that
illustrate this correspondence with the Whitham equations; see section \ref{num:whitham}.
In the forthcoming paper \cite{JNRZ1}, the authors provide a rigorous justification of Whitham's equations in 
the context of viscous conservations laws, which applies in particular to equation \eqref{e:KS}. 
The approach proposed 
there extends readily to other second and 
higher order parabolic equations such as the Cahn-Hilliard equation, general fourth order thin film models,
and to general systems of quasilinear $2r$-parabolic equations 
such as the Swift Hohenberg equation. 
 
The paper is organized as follows: 
in the current section, we introduce the assumptions of our main theorem, both on the set of periodic traveling wave solutions of
\eqref{e:KS} and on the spectrum of the linearization of \eqref{e:KS} about such a wave.  After stating our main theorem
 we continue
with section \ref{s:num} where we, among other things, check numerically that the spectral stability assumptions
of our main theorem are satisfied for some families of periodic traveling wave solutions of \eqref{e:KS}.  In particular,
we show that there exist bands in parameter space of spectrally stable periodic traveling waves and verify that these
bands agree in particular asymptotic limits with previous numerical studies 
\cite{FST,CDK,BaN}.
Section \ref{s:proof} is dedicated to the proof of the main theorem of our 
paper: that is, that spectral modulational
stability implies linear and nonlinear modulational stability to small localized perturbations.  The proof is an adaptation of the
analysis recently given for the St. Venant equations and the case of general, second order, viscous conservation laws; see \cite{JZN,JZ1}.

For completeness, we present in Appendix \ref{a:survey} a survey of the existence theory in the periodic context
in a variety of asymptotic regimes; small amplitude, near the classic Kuramoto-Sivashinsky equation ($\eps\to 0$, $\gamma=\delta=1$), and
near the KdV equation ($\eps=1$, $\gamma=\delta\to 0$).  
In Appendix \ref{a:sh}, we prove a nonlinear stability result, 
analogous to our main theorem, in the case of the Swift-Hohenberg equation.
The nonconservative nature of this equation
makes the nonlinear stability analysis considerably simpler than that presented in Section \ref{s:proof}, and the reader unfamiliar
with the techniques of Section \ref{s:proof} may find it helpful to read the Swift-Hohenberg analysis as a 
precursor
to Section \ref{s:proof}.  Also in Appendix \ref{a:sh} we conduct a numerical study similar to but 
less detailed than that
given in Section \ref{s:num}.  In particular, this numerical study complements and extends the analytical results of \cite{M1,Sc}.
In Appendices \ref{s:stats} and \ref{s:alg}, we describe, 
respectively, computational statistics and the shooting algorithm
\cite{BJNRZ1} used to 
estimate the 
Evans function at a given frequency pair $(\lambda, \xi)$.
Finally, in Section \ref{s:behavior}, we estimate the rate of growth
of perturbations of unstable waves lying near stability boundaries
of hyperbolic vs. diffusive type.

\subsection{Equations and assumptions}
Consider the generalized Kuramoto-Sivashinsky (gKS) equation given in \eqref{e:KS}, written here in the form
\be\label{e:gen}
\partial_t u+\partial_x^4u+ \eps \partial_x^3u+ \delta \partial_x^2u+\partial_x f(u)=0
\ee
with $f\in C^2(\RM)$ and $\delta>0$.  The main
goal of this paper is to establish that spectrally stable 
(in a sense made precise in (D1)--(D3), Section \ref{s:stabassumptions})
periodic traveling wave solutions to
\eqref{e:gen}
are 
nonlinearly stable with respect to small localized perturbations.  As \eqref{e:gen} is conservative,
the 
stability analysis is an adaptation of that of \cite{JZ1,JZN}
in the second-order viscous conservation law and hyperbolic-parabolic system cases, respectively.  The
main new aspect here is to show that the principles contained in these previously considered
cases extend to equations with higher-order derivatives.

It should be noted that this analysis applies also, with slight modifications, to more general
quasilinear equations (see \cite{JZN}), to the Cahn-Hilliard equation, and to other fourth-order models for thin
film flow as discussed, for example, in \cite{LP,BMS,Ho}.\footnote{In these cases, which concerned
scalar equations, periodic waves were shown to be unstable in a wide
variety of circumstances.}  Indeed, the argument and results extend to arbitrary $2r$-order parabolic
systems, 
so are
essentially completely general for the diffusive case.
(As already seen in \cite{JZ1}, they can apply also to mixed-order and
relaxation type systems in some cases as well.)

We emphasize here as in the introduction above
that if one considers rather the 
generalized
Swift-Hohenberg equation
\be\label{e:sh}
\partial_t u+\partial_x^4u+ \eps \partial_x^3u+ \delta \partial_x^2u+ f(u)=0,
\ee
considered here as a non-conservative counterpart of \eqref{e:gen}, the verification
of nonlinear stability of spectrally stable periodic traveling waves is considerably 
less complicated.
Indeed, this is reminiscent of the distinction between general second-order 
parabolic conservation
laws and systems of reaction diffusion equations; see \cite{JZ1} and \cite{JZ2}.  To aid in the understanding
of 
the forthcoming analysis of the conservative equation \eqref{e:gen},
therefore,
we present in Appendix \ref{a:sh} the proof of the analogous
nonlinear stability result for periodic traveling wave solutions of \eqref{e:sh}. Our approach 
there
recovers and slightly extends the results of Schneider \cite{Sc}, yielding the same heat kernel rate of decay with respect to localized 
perturbations while removing the assumption  that nearby periodic waves have constant speed.

\subsubsection{Traveling-wave equation and structural assumptions}

Throughout, we consider traveling wave solutions of \eqref{e:gen} or, equivalently, stationary solutions of the traveling gKS equation
\begin{equation}\label{e:travks}
\partial_t u-c\partial_x u+\partial_x^4u+\eps\partial_x^3u+\delta\partial_x^2 u+\partial_xf(u)=0
\end{equation}
for some wave speed $c\in\RM$.  Such profiles are necessarily solutions of the traveling wave ODE
\be\label{e:tw}
-c u'+u''''+\eps u'''+ \delta u'' + f(u)' =0,
\ee
which, integrated once, 
reads
\be\label{e:inttw}
-c u+u'''+\eps u''+ \delta u' + f(u) =q,
\ee
where $q\in \R$ is a constant of integration.
It follows 
that periodic solutions of \eqref{e:gen}
correspond to values
$(X,c,q,b)\in \RR^6$, where $X$, $c$, and $q$ denote period,
speed, and constant of integration, and $b=(b_1,b_2,b_3)$ denotes
the values of $(u,u',u'')$ at $x=0$, such that
the values of $(u,u',u'')$ at $x=X$ of the solution of
\eqref{e:tw} are equal to the initial values $(b_1,b_2,b_3)$.

Following \cite{Se,JZ1}, we make the following technical assumptions:
\begin{itemize}
\item[(H1)] $f\in C^{K+1}(\RM)$, $K\ge 5$.
\item[(H2)] The map $H: \,
\R^6  \rightarrow \R^3$	
taking $(X,b,c,q) \mapsto (u,u',u'')(X; b,c,q)-b$
is full rank at $(\bar{X},\bar b, \bar c,\bar q)$,  where $(u,u',u'')(\cdot; b,c,q)$ is the solution of \eqref{e:inttw} such that 
$$
(u,u',u'')(0; b,c,q)=b.
$$
\end{itemize}
By the Implicit Function Theorem,
conditions (H1)--(H2) imply that the set of periodic solutions
in the vicinity of the $\bar X$-periodic function 
$\bar u=u(\cdot;\bar b, \bar c,\bar q)$ 
forms a smooth $3$-dimensional manifold
\be\label{e:manifold}
\left\{\ (x,t)\mapsto U(x-\alpha-c(\beta)t;\beta)\ \middle|\ (\alpha,\beta)\in\RM\times\mathcal{U}\ \right\},
\quad
\hbox{\rm with $\mathcal{U}\subset \RR^{2}$}.
\ee

\br\label{H2rmk}
\textup{
As noted in \cite{JZN}, the transversality assumption (H2)
is necessary for our notion of spectral stability; 
see (D3) in Section \ref{s:stabassumptions} below.
Hence, there is no loss of generality in making this assumption.
}
\er

\subsubsection{Linearized equations and the Bloch transform}\label{s:lin}
To begin our stability analysis we consider the linearization of \eqref{e:travks} about a fixed periodic
traveling wave solution 
$\bar{u}=U(\cdot;\bar \beta)$, 
where we assume without loss of generality $\bar X=1$.
To this end, consider a nearby solution of \eqref{e:travks} of the form
\[
u(x,t)=\bar{u}(x)+v(x,t)
\]
where $v(\cdot,t)\in L^2(\RM)$ is small in some topology (to be defined precisely later), corresponding to a
small localized perturbation of the underlying solution $\bar{u}$.  Directly
substituting this into \eqref{e:travks} and neglecting quadratic order terms in $v$ leads
us to the linearized equation
\be\label{e:lin}
\partial_t v=Lv:= \partial_x\left((c-a) v\right) -\partial_x^4 v-\eps \partial_x^3v -\delta \partial_x^2 v,
\qquad a:= f'(\bar u),\ c:=\bar c.
\ee
Seeking separated solutions of the form $v(x,t)=e^{\lambda t}v(x)$, $v\in L^2(\RM)$ and $\lambda\in\C$ yields
the eigenvalue problem
\be\label{e:eig}
\lambda v=Lv
\ee
considered on the real Hilbert space $L^2(\RM)$.

As the coefficients of $L$ are $1$-periodic functions of $x$, Floquet theory implies that the
spectrum of the operator $L$, considered here as a densely defined operator on $L^2(\RM)$,
is purely continuous and, in particular, $\lambda\in\sigma(L)$ if and only if the spectral problem
\eqref{e:eig} has a bounded eigenfunction of the form
\be\label{e:ansatz}
v(x;\lambda,\xi)=e^{i\xi x}w(x;\lambda,\xi)
\ee
for some $\xi\in[-\pi,\pi)$ and 
$w\in L^2_{\rm per}([0,1])$. 
Following \cite{G1,S1,S2}, we find that substituting the ansatz \eqref{e:ansatz}
into \eqref{e:eig} leads one to consider the one-parameter family of Bloch operators $\{L_\xi\}_{\xi\in[-\pi,\pi)}$ acting on $L^2_{\rm per}([0,1])$
via
\be\label{e:Lxi}
\left(L_\xi w\right)(x):=e^{-i\xi x}L\Big[e^{i\xi \cdot}w(\cdot)\Big](x).
\ee
Since the Bloch operators have
compactly embedded domains $H^4_{\rm per}([0, 1])$ in $L^2_{\rm per}([0,1])$, their spectrum consists entirely of discrete eigenvalues
which, furthermore, depend continuously on the Bloch parameter $\xi$.  It follows then by standard considerations that
\[
\sigma_{L^2(\RM)}(L)=\bigcup_{\xi\in[-\pi,\pi)}\sigma_{L^2_{\rm per}([0,1])}\left(L_\xi\right);
\]
see \cite{G1} for details.  More precisely, the spectrum of $L$ can be characterized by the union of countably many surfaces $\lambda(\xi)$
such that $\lambda(\xi)\in\sigma(L_\xi)$ for each $\xi\in[-\pi,\pi)$.

Given a function $u\in L^2(\RM)$ we now recall its inverse Bloch-representation
\be\label{Bloch}
u(x)=
 \int_{-\pi}^{\pi}e^{i\xi\cdot x}\check u(\xi, x) d\xi,
\ee
where $\check u(\xi, x):=\sum_k e^{2\pi ikx}\hat u(\xi+ 2\pi k)$
are 1-periodic functions and $\hat u(\cdot)$
denotes the Fourier transform of $u$, defined here as
\[
\hat{u}(z)=\frac{1}{2\pi}\int_{\RM}e^{-i\omega z}u(\omega)d\omega.
\]
Denote by 
$\mathcal{B}:u\mapsto \check u$ 
the Bloch transform operator 
taking $u\in L^2(\RM)$
to its Bloch transform 
$\check u\in L^2([-\pi,\pi);L^2_{\rm per}([0,1]))$.\footnote{
Here, and elsewhere, we are adopting the notation
$\|f\|_{L^q([-\pi,\pi),L^p([0,1]))}:=\Big(\int_{-\pi}^\pi\|f(\xi,\cdot)\|_{L^p([0,1])}^qd\xi\Big)^{1/q}$.}
%
For a given linear operator $L$ with $1$-periodic coefficients, 
it can readily be verified that 
$\mathcal{B}\left(Lu\right)(\xi,x)=L_\xi\left[\check{u}(\xi,\cdot)\right](x)$ 
hence the associated Bloch operators $L_\xi$ may be viewed as 
operator-valued symbols under $\mathcal{B}$
acting on $L^2_{\rm per}([0,1])$.   Furthermore, the identity
$\mathcal{B}\left(e^{Lt}u\right)(\xi,x)=e^{L_{\xi}t}\left[\check{u}(\xi,\cdot)\right](x)$ naturally
yields the inverse representation formula
%
\be\label{IBFT}
e^{Lt}u_0= \int_{-\pi}^{\pi}
e^{i\xi \cdot x}e^{L_\xi t}\left[\check u_0(\xi, \cdot)\right](x)
d\xi\
\ee
relating behavior of the linearized equation \eqref{e:eig} to
that of the diagonal Bloch operators $L_\xi$.

The representation formula \eqref{IBFT} suggests
that bounds on the linear solution operator $e^{Lt}$ acting
on some function in $L^2(\RM)$ may be obtained from bounds on the associated operators $e^{L_\xi t}$ acting
on $L^2_{\rm per}([0,1])$.  To facilitate these estimates, we notice by
Parseval's identity that the Bloch transform
$u\mapsto \mathcal{B}(u)=\check u$ is an isometry in $L^2$, i.e.
\be\label{iso}
\|u\|^2_{L^2(\RM)}=2\pi\int_{-\pi}^\pi\int_0^1\left|\check{u}(\xi,x)\right|^2dx~d\xi=
2\pi\|\check u\|^2_{L^2([-\pi,\pi); L^2([0,1]))},
\ee
%
More generally, we have for any one-parameter family of $1$-periodic functions $f(\xi,\cdot)$, $\xi\in[-\pi,\pi)$,
the generalized Hausdorff-Young inequality
\be\label{hy}
\Big\|\int_{-\pi}^\pi e^{i\xi\cdot}f(\xi,\cdot)d\xi\Big\|_{L^p(\RM)}\leq (2\pi)^{\frac1p}\|f\|_{L^q([-\pi,\pi),L^p([0,1]))}
\ee
valid for any $q\leq 2\leq p$ and $\frac{1}{p}+\frac{1}{q}=1$;
this is readily obtained interpolating Parseval's identity with Triangle Inequality (case $q=1$, $p=\infty$), see \cite{JZ1}. 
It 
is
in the context of the above functional framework 
that we will obtain our crucial linearized estimates;
see Section \ref{s:linbds} below.

\subsubsection{Spectral stability assumptions}\label{s:stabassumptions}

Taking 
variations in a neighborhood of $\bar{u}$ along the $3$-dimensional periodic solution 
manifold
\eqref{e:manifold} in
directions for which the period does not change,
we find that
the generalized kernel of the Bloch operator $L_0$ is at least two-dimensional by
assumption (H2); 
see \cite{NR2} for more details.
Following
\cite{JZ1}, then, we assume along with (H1)--(H2) the following {\it spectral diffusive stability} conditions:
\begin{itemize}
\item[(D1)] $\sigma(L) \subset \left\{\ \lambda\in\CM \ \middle|\ \hbox{\rm Re}(\lambda) <0\ \right\}\cup\{0\}$.
\item[(D2)] $\sigma(L_\xi) \subset \left\{\ \lambda\in\CM \ \middle|\ \hbox{\rm Re}(\lambda)\leq -\theta |\xi|^2\ \right\}$, for some $\theta>0$ and any $\xi\in[-\pi,\pi)$.
\item[(\DDD)] $\lambda=0$ is an eigenvalue of $L_{0}$ of algebraic multiplicity two.
\end{itemize}
Assumptions (H1)-(H2) and
(D3)
imply that there exist two eigenvalues of the form
\be\label{e:surfaces}
\lambda_j(\xi)= -i a_j \xi +o(|\xi|)
\ee
of $L_\xi$ bifurcating from $\lambda=0$ at $\xi=0$;
see Lemma \ref{blochfacts} below.  Following \cite{JZ1}, we make the following non-degeneracy
hypothesis:
\begin{itemize}
\item[(H3)] The coefficients $ a_j$ in \eqref{e:surfaces} are distinct.
\end{itemize}
Hypothesis (H3) ensures the analyticity of the functions $\lambda_j(\cdot)$; again, see Lemma \ref{blochfacts} below for details.

Continuing, we notice that when $f(u)=\frac{u^2}{2}$
the Galilean invariance of \eqref{e:KS}, along with assumptions (H1)-(H2) and (D3), implies
that the zero-eigenspace of $L_0$ admits a non-trivial Jordan chain of height 
two \cite{NR2}. 
%
Indeed, in this case it is readily checked that the map
\[
\mathcal{G}_cu(x,t)=u(x-ct,t)+c
\]
sends solutions of 
\eqref{e:travks} to solutions of \eqref{e:KS} 
so that, in particular, variations in wave speed are 
period-preserving.  
It follows 
that the
periodic solution manifold \eqref{e:manifold} can be parameterized as
\be\label{e:manifold2}
\left\{\ (x,t)\mapsto U(x-\alpha-ct;c,b)\ \middle|\ (\alpha,c,\beta)\in\R^2\times\mathcal{I}\ \right\},
\quad
\hbox{\rm with $\mathcal{I}\subset \RR$}.
\ee
that is, wave speed serves as a non-degenerate local coordinate on \eqref{e:manifold}.
Noting that $(\d_t-(L-c\d_x))\mathcal{G}_c=\mathcal{G}_c(\d_t-L)$, we remark that variations along the periodic manifold \eqref{e:manifold2} correspond to solutions
of the (formal) linearized equation $(\partial_t-(L-c\d_x))v=0$ and find, setting $\bar{u}=U(\cdot;\bar c,\bar b)$, that
\[
\left(\partial_t-(L-\bar c\d_x)\right)\left(\d_c[\mathcal{G}_cU(\cdot;c,\bar b)]_{|c=\bar c}\right)=0.
\]
Since $\partial_c[\mathcal{G}_c\bar{u}(\cdot;c,\bar b)]_{|c=\bar c}=\mathcal{G}_{\bar c}(-t\bar{u}_x)+\mathcal{G}\d_cU_{|(c,b)=(\bar c,\bar b)}$ and $\d_cU_{|(c,b)=(\bar c,\bar b)}$ is periodic of period $1$, this verifies the existence of a Jordan chain of height two 
(recall, the height is at most by (D3)) ascending above the translation mode $\bar{u}_x$.

For general nonlinearities $f$, however, the existence of such a Jordan block does not immediately follow
from assumptions (H1)-(H2) and (D3).  Indeed, setting $\beta=(\beta_1,\beta_2)\in\RM^2$ in \eqref{e:manifold}, we find, 
taking variations along \eqref{e:manifold} in $\beta_j$ near 
$\bar{u}=U(\cdot;\bar \beta)$, 
that
\be\label{ubetaj}
L\d_{\beta_j}U|_{\beta=\bar\beta}=-\d_{\beta_j}c|_{\beta=\bar\beta}\bar{u}_x.
\ee
As the function 
$\d_{\beta_1}X(\bar\beta)\d_{\beta_2}U(\cdot;\bar\beta)-\d_{\beta_2}X(\bar\beta)\d_{\beta_1}U(\cdot;\bar\beta)$
is periodic of period  $X(\bar \beta)=1$\footnote{
Here, we notice by (D3) that $\bar\beta$ can not be a critical
point of the period $X$. Indeed, if $dX(\bar\beta)=0$ then $\d_{\beta_1}U(\cdot;\bar\beta)$ and $\d_{\beta_1}U(\cdot;\bar\beta)$
would both be 1-periodic, hence $\lambda=0$ would be an eigenvalue of $L_0$ of algebraic multiplicity three by \eqref{ubetaj}.},
we see that the existence of a non-trivial Jordan block requires the additional non-degeneracy hypothesis\footnote{For more general systems, this requirement
would be replaced with the condition that the gradients of the wave speed and period be linearly independent at 
$\bar\beta$.}
\be\label{e:jordancondition}
\det\left(\frac{\partial (c,X)}{\partial (\beta_1,\beta_2)}(\bar\beta)\right)\neq 0.
\ee
For a much more thorough and precise discussion on this topic in the context of general systems of conservation laws,
please see the 
forthcoming 
work \cite{JNRZ1}; see also Remark \ref{r:linphasecouple} and Lemma \ref{blochfacts} below.
Throughout this work, we assume 
that such a Jordan block exists by enforcing
our final technical assumption: 
\begin{itemize}
\item[(H4)] The eigenvalue $0$ of $L_0$ is nonsemisimple, i.e., $\dim \ker L_0=1$.\footnote{The degenerate case where \eqref{e:jordancondition}, hence (H4), fails can be treated as in \cite{JZ3}.}
\end{itemize}

\br\label{condrmks}
\textup{
The coefficients $a_j$ are the characteristics
of an associated Whitham averaged system
formally governing slowly modulated solutions; see
 \cite{NR2}.
Thus, 
as discussed in \cite{OZ4,JZ1,JZN},
(D1) implies weak hyperbolicity of the Whitham averaged system (reality of $a_j$),
sometimes used as a formal definition of modulational stability;
condition (H4) holds generically, and corresponds to the assumption
that speed $c$ is non-stationary along the manifold of nearby
stationary solutions, see Lemma \ref{blochfacts}; 
and condition (D2) corresponds to ``diffusivity'' of
the large-time ($\sim$ small frequency) behavior of the linearized system,
and holds generically given (H1)--(H4), (D1), and (D3), see also \cite{S1,S2}.
}
\er

\br\label{r:linphasecouple}
\textup{ 
As 
noted
in 
\cite{NR2,JNRZ1}, 
the conservative structure of \eqref{e:KS} implies that, assuming (H1) and (D3), one can take $\beta=(X,M)$ in \eqref{e:manifold},
where 
$M=X^{-1}\int_{0}^XU(s;\beta)ds$ denotes the 
mean 
%
mass of the associated wave. Indeed, following the computations of \cite{OZ3}
we find under these assumptions that for $|\lambda|\ll 1$
\[
D_0(\lambda)=\Gamma\det\left(\frac{\partial(k,M)}{\partial(\beta_1,\beta_2)}(\bar \beta)\right)\lambda^2
      +\mathcal{O}(|\lambda|^3),
\]
where $\Gamma\neq 0$ is a constant, 
$k=X^{-1}$ denotes wave number 
and $D_0(\cdot)$ represents the Evans function\footnote{Recall that
the algebraic multiplicity of the root of the Evans
function at $\lambda=0$ agrees with the algebraic multiplicity of the eigenvalue $\lambda=0$ of
the associated linear operator: see \cite{G1}.}
for the Bloch operator $L_0$, from which our claim follows by assumption (D3).  Using this parameterization of the manifold
\eqref{e:manifold}, condition \eqref{e:jordancondition} reduces to
\[
\d_M c\,(\bar\beta)\neq 0
\]
corresponding to \emph{linear phase coupling} in the language of \cite{JNRZ1}.  In particular, notice that in the case
$f(u)=\frac{u^2}{2}$ one can use Galilean invariance 
%
to show
\[
\d_M c\,(\bar\beta)=1,
\]
yielding another verification of (H4), assuming (H1)-(H2) and (D1)-(D3), in this case.
}
\er

\subsection{Main result}

With the spectral assumptions discussed in the previous section, we are now prepared to state the main theorem of
this paper.

\begin{theo}\label{main}
Let $\bar u$ be any steady 1-periodic solution of \eqref{e:travks} such that (H1)--(H4) and (D1)--(D3) hold. 
Then there exist constants $\varepsilon_0>0$ and $C>0$ such that for any $\tilde u_0$ with $\|\tilde{u}_0-\bar{u}\|_{L^1(\RM)\cap H^K(\RM)}\leq\varepsilon_0$, where $K$ is
as in assumption (H1), there exist $\tilde{u}$ a solution of \eqref{e:travks} satisfying $\tilde u(\cdot,0)=\tilde u_0$ 
and a function $\psi(\cdot,t)\in W^{K,\infty}(\RM)$ such that for all $t\geq 0$ and $2\leq p\leq\infty$
we have the estimates
\begin{equation}\label{eq:smallsest}
\begin{aligned}
\left\|\tilde{u}(\cdot+\psi(\cdot,t),t)-\bar{u}\right\|_{L^p(\RM)}
         &\leq C\left(1+t\right)^{-\frac{1}{2}\left(1-\frac{1}{p}\right)}\left\|\tilde{u}(\cdot,0)-\bar{u}(\cdot)\right\|_{L^1(\RM)\cap H^K(\RM)},\\
\left\|(\psi_t,\psi_x)(\cdot,t)\right\|_{L^p(\RM)}
         &\leq C\left(1+t\right)^{-\frac{1}{2}\left(1-\frac{1}{p}\right)}\left\|\tilde{u}(\cdot,0)-\bar{u}(\cdot)\right\|_{L^1(\RM)\cap H^K(\RM)},\\
\left\|\tilde{u}(\cdot+\psi(\cdot,t),t)-\bar{u}\right\|_{H^K(\RM)}
         &\leq C\left(1+t\right)^{-\frac{1}{4}}\left\|\tilde{u}(\cdot,0)-\bar{u}(\cdot)\right\|_{L^1(\RM)\cap H^K(\RM)},\\
\left\|(\psi_t,\psi_x)(\cdot,t)\right\|_{H^K(\RM)}
         &\leq C\left(1+t\right)^{-\frac{1}{4}}\left\|\tilde{u}(\cdot,0)-\bar{u}(\cdot)\right\|_{L^1(\RM)\cap H^K(\RM)}.
\end{aligned}
\end{equation}
Moreover, we have the $L^1(\RM)\cap H^K(\RM)\to L^\infty(\RM)$ nonlinear stability estimate
\begin{equation}\label{eq:stab}
\left\|\tilde{u}(\cdot,t)-\bar{u}\right\|_{L^\infty(\RM)},~\left\|\psi(\cdot,t)\right\|_{L^\infty(\RM)}\leq C\left\|\tilde{u}(\cdot,t)-\bar{u}\right\|_{L^1(\RM)\cap H^K(\RM)}
\end{equation}
valid for all $t\geq 0$.
\end{theo}

\noindent
Theorem \ref{main} asserts
asymptotic $L^1\cap H^K\to L^p$
convergence of the modulated solution $\tilde u(\cdot+\psi(\cdot,t),t)$ toward $\bar u$,
but only bounded $L^1\cap H^K\to L^\infty$ nonlinear stability of the unmodulated solution $\tilde{u}(\cdot)$ about $\bar u$. 

\br\label{nat:coordinates}
\textup{
It may seem more natural
to introduce $\psi$ via 
$
\displaystyle
v(x,t)=\tilde{u}(x,t)-\bar{u}(x+\psi(x,t))
$
However, in doing so one introduces in the equation for $v$ terms involving only $\psi$ and thus not 
decaying in time. For this reason we work instead with 
$
v(x, t) =\tilde{u}(x+\psi(x, t), t)-\bar{u}(x), 
$
that is, 
$
\tilde{u}(x, t) =\bar{u}(Y(x, t))+v(Y(x, t), t),
$ 
where $Y$ is such that 
$Y (x, t) +\psi(Y (x, t), t) = x$, $Y (y +\psi(y, t), t) = y$.
Notice that we insure the existence of such a map $Y$ by keeping, for any
 $t$, $\|\psi(\cdot,t)\|_{L^{\infty}(\R)}$ bounded and $\|\psi_x(\cdot,t)\|_{L^{\infty}(\R)}$. 
It should be stressed, however, that
$$
Y (x, t) = x-\psi(x, t)+\mathcal{O}\left( \|\psi(\cdot,t)\|_{L^{\infty}(\R)}\|\psi_x(\cdot,t)\|_{L^{\infty}(\R)}\right)
$$
so that we are not so far from the natural (but inappropriate) approach;
see \cite{JNRZ1} for a detailed discussion. 
%
Notice, moreover, 
that introducing the map Y above enables one to go back to the original unknown $\tilde{u}(x,t)$. 
}
\er

The proof of Theorem \ref{main} is presented in Section \ref{s:proof}.
As noted earlier,
the proof
of the corresponding theorem for periodic traveling wave solutions of the (non-conservative) Swift-Hohenberg
equation \eqref{e:sh} is considerably less complicated than that of 
Theorem \ref{main}; the reader unfamiliar
with the methods of this paper may 
thus
wish to consult Appendix \ref{a:sh} before proceeding to the
proof of Theorem \ref{main}.

While it is generally accepted that the structural hypotheses (H3)-(H4) and (D3) should  generically hold, it
is unclear 
a priori whether there exist periodic traveling wave solutions 
of \eqref{e:gen} 
that
satisfy the spectral diffusive
stability assumptions (D1)-(D2).  While an analytical verification of these conditions seems beyond the scope
of our current machinery, due to the complexity of the governing equations, we provide a numerical study of this topic in Section \ref{s:num} below giving convincing evidence of existence of stable bands of periodic waves.
We view the
rigorous numerical or analytical proof of spectral stability in some interesting
regime as the primary remaining open problem in this 
area 
of investigation.
See in particular \cite{BaN} and \cite{JNRZ4} 
for progress toward a numerical proof in the KdV limit 
$|(\gamma,\delta)|\to 0$.
%

\section{A numerical study: spectral stability and time evolution}\label{s:num}

In this section, we address 
numerically
the issue of spectral stability of periodic traveling wave solutions 
of \eqref{e:KS}. 
Our numerical study is based on complementary tools, namely,
Hill's method and Evans function computations; 
see \cite{BZ,BJNRZ1}. On the one hand, we use SpectrUW numerical software \cite{CDKK}, based on Hill's method, which is a Galerkin method of approximation. More precisely, the coefficients of the linear operators $L_{\xi}$ (defined by \eqref{e:Lxi}) together with periodic eigenvectors are expanded by using Fourier series and then a frequency cutoff is used to reduce the problem to finding eigenvalues of a finite matrix. 
It is known that  Hill's method converges ``globally'' very fast, 
but with rates not particularly guaranteed in 
fixed areas; 
this method is used here to obtain a good localization of the spectrum. 
In order to test  rigorously the spectral stability, we use an Evans function approach:
after bounding the region where unstable spectra might exist, we compute a winding number along various contours
to localize the spectrum.
In some sense, Hill's method helps to find suitable contours to test the spectral stability. 
To analyze stability near the origin, we also compute a Taylor expansion of the critical eigenvalues bifurcating from
the origin at the $\xi=0$ state.  The method of moments (to be described in section \ref{meth:mom}) then guarantees the 
stability of these critical eigenvalues in the neighborhood of $(\lambda,\xi)=(0,0)$.
While most of the results of this section are numerical, we point out that the numerics are well-conditioned with
rigorous error estimates; see 
\cite{BJZ,Z1,BHZ2} 
for more details.

Throughout this section, both for definiteness and in order to compare with 
some previous numerical 
results (see e.g. 
\cite{BaN,FST,CDK}), we consider \eqref{e:KS} 
unless otherwise specified
with the specific nonlinearity $f(u)=\Lambda\frac{u^2}{2}, \Lambda>0$,  i.e. the equation
\be\label{e:cks1}
\partial_t u+\delta \partial_x^2 u+\eps \partial_x^3 u+\gamma\partial_x^4 u+\Lambda u\partial_x u=0
\ee
with various values of $\eps, \delta, \gamma,\Lambda$. If one introduces the characteristic amplitude $U$, length scale $L$ and time scale $T=L/\Lambda U$, equation \eqref{e:cks1} is alternatively written as
\be\label{e:cksad}
\partial_{\tilde t} \tilde{u}+\tilde{u}\partial_{\tilde x} \tilde{u}+\frac{\eps}{\Lambda UL^2}\partial_{\tilde x}^3 \tilde{u}+\frac{\delta}{\Lambda UL}\partial_{\tilde x}^2 \tilde u+\frac{\gamma}{\Lambda UL^3}\partial_{\tilde x}^4 \tilde u=0.
\ee
As a result, one can always assume $\Lambda=1$ in \eqref{e:cks1}. Depending on  applications and 
%
numerical purposes, one can  remove two of the three other parameters in the study of \eqref{e:cks1}. We have chosen here two 
reduced 
parameterizations of the problem. On the one hand, 
we first choose some $\bar \eps$ then, setting $\bar\eps=\eps/\Lambda UL^2$ 
and $\gamma=\Lambda UL^3$ into \eqref{e:cksad}, one finds that it is equivalent to choose $\gamma=1$ and $\eps$ is a fixed parameter and let $\delta$ vary as a free parameter in \eqref{e:cks1}. On the other hand, letting $ L^2=\gamma/\delta$ and $U^2=(\eps/\Lambda L^2)^2+(\delta/\Lambda L)^2$, one finds that it is equivalent to 
choose\footnote{
Notice that, applying the transformation $(x,u)\mapsto(-x,-u)$ if necessary, we may always take $\eps\geq 0$ in \eqref{e:cks1}.
Thus, here  we set $\eps=\sqrt{1-\delta^2}$ or $\delta=\sqrt{1-\eps^2}$.}
%
$\gamma=\delta$ and $\delta^2+\eps^2=1$ in \eqref{e:cks1}. This latter case is particularly of interest since it 
is found in thin film 
theory\footnote{Where the equation is designed to investigate formation of 
patterns of size 1 at a threshold of instability and therefore is scaled 
as $\gamma=\delta$ so that, for constant states, the most linearly-unstable perturbation has frequency 1.}
 and 
since
{\it all} equations \eqref{e:cks1} can be reduced to this particular form through rescaling; that is we cover {\it all} the values of 
parameters 
$\eps,\delta,\gamma$. However this form of the equation may lead to a singularly perturbed problem as $\delta\to 0$; 
that is why we have also focused on the first form of \eqref{e:cks1} with $\gamma=1$
even if, to cover all cases, we need then to introduce two families, one with $\eps$ fixed to $0$ and the other one with $\eps$ fixed to some arbitrary positive value. 
%

The plan of the study is as follows.
As a first (testing) step, we consider the spectral stability case $\gamma=1$ and $\eps$ fixed. 
In this test case, we give all the details of our numerical approach:  we first compute 
a priori bounds on possible unstable eigenvalues and then describe our numerical 
Galerkin/Evans functions based approach. 
In order to compare our results with the existing literature, specifically the results in \cite{FST},
we first carry out a numerical study in the case of the classical Kuramoto-Sivashinsky equation ($\eps=0$, $\gamma=1$).
We follow this with a numerical study for
the generalized KS equation when $\eps=0.2$ and $\gamma=1$.  
Next, we switch to the ``thin film scaling'' $\gamma=\delta$, $\delta^2+\eps^2=1$ and show numerically that there are (sometimes several) 
bands of stable  periodic traveling waves for all choices of the model parameters, comparing to previous numerical studies of \cite{BaN} and \cite{CDK}

We end this section with a discussion of time-evolution studies nearby various periodic traveling wave profiles 
numerically computed and compare with the dynamics predicted by the associated Whitham averaged system.

\subsection{Spectral stability analysis: $\gamma=1$, $\eps$ fixed}

\subsubsection{Continuation of periodic traveling waves}

Let us consider periodic traveling wave solutions of \eqref{e:cks1} in the simple case $\gamma=1$ and $\eps$ is fixed. After integrating once, traveling waves of the form $u(x,t)=u(x-ct)$ are seen to satisfy the first order nonlinear system of ODE's
\be\label{e:ckssystem}
\left(
  \begin{array}{c}
    u \\
    u' \\
    u''\\
  \end{array}
\right)'
   =\left(
      \begin{array}{c}
        u' \\
        u'' \\
        cu-\eps u''-\delta u' -\frac{u^2}{2}+q \\
      \end{array}
    \right)
\ee
for some constant of integration $q\in\RM$.  Noting that we can always take $c=0$ by Galilean invariance,  we find
two fixed points in the associated three-dimensional phase space given by
\[
U_-(q)=\left(
         \begin{array}{c}
           -\sqrt{2q}\ \\
           0 \\
           0 \\
         \end{array}
       \right)\quad\textrm{and}\quad
U_+(q)=\left(
         \begin{array}{c}
           \sqrt{2q}\ \\
           0 \\
           0 \\
         \end{array}
       \right).
\]
Following the computations of Section \ref{a:hopf}, we find that the fixed point $U_+(q)$ undergoes a Hopf
bifurcation when $\eps\delta_{\rm Hopf}(q)=\sqrt{2q}$; see Figure \ref{f:orbits}(a) when $\eps=0.2$.  Moreover, near this bifurcation point we find a 
two dimensional manifold of small amplitude periodic orbits of \eqref{e:ckssystem} parameterized by period and the integration
constant $q$.
Observe that due to the presence of the destabilizing second order and stabilizing fourth order
terms in \eqref{e:gen}, an easy calculation shows that the equilibrium solution $U_+(q)=\sqrt{2q}$, and in fact all constant solutions,
are 
linearly 
unstable to perturbations of the form $v(x,t)=e^{\lambda t+ikx}$, $\lambda\in\CM$, $k\in\RM$, with $|k|\ll 1$, while 
linearly 
stable to such perturbations when $|k|\gg 1$.
As a result we find that all small amplitude periodic orbits of \eqref{e:ckssystem} emerging from the Hopf bifurcation correspond to spectrally unstable periodic
traveling waves of \eqref{e:cks1}.

Not surprisingly, we are able to numerically continue this 
two parameter family of small amplitude periodic orbits
to obtain periodic 
orbits of \eqref{e:cks1} with large amplitude.
Furthermore, in the natural three-dimensional phase space of \eqref{e:ckssystem} and for $\eps$ fixed, we are able
to find a one-parameter family of periodic orbits of \eqref{e:ckssystem} with fixed period $X$ by continuing the periodic
orbits emerging from the Hopf point $U_+(q)$ with respect to the integration constant $q$; 
see Figure \ref{f:orbits}(b) for $\eps=0.2$ and $X=6.3$.  Notice that in order to keep the period $X$ fixed, 
one naturally has to vary the modeling parameter $\delta$ 
with respect to $q$.  

What is perhaps surprising, however, 
even though by now well known \cite{FST,CDK},
is that there 
appears numerically to be a small band of spectrally stable periodic traveling waves within this 
one-parameter family, outside of which all waves are
spectrally unstable; 
see Figure \ref{f:specplots} and Figure \ref{f:specplots-zoom}.
Our immediate goal is to substantiate this claim with careful numerical procedures 
completed 
with rigorous error bounds.  
Although, due to machine error, the justifications we discuss do not 
constitute a proof, given the rigorous error bounds involved, 
our claims could in principle be
translated into a numerical proof with the help of interval arithmetic. 
This would be an 
interesting direction for further investigation.

\br\label{fixedX}
\textup{
As well explored in the literature (see for example \cite{CDK}),
there occur a number of period doubling bifurcations as parameters
are varied, which might in principle be difficult to follow by
continuation.  However, this difficulty is easily avoided by rescaling
so as to fix the period, while letting other parameters vary instead.
This is our reason 
for computing with fixed period as we do.
}
\er

\begin{figure}[htbp]
\begin{center}
(a)\includegraphics[scale=.35]{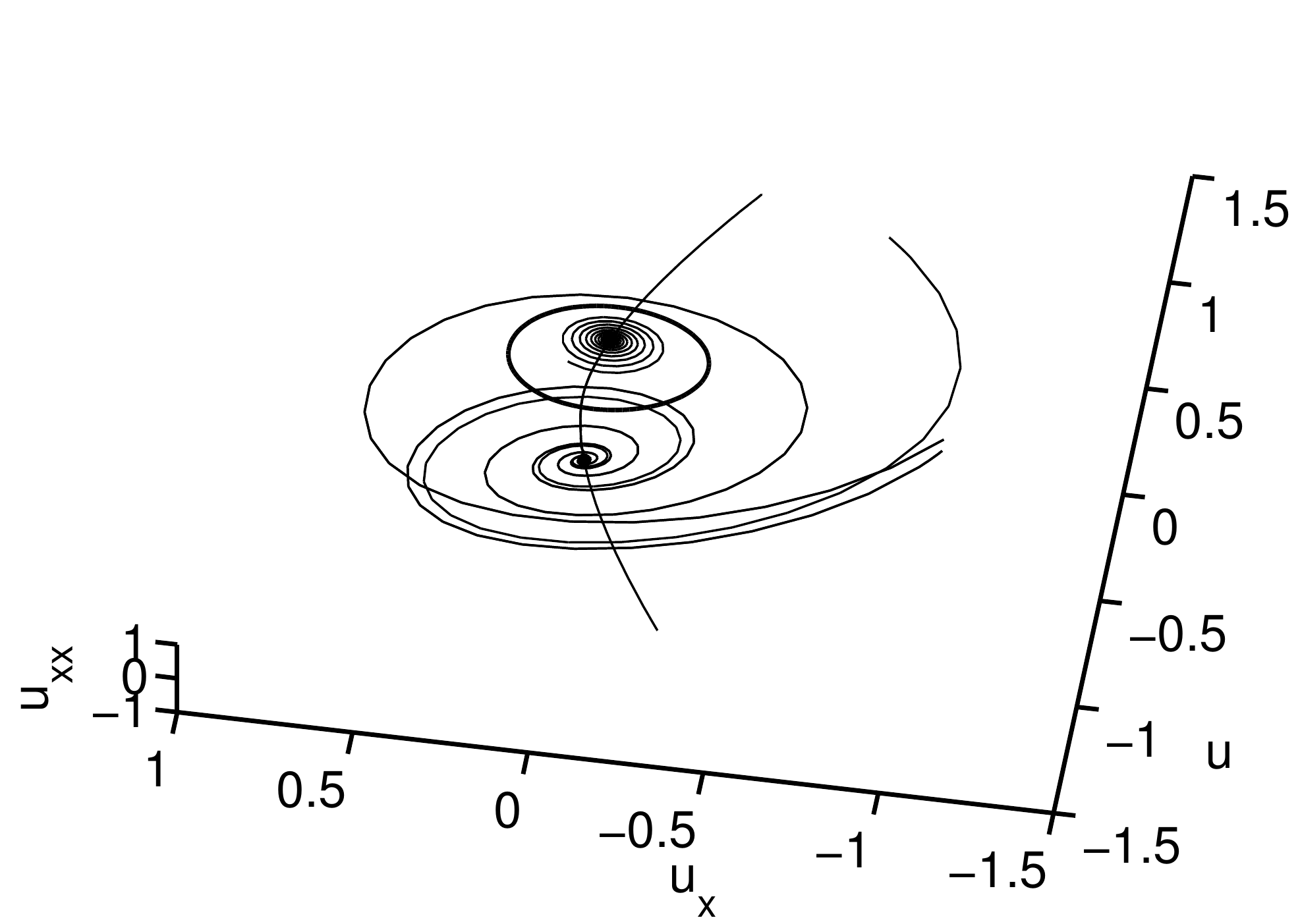}\quad (b)\includegraphics[scale=0.35]{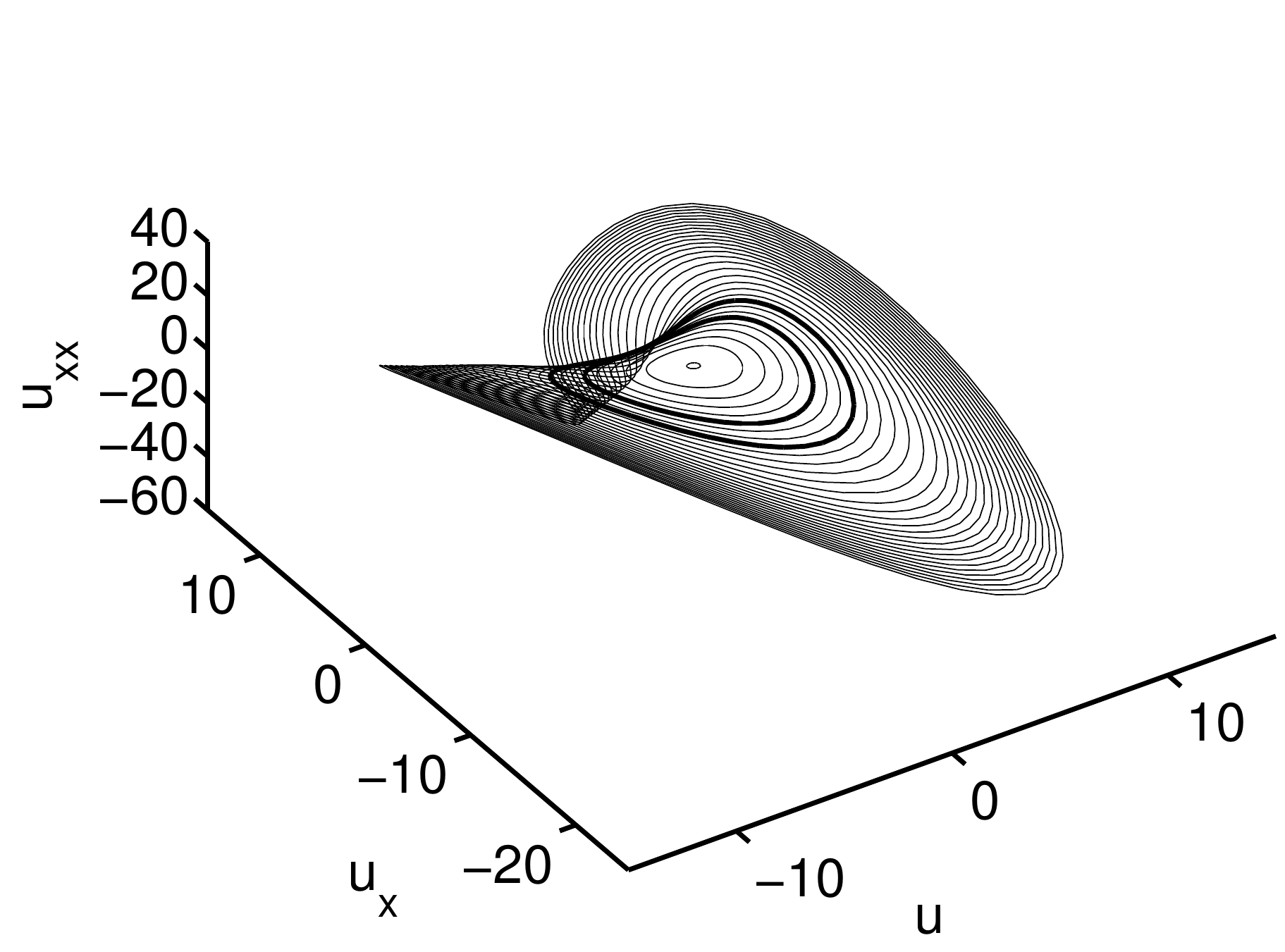}
\caption{(a) Here, we plot a few of the trajectories in the phase portrait of the \eqref{e:ckssystem} for $\eps=0.2$,
$q=0.04$, and $\delta\approx 1$.   The periodic trajectory bifurcating from the nearby unstable equilibrium solution $U_+(q)$
is plotted in bold, and has period $X=6.3$.  (b) This figure depicts the one-parameter family (up to translations) of periodic orbits
of period $X=6.3$ and wave speed $c=0$ emerging from the unstable equilibrium solution $U_+(q)$.  Here, $\eps=0.2$ and $q$ (and accordingly $\delta=\delta(q)$) varies from 1 (corresponding
to the orbits near the Hopf bifurcation) to 30 by unit steps. The bold periodic orbits 
near the center
correspond to the lower and upper stability boundaries: in particular,
all periodic traveling wave solutions of \eqref{e:cks1} corresponding to periodic orbits of \eqref{e:ckssystem} between these bold orbits
are spectrally stable 
with respect
to small localized perturbations.
}\label{f:orbits}
\end{center}
\end{figure}


\begin{figure}[htbp]
\begin{center}
(a)\includegraphics[scale=.15]{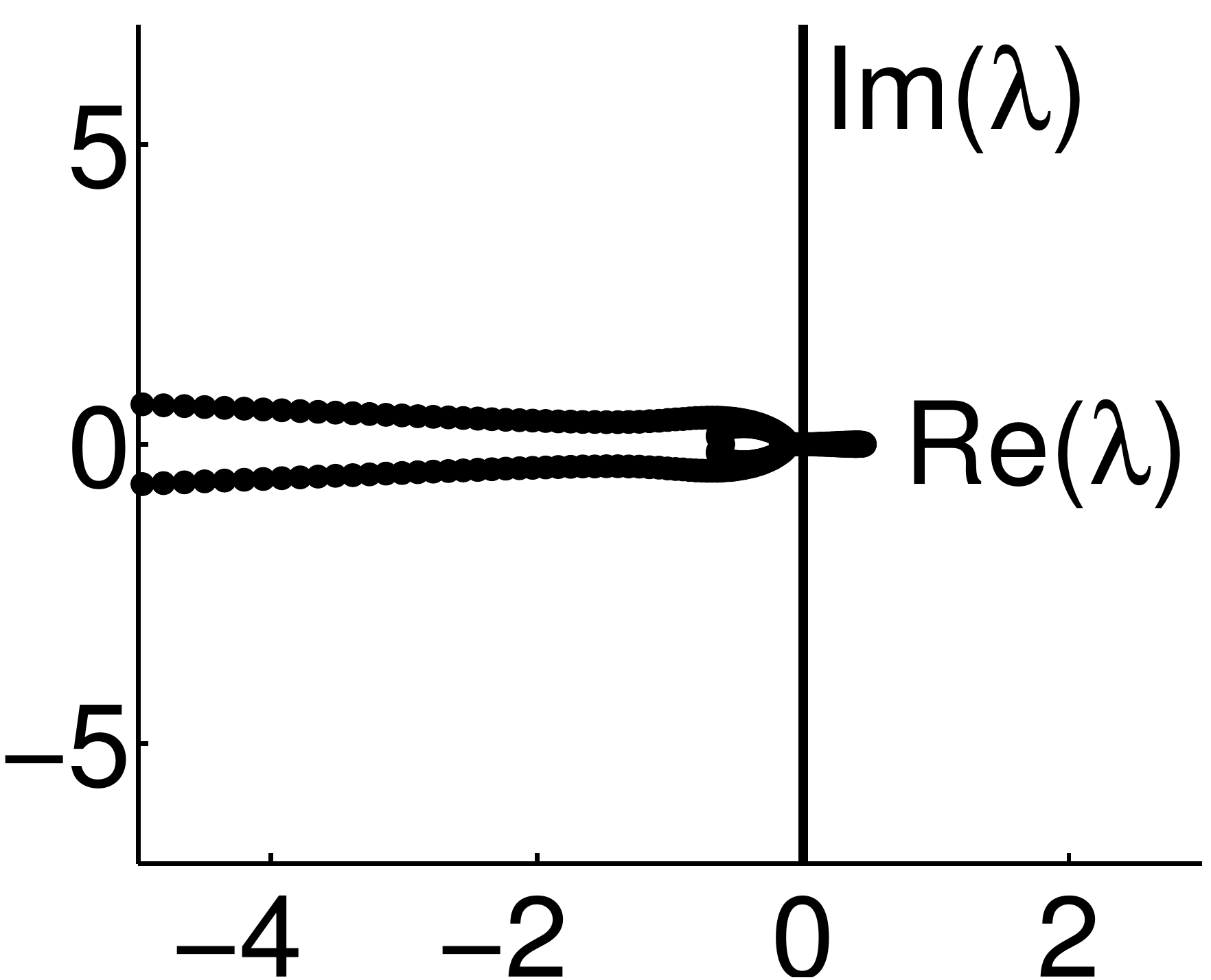}\quad (b)\includegraphics[scale=0.15]{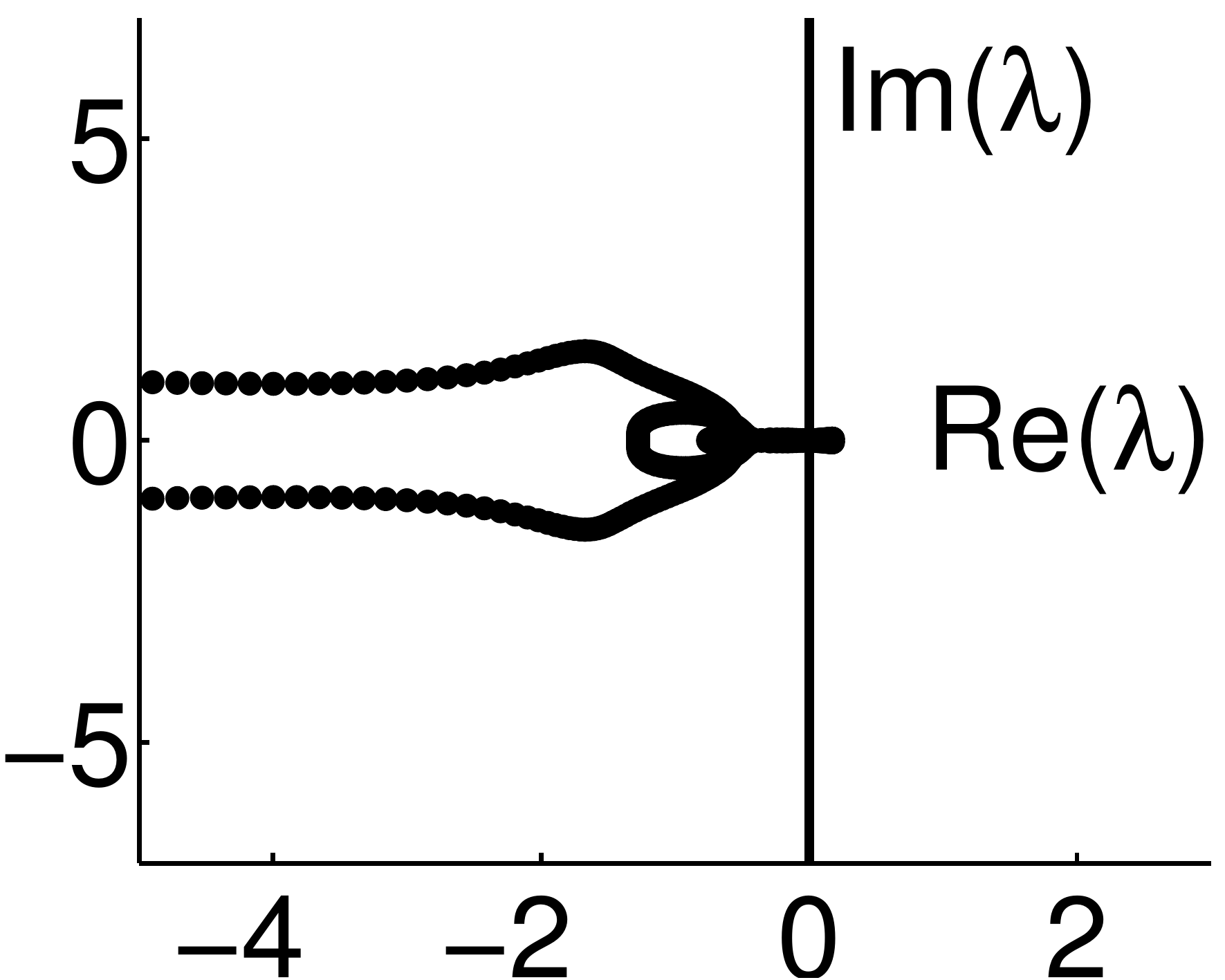} (c)\includegraphics[scale=0.15]{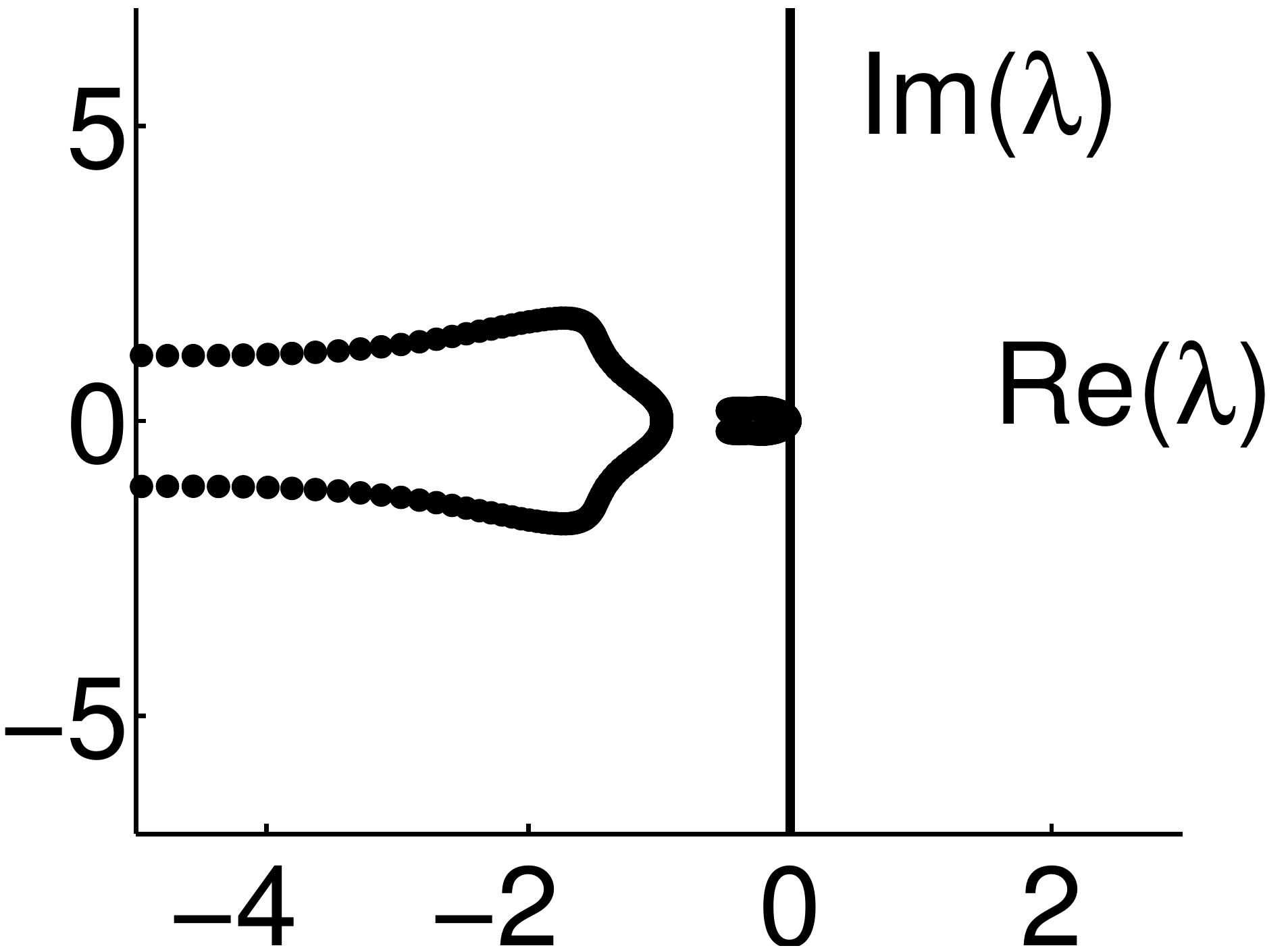}\\
(d)\includegraphics[scale=0.15]{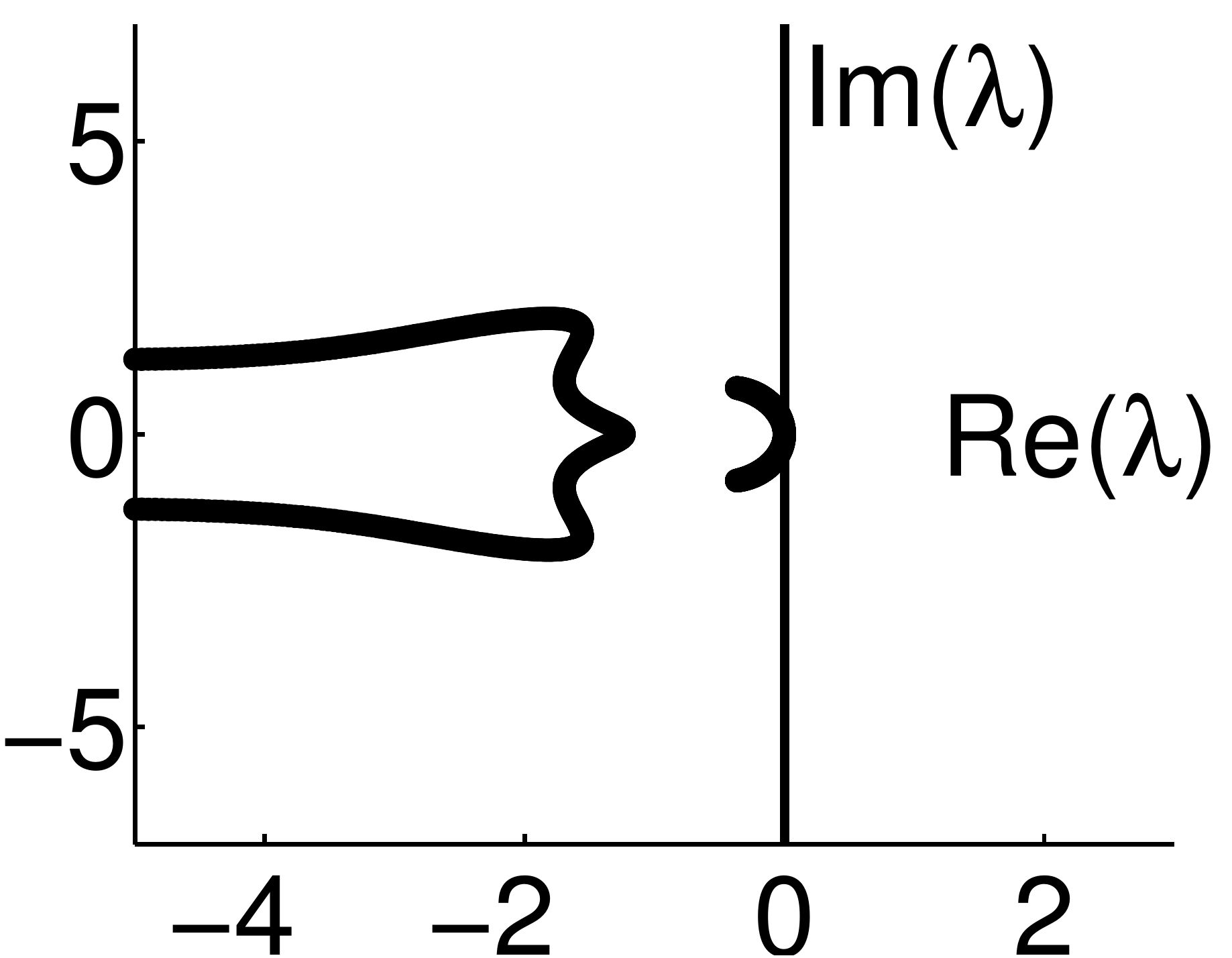}\quad (e)\includegraphics[scale=0.15]{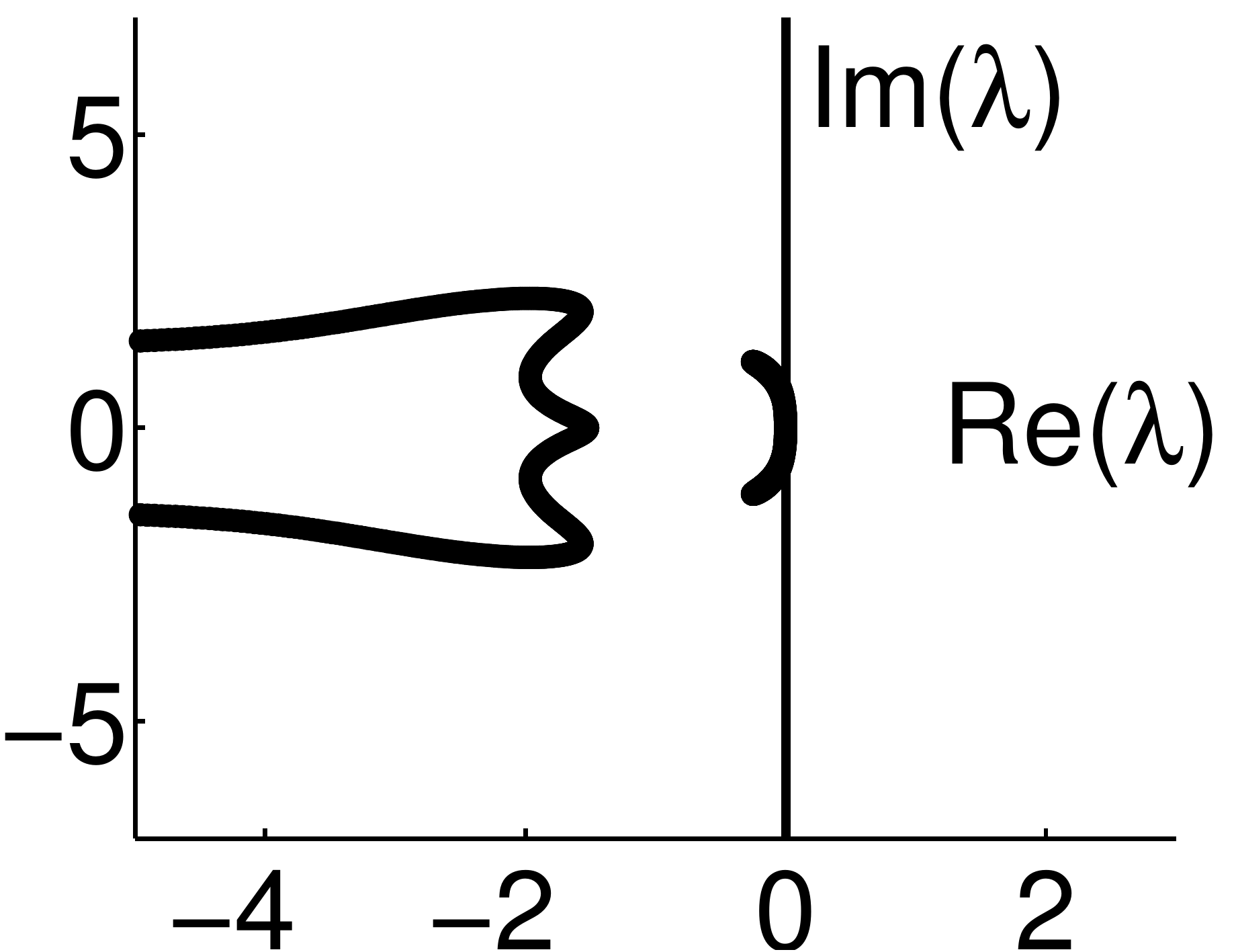}\quad (f)\includegraphics[scale=0.15]{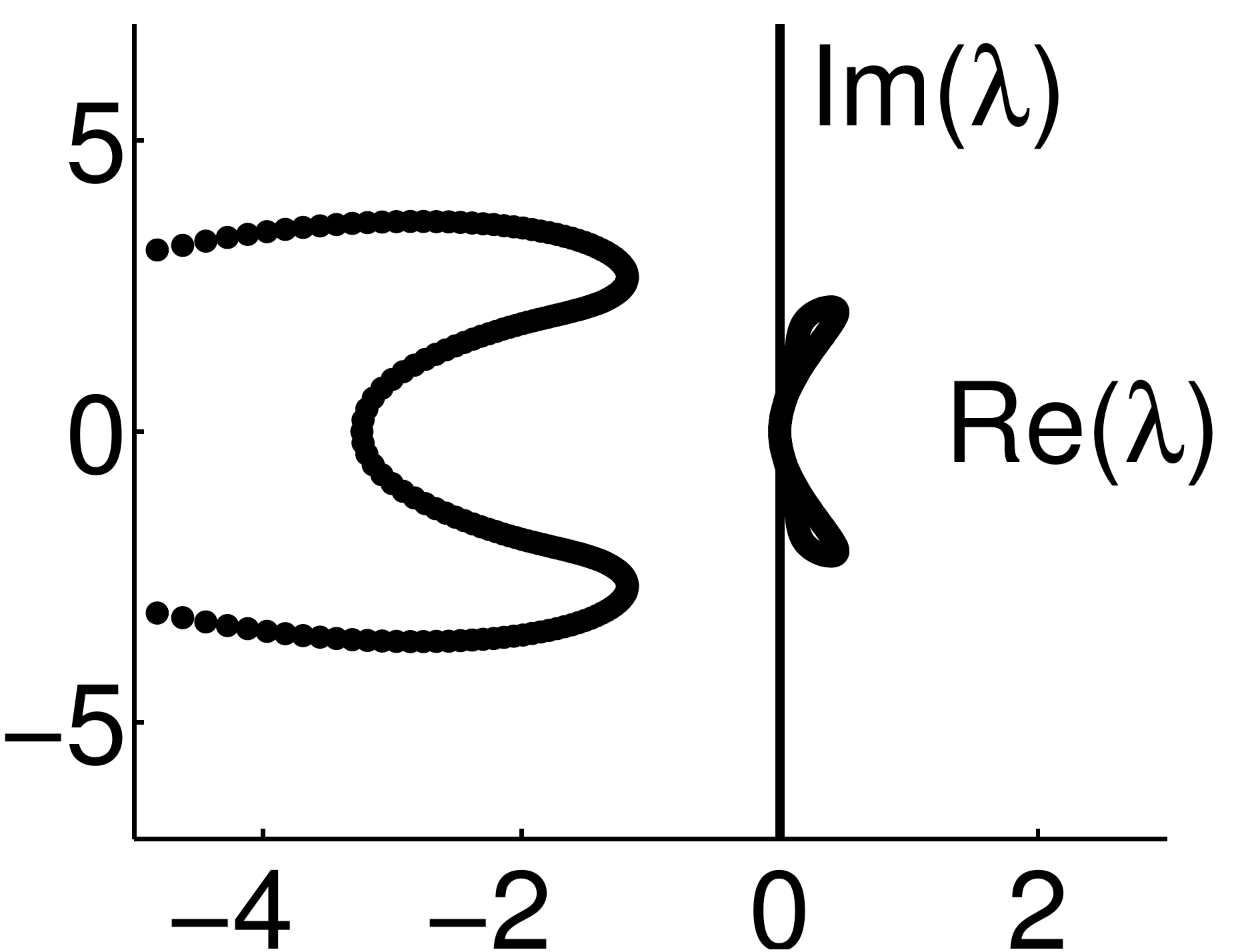}\\
(g)\includegraphics[scale=0.15]{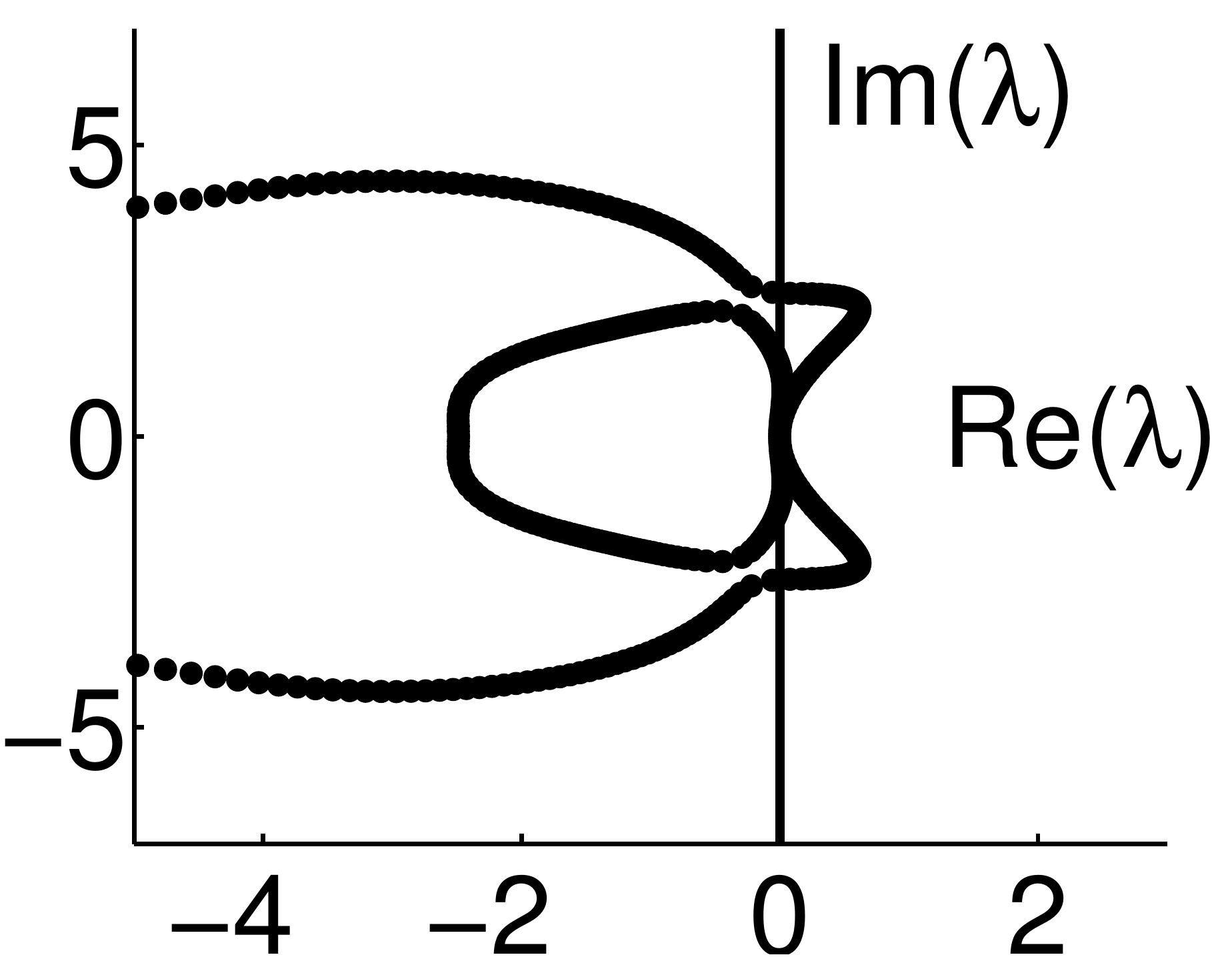}\quad (h)\includegraphics[scale=0.15]{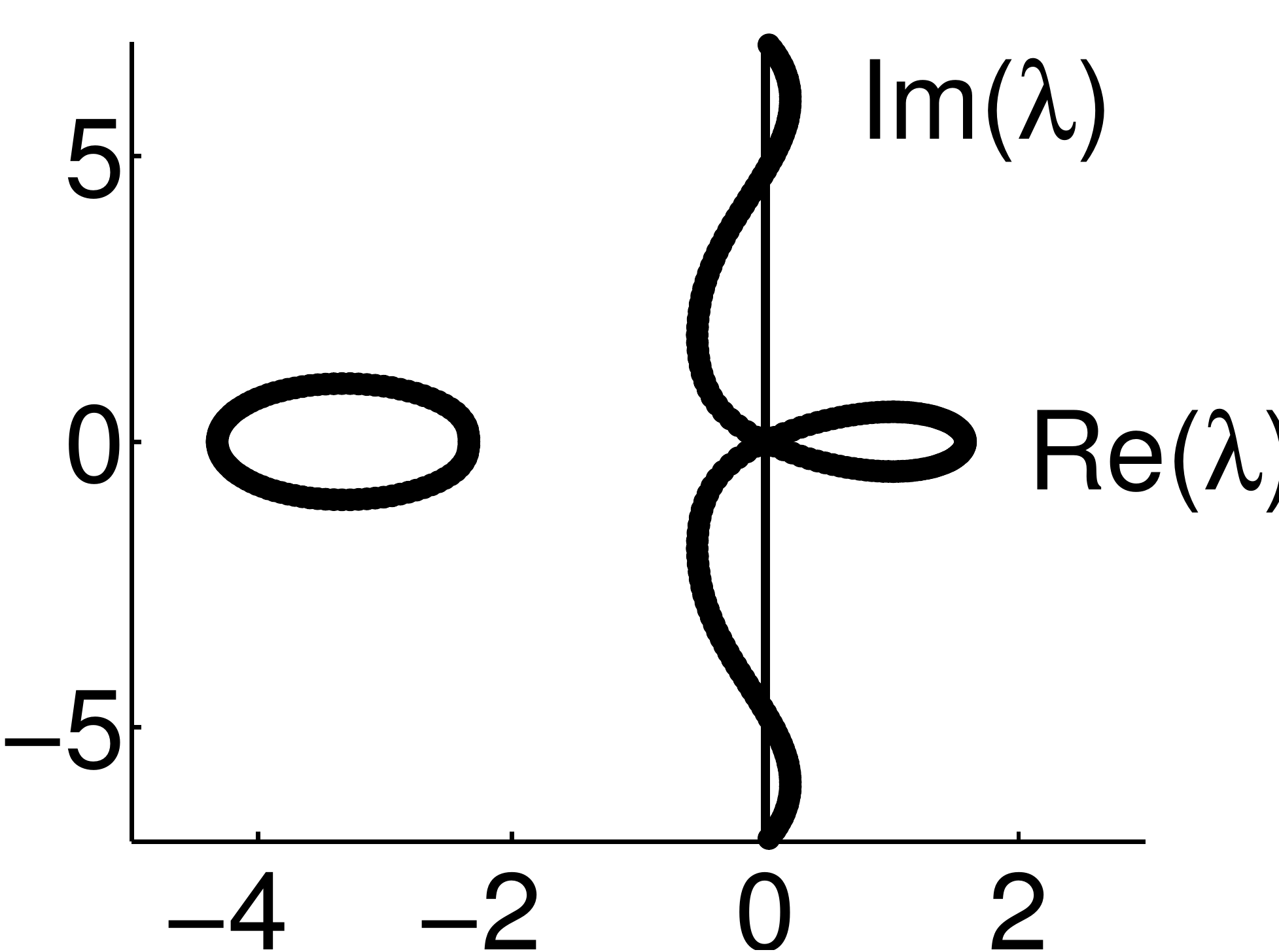}\quad(i)\includegraphics[scale=0.15]{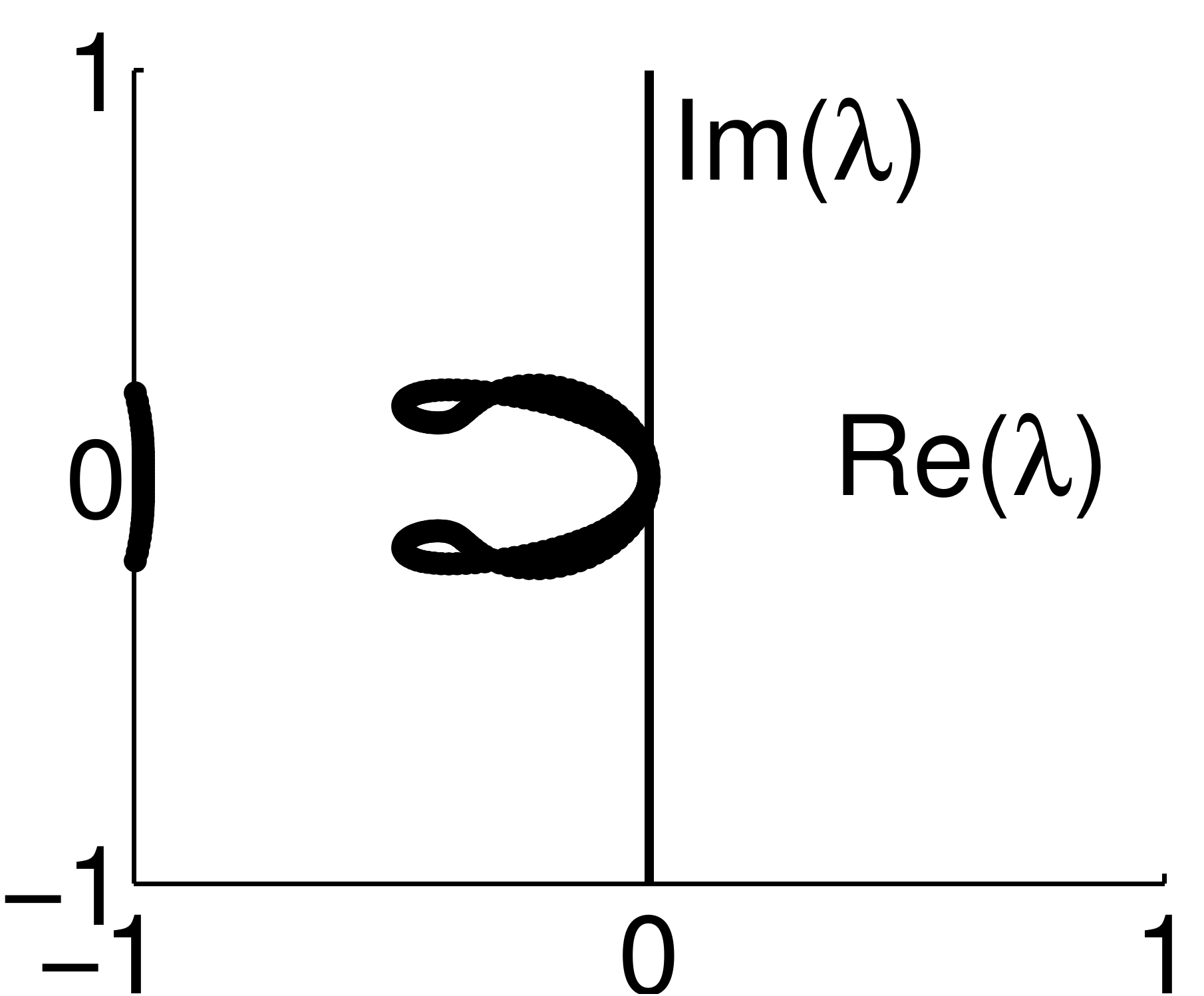}
\caption{A numerical sampling of the continuous spectrum of the linearization of \eqref{e:cks1} when $\eps=0.2$, plotted here as $\Re(\lambda)$ vs. $\Im(\lambda)$, about the periodic traveling wave profiles
on the periodic manifold depicted in Figure \ref{f:orbits}(b) for (a) $q=1$, (b) $q=4$, (c) $q=5$, (d) $q=6$, (e) $q=7$, (f) $q=12$, (g) $q=13.8$,
and (h) $q=30$.  Moreover in (i) we zoom in near the origin in the spectral plot given in (c).  All pictures here were generated with the Galerkin based
SpectrUW package developed at the University of Washington \cite{CDKK}.
}\label{f:specplots}
\end{center}
\end{figure}

\begin{figure}[htbp]
\begin{center}
(a)\includegraphics[scale=.2]{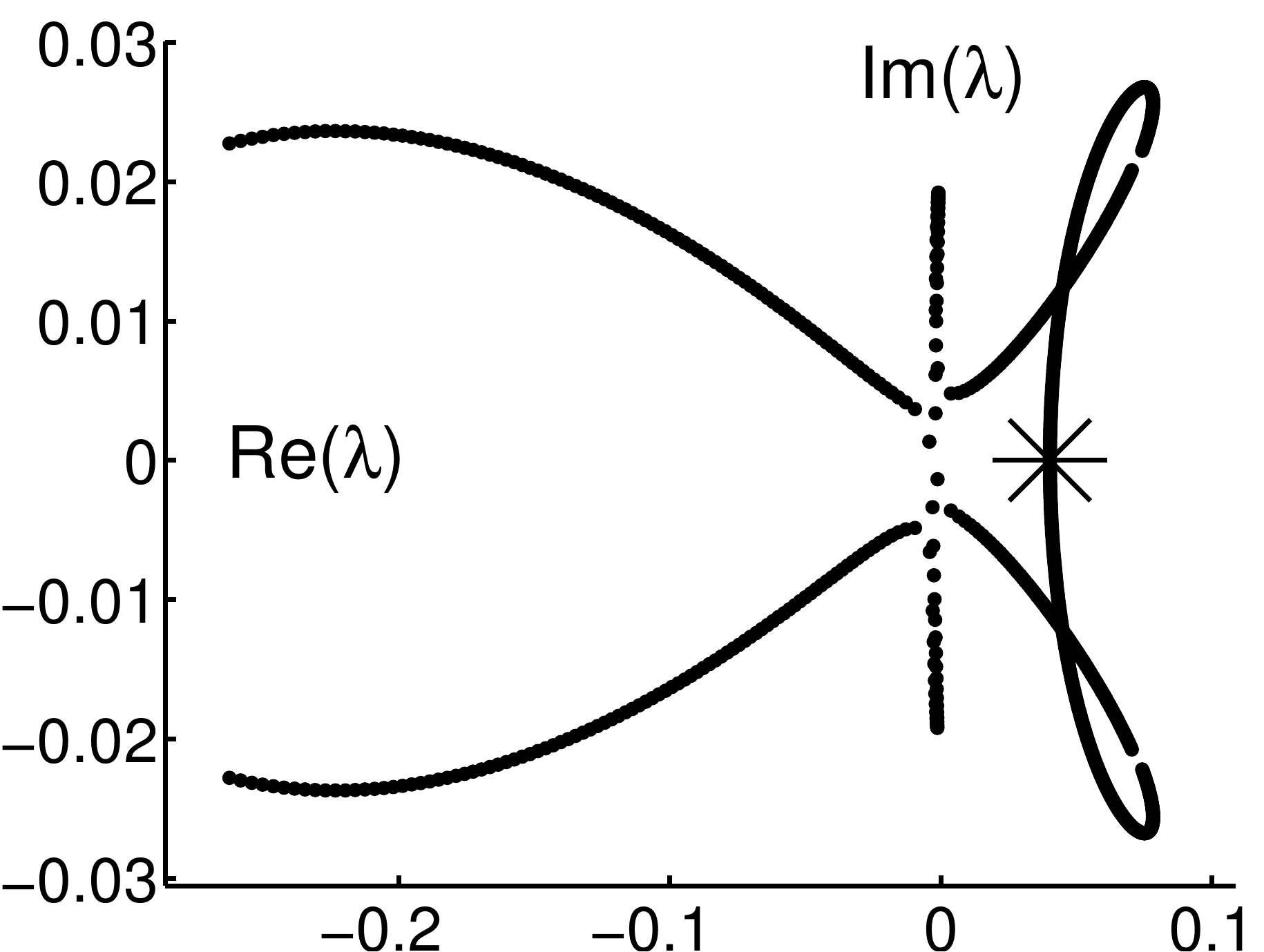}\quad (b)\includegraphics[scale=0.2]{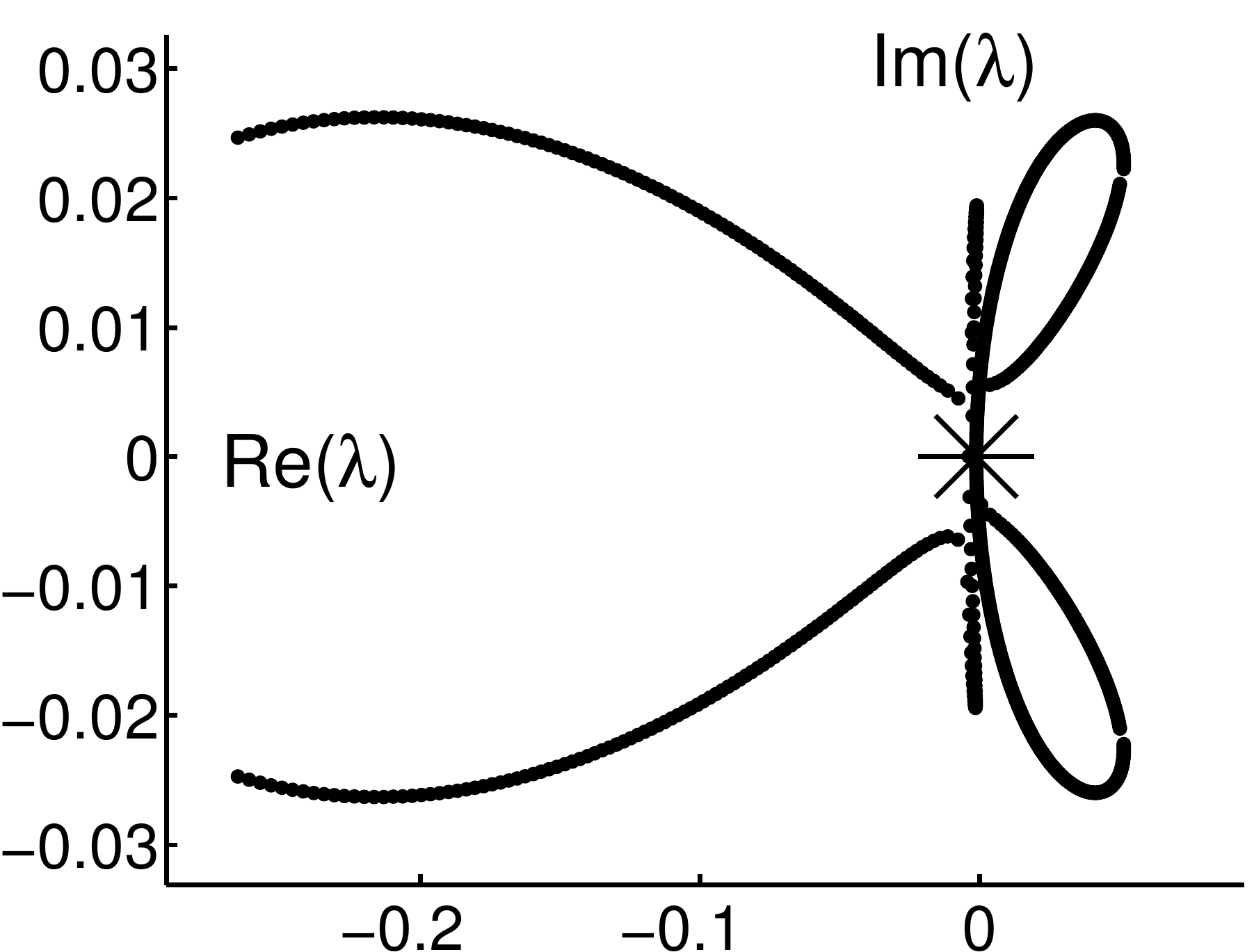} (c)\includegraphics[scale=0.23]{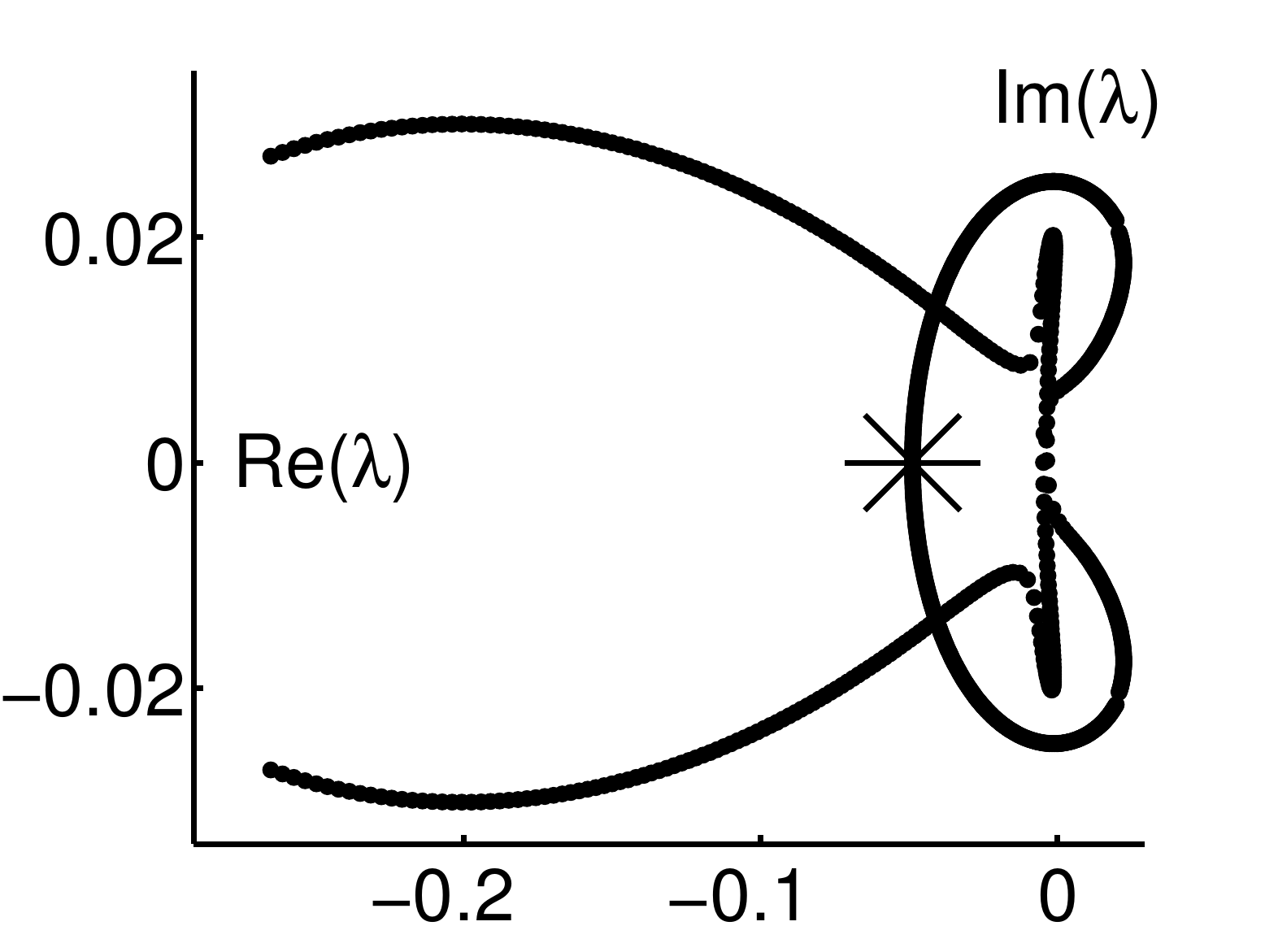}\\
\caption{Here, we zoom in near the origin in the $\lambda$-plane near, but just below, the lower stability boundary
when considering 6.3-periodic stationary solutions of \eqref{e:cks1} in the case $\eps=0.2$.  The three figures correspond
to (a) $q=4.4$, (b) $q=4.5$, and (c) $q=4.6$.  We find that prior to stability, a real anti-periodic eigenvalue
(corresponding to $\xi=-\pi$, and marked with a large star) crosses  through the origin and then the rest of the unstable spectrum
crosses the imaginary axis away from the origin. This crossing of a real anti-periodic eigenvalue corresponds to a period doubling bifurcation; see \cite{KE} for more details on period-multiplying bifurcations for KS. The spectral curves were generated by computing the roots of the Evans function using the method of moments described in Section \ref{meth:mom}.  In particular,
the transition to stability is not signaled by an eigenvalue crossing through the origin.
}\label{f:specplots-zoom}
\end{center}
\end{figure}


\subsubsection{High-frequency analysis}
We begin our study of the spectrum of the linear operator $L$, obtained by linearizing \eqref{e:cks1} about a fixed periodic traveling wave solution, by eliminating the possibility that the associated Bloch operators $L_\xi$ admit arbitrarily large unstable eigenvalues.
As the resulting estimates are independent of the Bloch frequency $\xi$, this allows us to reduce our search
for unstable spectrum of $L$ into a fixed compact domain in $\lambda$.   For generality, here we consider \eqref{e:KS} with a general
(smooth, say) nonlinearity $f$ rather than the quadratic nonlinearity leading to the formulation in \eqref{e:cks1} and $\gamma=1$.

Letting $\bar{u}$ be a fixed $\bar{X}$-periodic solution of \eqref{e:gen}, for each $\xi\in[-\pi/\bar X,\pi/\bar X)$ we must study the eigenvalue
problem
\[
\lambda u+\left(\partial_x+i\xi\right)\left(f'(\bar{u})u\right)+\eps\left(\partial_x+i\xi\right)^3 u+\delta\left(\partial_x+i\xi\right)^2 u+\left(\partial_x+i\xi\right)^4u=0
\]
considered on $L^2_{\rm per}([0,\bar X])$  We restrict ourselves to rational Bloch frequencies,
that is to cases where $\xi\bar X/(2\pi)\in \QM$, 
%
(recall that the spectra of $L_{\xi}$ is continuous in $\xi$)
and thus this is equivalent with studying the family of eigenvalue problems
\begin{equation}\label{sp_pb}
\lambda u+\left(f'(\bar{u})u\right)'+\eps u'''+\delta u''+u''''=0
\end{equation}
considered on $L^2_{\rm per}([0,n\bar X])$ for each $n\in\NM$.  In this 
latter
formulation, we have the following estimates on the modulus of possible unstable eigenvalues.

\bl \label{l:hfbds}
Let $n\in\NM$ be arbitrary, and suppose that there exists a function $u\in L^2_{\rm per}([0,n\bar X])$ which satisfies \eqref{sp_pb} for some $\lambda\in\CM$. 
%
Then $\lambda$ must satisfy the estimates
\begin{equation}\label{kRe_lambda}
\Re(\lambda)\leq \frac{1}{2}\|f''(\bar{u})\bar{u}'\|_{L^{\infty}([0,\bar X])}+\frac{\delta^2}{4},
\end{equation}
and
\begin{equation}\label{kMod_lambda}
\Re \lambda +|\Im \lambda|\le\frac{1}{2}\left(\|f''(\bar{u})\bar u'\|_{L^{\infty}([0,\bar X])}+\|f'(\bar{u})\|^2_{L^{\infty}([0,\bar X])}+\delta^2+\frac{(1+2\eps^2)^2}{2}\right).
\end{equation}
In particular, the above estimates are independent of $n$.
\el

\begin{proof}

By taking the real part of the complex scalar product of (\ref{sp_pb}) with ${u}$,
one obtains
\be\label{temp}
\Re(\lambda)\|u\|^2_{L_{per}^2([0,n\bar X])}=-\frac{1}{2}\int_0^{n\bar X}(f'(\bar u))'|u|^2(x)dx
+ \delta\|u'\|^2_{L^2_{per}([0, n\bar X])}-\|u''\|^2_{L^2_{per}([0, n\bar X])}.
\ee
Then, use  the Sobolev interpolation bound
$$
\|u'\|^2_{L^2_{per}([0, n\bar X])}\le
\frac{1}{2C}\ \|u''\|^2_{L^2_{per}([0, n\bar X])} +
\frac{C}{2}\ \|u\|^2_{L^2_{per}([0, n\bar X])},\quad C>0,
$$
with $C=\frac{\delta}{2}$ to deduce \eqref{kRe_lambda}.
Next, by taking the imaginary part of the complex scalar product of (\ref{sp_pb}) with ${u}$,
one obtains
$$
\Im(\lambda)\|u\|^2_{L_{per}^2([0,n\bar X])}=
 \Im\langle u',f'(\bar u)u\rangle_{L^2_{per}([0, n\bar X])}
+\eps\Im \langle u',u''\rangle_{L^2_{per}([0, n\bar X])},
$$
%
then
%
$$
\begin{aligned}
|\Im(\lambda)|\|u\|^2_{L_{per}^2([0,n\bar X])}&\le
 \|f'(\bar u)\|_{L^\infty([0,\bar X])}
\|u\|_{L^2_{per}([0, n\bar X])} \|u'\|_{L^2_{per}([0, n\bar X])} \\
&
\quad+\eps\|u'\|_{L^2_{per}([0, n\bar X])} \|u''\|_{L^2_{per}([0, n\bar X])} .
\end{aligned}
$$
By using Young's inequality, one finds
$$
\displaystyle
|\Im(\lambda)|\|u\|^2_{L_{per}^2([0,n\bar X])}\le
 \frac{1}{2}\|f'(\bar u)\|^2_{L^\infty([0,\bar X])}
\|u\|^2_{L^2_{per}([0, n\bar X])}+\frac{1+\eps D}{2} \|u'\|^2_{L^2_{per}([0, n\bar X])}+\frac{\eps}{2D} \|u''\|^2_{L^2_{per}([0, n\bar X])}
$$
valid for any $D>0$.  By using again the above Sobolev interpolation bound, one finds
\begin{equation}\label{ineq1}
\begin{aligned}
\displaystyle
|\Im(\lambda)|\|u\|^2_{L_{per}^2([0,n\bar X])}&\le
\left( \frac{1}{2}\|f'(\bar u)\|^2_{L^\infty([0,\bar X])}+\frac{C(1+\eps D)}{4}\right)\|u\|^2_{L^2_{per}([0, n\bar X])}\\
&\quad\quad        +\left(\frac{1+\eps D}{4C}+\frac{\eps}{2D}\right) \|u''\|^2_{L^2_{per}([0, n\bar X])} ,
\end{aligned}
\end{equation}
where now $C,D>0$ are arbitrary constants. From 
\eqref{temp}, 
one also deduces that
\be\label{ineq2}
\displaystyle
\Re(\lambda)\|u\|^2_{L_{per}^2([0,n\bar X])}\leq \frac{1}{2}(\|f''(\bar{u})\bar u'\|_{L^{\infty}([0,\bar X])}+\delta^2)\|u\|_{L^2([0, n\bar X])}-\frac{1}{2}\|u''\|^2_{L^2([0, n\bar X])}.
\ee
Choosing now 
$D=2\eps$, $C=1+D\eps$ 
and adding (\ref{ineq1}) and (\ref{ineq2}) yields
\[
\displaystyle
\Re(\lambda)+|\Im(\lambda)|\le \frac{1}{2}\Big(\|f''(\bar{u})\bar u'\|_{L^{\infty}([0,\bar X])}+\|f'(\bar{u})\|^2_{L^{\infty}([0,\bar X])}+\delta^2+\frac{(1+2\eps^2)^2}{2}\Big),
\]
as claimed.
\end{proof}

\br
\textup{
Notice that by choosing different constants $C$ and $D$ above, along with a different choice of the
Young's scaling in \eqref{ineq2}, various other bounds of the form \eqref{kMod_lambda} can be obtained which could be beneficial
in certain parameter regimes.  Specifically, given any constants $C_1,C_2,C_3,\alpha,\beta,A>0$ such that
\[
\frac{\delta}{2C_1}-1+A\left(\frac{C_2}{4\alpha}\|f'(\bar{u})\|_{L^\infty([0,\bar X])}+\frac{\eps\beta C_3}{4}+\frac{\eps}{2\beta}\right)=0
\]
we obtain the inequality
\[
\Re(\lambda)+A|\Im(\lambda)|\leq
\frac{1}{2}\|f''(\bar{u})\bar u'\|_{L^\infty([0,\bar X])}+\frac{\delta C_1}{2}
+\frac{A}{2}\left(\|f'(\bar{u})\|_{L^\infty([0,\bar X])}\left(\alpha+\frac{1}{2\alpha C_2}\right)+\frac{\eps\beta}{2C_3}\right).
\]
The inequality \eqref{kMod_lambda} corresponds to the choices
\[
C_1=\delta,\quad C_2=C_3=\frac{1}{1+2\eps^2},\quad \alpha=\|f'(\bar{u})\|_{L^\infty([0,\bar X])},\quad\beta=2\eps,\quad A=1,
\]
while an example of an alternative bound, which may be useful for large $\eps>0$, is given by
\begin{equation}\label{kMod_lambdaeps}
\Re(\lambda)+\frac{1}{2\eps^2}|\Im(\lambda)|\leq
\frac{1}{2}\left(\|f''(\bar{u})\bar u'\|_{L^\infty([0,\bar X])}+\delta^2
+\frac{1}{\eps^2}\left(\frac14\|f'(\bar{u})\|^2_{L^\infty([0,\bar X])}+\frac{4}{3\eps^2}\right)\right)
\end{equation}
corresponding to the choices
\[
C_1=\delta,\quad C_2=C_3=\frac{3\eps^2}{4},\quad\alpha=\frac12\|f'(\bar{u})\|_{L^\infty([0,\bar X])},\quad\beta=\frac{2}{\eps},\quad A=\frac{1}{2\eps^2}.
\]
}
\er

By Lemma \ref{l:hfbds}, for  fixed modeling parameters $\eps$ and $\delta$ and a periodic profile $\bar{u}$ of \eqref{e:gen},
there exists a real number $R_0=R_0(\eps,\delta,\bar{u})$ such that the unstable 
(and marginally stable) 
spectrum of the associated linearized operator $L$
satisfies
\be\label{compact}
\sigma\left(L\right)\cap\left\{\ \lambda\in\CM\ \middle|\ \Re(\lambda)\geq0\ \right\}\ 
\subset\ B(0,R_0).
\ee
%
This reduces to a compact set the region in the spectral plane where we need to search for unstable spectra of $L$.
In order to verify the spectral stability hypotheses 
(D1)--(D3) 
and (H3) we 
use
the periodic Evans function
and basic analytic function theory, the necessary details of which we review in the next section.

\subsubsection{Low- and mid-frequency analysis: Evans function and numerical methods}
\label{sec:lowfreqanal}

We now begin our search for unstable spectra of $L$ in the ball $B(0,R_0)$, where $R_0$ is given by the high-frequency
bounds in the previous section.  To this end, we 
use
a complex-analytic function, known as the periodic Evans function 
\cite{G1}, 
that
is well-suited to the task at hand.
First, let $\bar{u}$ be a fixed $\bar X$-periodic traveling wave solution of \eqref{e:cks1}, 
$\bar X$ arbitrary,
and notice that the associated spectral problem \eqref{e:eig} can be 
rewritten as a first order system of the form
\be\label{e:system}
{\bf Y}'(x;\lambda)=\mathbb{H}(x,\lambda){\bf Y}(x;\lambda), 
\quad {\bf Y}\in\C^4
\ee
and 
that
$\lambda\in\sigma(L)$ if and only if \eqref{e:system} admits a non-trivial solution satisfying
\[
{\bf Y}(x+\bar X;\lambda)=e^{i\xi \bar X}{\bf Y}(x;\lambda),\quad\forall x\in\RM
\]
for some $\xi\in[-\pi/\bar X,\pi/\bar X)$.\footnote{
Here, we have allowed the domain of $\xi$ to depend on the period $\bar X$.
This choice is to ensure that, when linearizing about a constant state, 
the dispersion relation $D(\lambda,\xi)=0$ agrees with
that obtained by the standard Fourier transform, where $D$ denotes the 
periodic Evans function defined below in \eqref{e:evans}.
This differs from the convention of Gardner \cite{G1} and some others in which
the Evans function is parametrized instead by the Floquet number
$\bar \zeta:=e^{i\bar X \xi}\in S^1$, a quantity with domain
independent of $\bar X$.}
Letting $\Psi(x,\lambda)$ be a matrix solution of \eqref{e:system} with initial condition $\Psi(0,\lambda)={\bf I}$
for all $\lambda\in\CM$, we follow Gardner \cite{G1} and define the periodic Evans function for our problem to be
\be\label{e:evans}
D(\lambda,\xi):=\det\left(\Psi(\bar X,\lambda)-e^{i\xi \bar X}{\bf I}\right),\quad(\lambda,\xi)\in\CM^2,
\ee
where 
for later convenience
we allow for possibly complex Bloch frequencies $\xi$ in the above definition.
By construction, then, $\lambda\in\sigma(L)$ if and only
if there exists a $\xi\in[-\pi/\bar X,\pi/\bar X)\subset \RM$ 
such that $D(\lambda,\xi)=0$.  More precisely, for a fixed $\xi\in[-\pi/\bar X,\pi/\bar X)$ the roots
of the function $D(\cdot,\xi)$ agree in location and algebraic multiplicity with the eigenvalues of the Bloch operator $L_\xi$
\cite{G1}.

Since $\mathbb{H}(\cdot,\lambda)$ clearly depends analytically on $\lambda$, it follows that the Evans function is a complex-analytic
function of both $\lambda$ and $\xi$, hence our search for unstable spectra of $L$ may by reformulated as a problem
in analytic function theory.  Indeed, for a fixed $\xi$ the number of eigenvalues of $L_\xi$ within the bounded
component of a closed contour $\Gamma$, along which $D(\cdot,\xi)|_{\Gamma}$ is non-vanishing, can be calculated
via the winding number
\[
n(\xi;\Gamma):=\frac{1}{2\pi i}\oint_{\Gamma}\frac{\partial_{\lambda}D(\lambda,\xi)}{D(\lambda,\xi)}~d\lambda.
\]
Given a fixed contour $\Gamma$ then, notice as $n(\cdot;\Gamma)$ is a continuous function with integer values, it is locally constant
and hence can only change values when a root of $D(\cdot,\xi)$ crosses through $\Gamma$.  With these preparations in mind, we now outline
the general scheme of our numerical
procedure to verify the spectral stability hypotheses
(D1)--(D3) 
and (H3). As noted below, for numerical efficiency, the following
steps are sometime carried out in a slightly different order.  Of course then, there are several consistency checks
that follow to justify various assumptions made.  For similar reasons, it is at times advantageous
to execute a particular step in a slightly different way than suggested here; more will be said 
about
%
this below.  Nevertheless, the
general idea of the numerical procedure is contained in the following steps, to be completed for fixed.
\begin{itemize}
\item[{\bf\underline{Step 0}:}] First, numerically determine set 
$\mathcal{D}:=\{\,q\,|\,\d_\lambda^2D(0,0)=0\}$ 
by looking for sign changes\footnote{Since the map 
$q\mapsto \d_\lambda^2D(0,0)$ 
is real valued, this is sufficient to find the zeros.}  of 
$\d_\lambda^2D(\lambda,0)$ 
for $\lambda>0$, where here $D(\lambda,0)$ denotes
the Evans function for $\xi=0$ obtained by linearizing about the periodic profile 
$\bar u$ with parameter $q$.
\end{itemize}
Then away from this degenerate
parameter set $\mathcal{D}$, we may complete the following steps.
\begin{itemize}
\item[{\bf\underline{Step 1}:}] 
First, determine using the bounds of
Lemma \ref{l:hfbds} an $R_0>0$ satisfying \eqref{compact}.
\item[{\bf\underline{Step 2}:}] Let $0<r_0<R_0$, $\Omega_0:=B(0,R_0)\setminus B(0,r_0)$, and set
\be\label{e:contours}
\begin{aligned}
\Gamma_0&:=\partial\left(\Omega_0\cap\{\lambda\in\CM\,|\Re(\lambda)\geq 0\}\right),\\
\Gamma_1&:=\partial\left(B(0,R_0)\cap\{\lambda\in\CM\,|\Re(\lambda)\geq 0\}\right);
\end{aligned}
\ee
see Figure \ref{f:gammacurves}.  Now, verify that for some $0<k_0\ll 1$ the following two conditions hold:
\begin{itemize}
\item[(a)] $n(\xi,\Gamma_0)=0$ for all $\xi\in[-\pi/\bar X,\pi/\bar X)$.
\item[(b)] $n(\xi,\Gamma_1)=0$ for $k_0\leq|\xi|\leq\pi/\bar X$;
\end{itemize}
see Figure \ref{f:contours} below.
\item[{\bf\underline{Step 3}:}] Verify that for some $r_0<r_1<R_0$, $n(\xi,\partial B(0,r_1))=2$ for $|\xi|<k_0$.
\item[{\bf\underline{Step 4}:}] Finally, Taylor expand the two critical modes for $|\xi|\ll 1$ as
\be\label{e:lexp}
\lambda_j(\xi)=\alpha_j\xi+\beta_j\xi^2+\mathcal{O}(|\xi|^3),\quad j=1,2
\ee
and numerically verify that $\alpha_j\in\RM i$, the $\alpha_j$ are distinct, and $\Re(\beta_j)<0$.  Notice that such an expansion exists by Lemma \ref{blochfacts}.
\end{itemize}

\noindent 
If the criteria of each of the above steps are satisfied, then we
conclude (numerically) that the wave is {\it stable}.  We emphasize here the simplicity and generality of the above numerical
protocol, hence being directly applicable to many model problems.  In particular, the above numerical
method does not require an often lengthy and problem specific spectral perturbation expansion near the neutral
stability mode $(\lambda,\xi)=(0,0)$.
%

\begin{figure}[htbp]
\begin{center}
(a)\includegraphics[scale=0.3]{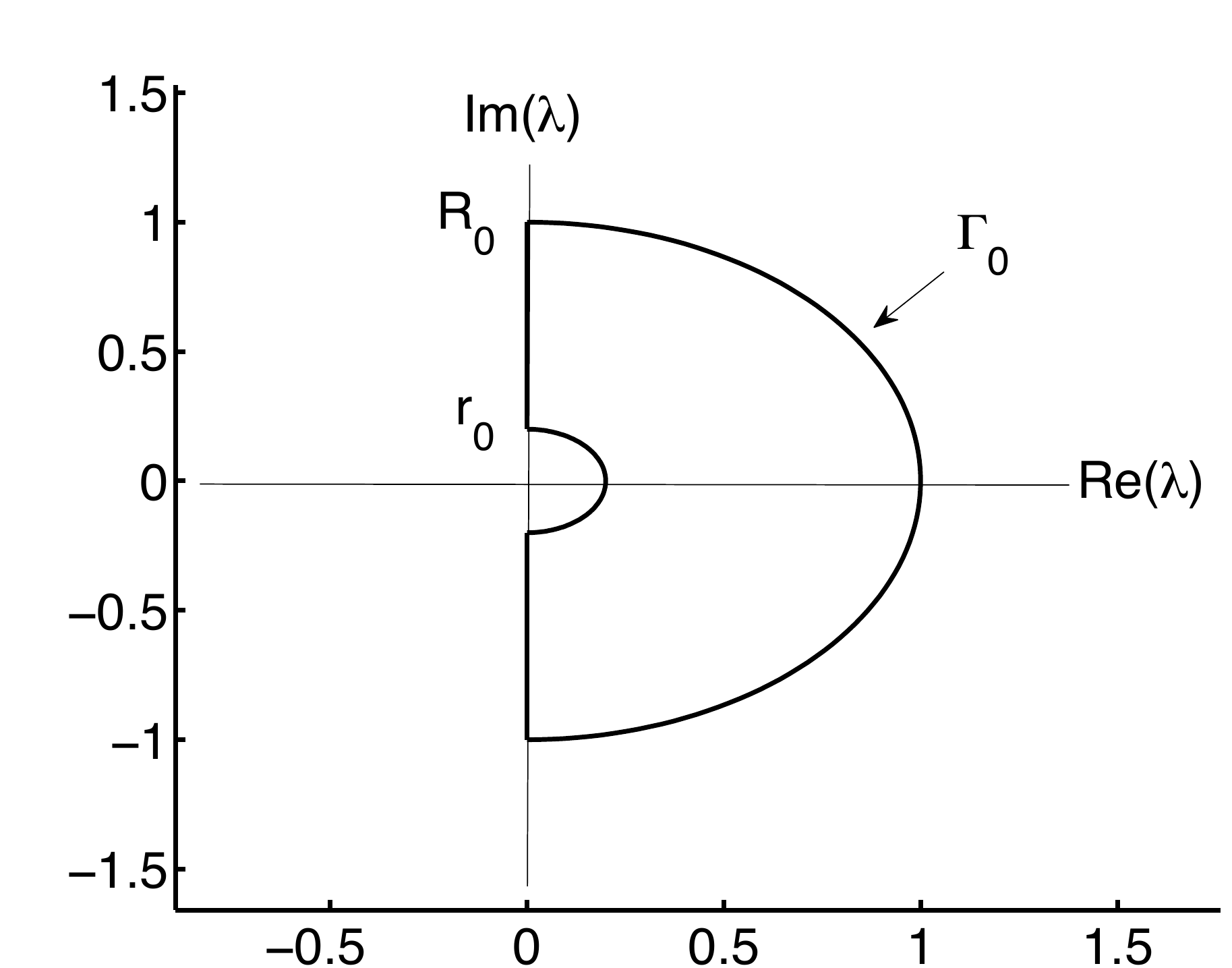}\quad\quad(b)\includegraphics[scale=0.3]{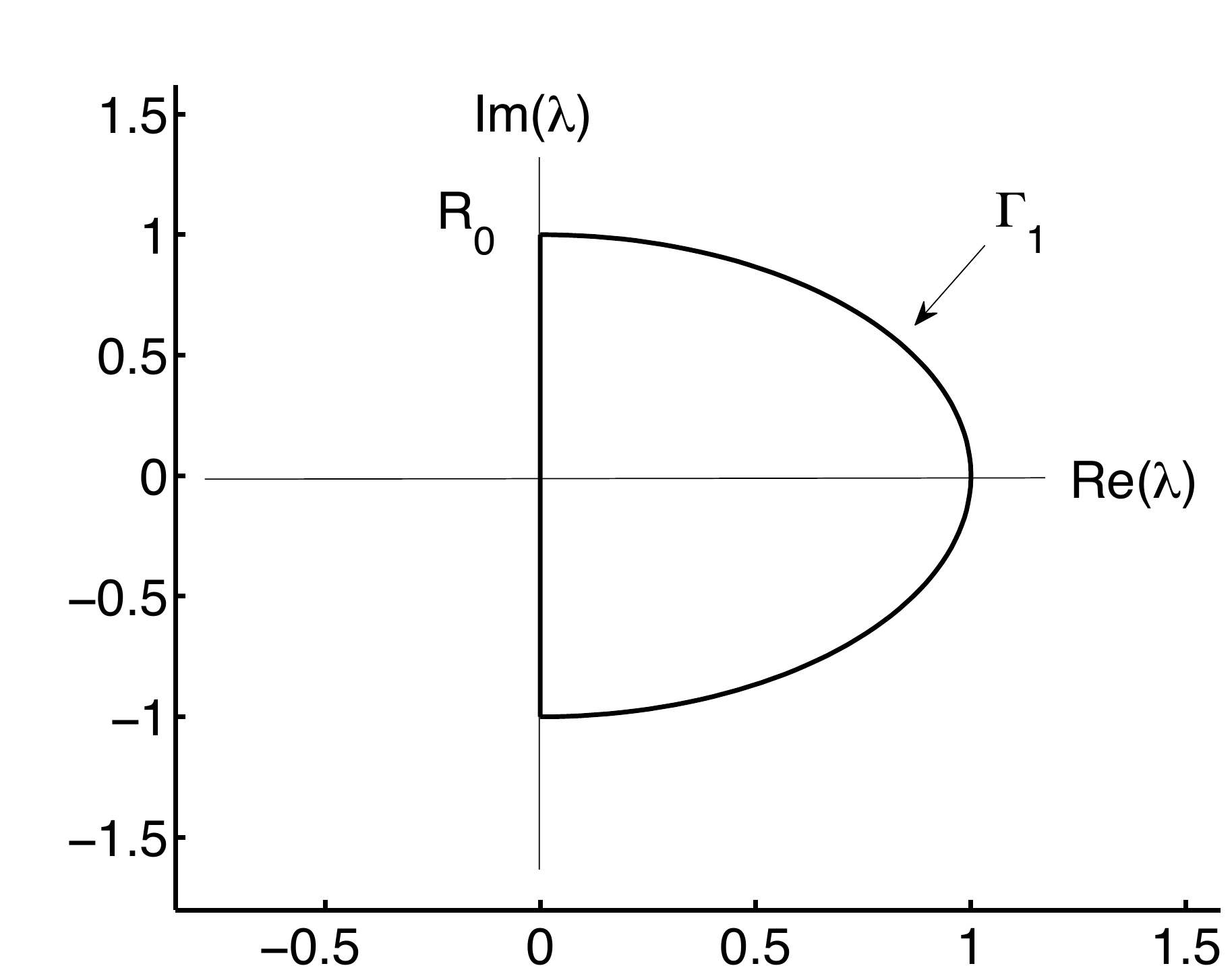}
\caption{Illustrations of the curves $\Gamma_0$ (a) and $\Gamma_1$ (b).  These pictures
are not drawn to reflect the actual dimensions used in the study, but to be a visual aid.}
\label{f:gammacurves}
\end{center}
\end{figure}
\begin{figure}[htbp]
\begin{center}
(a)\includegraphics[scale=0.3]{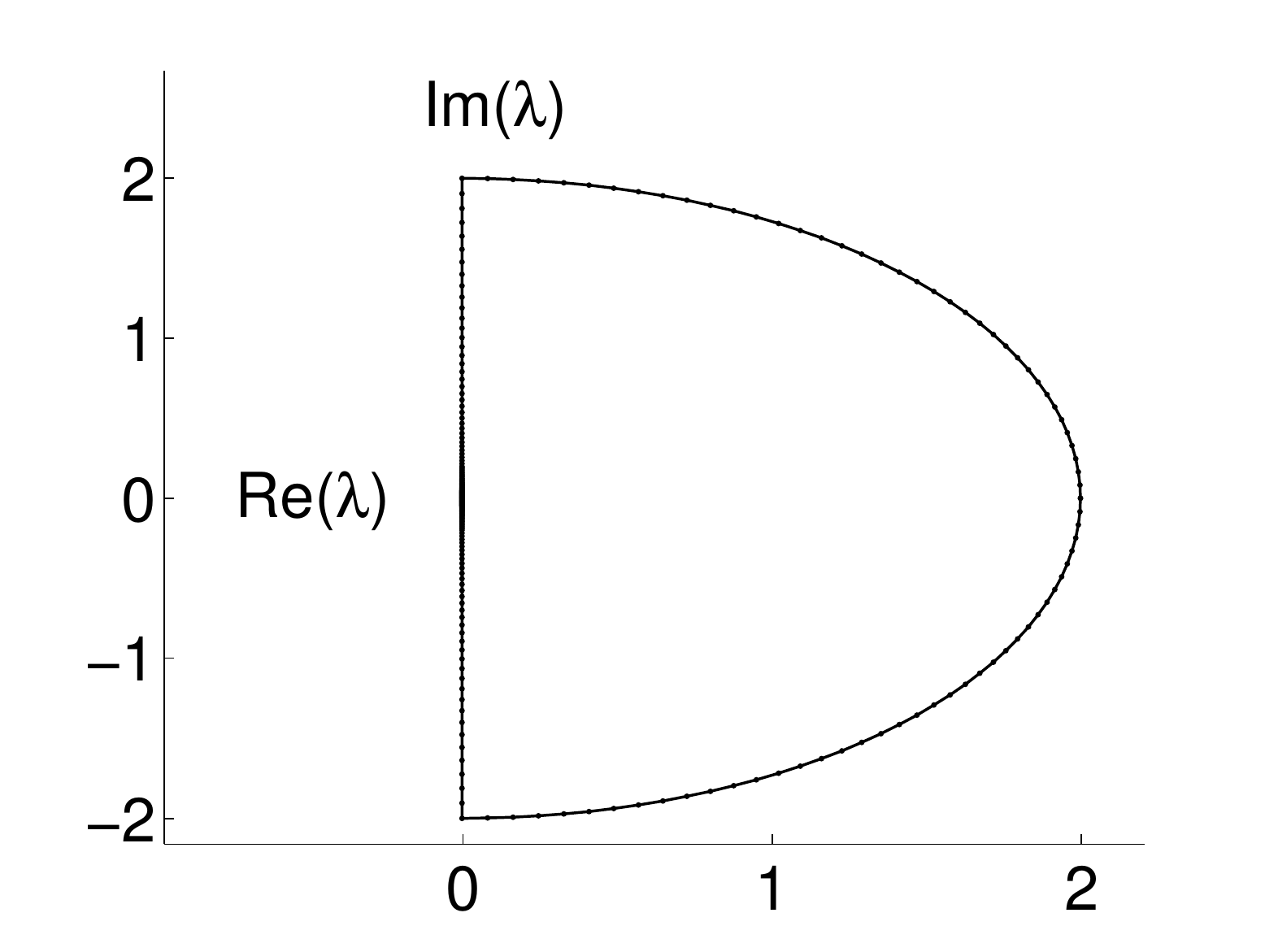}\quad\quad(b)\includegraphics[scale=0.3]{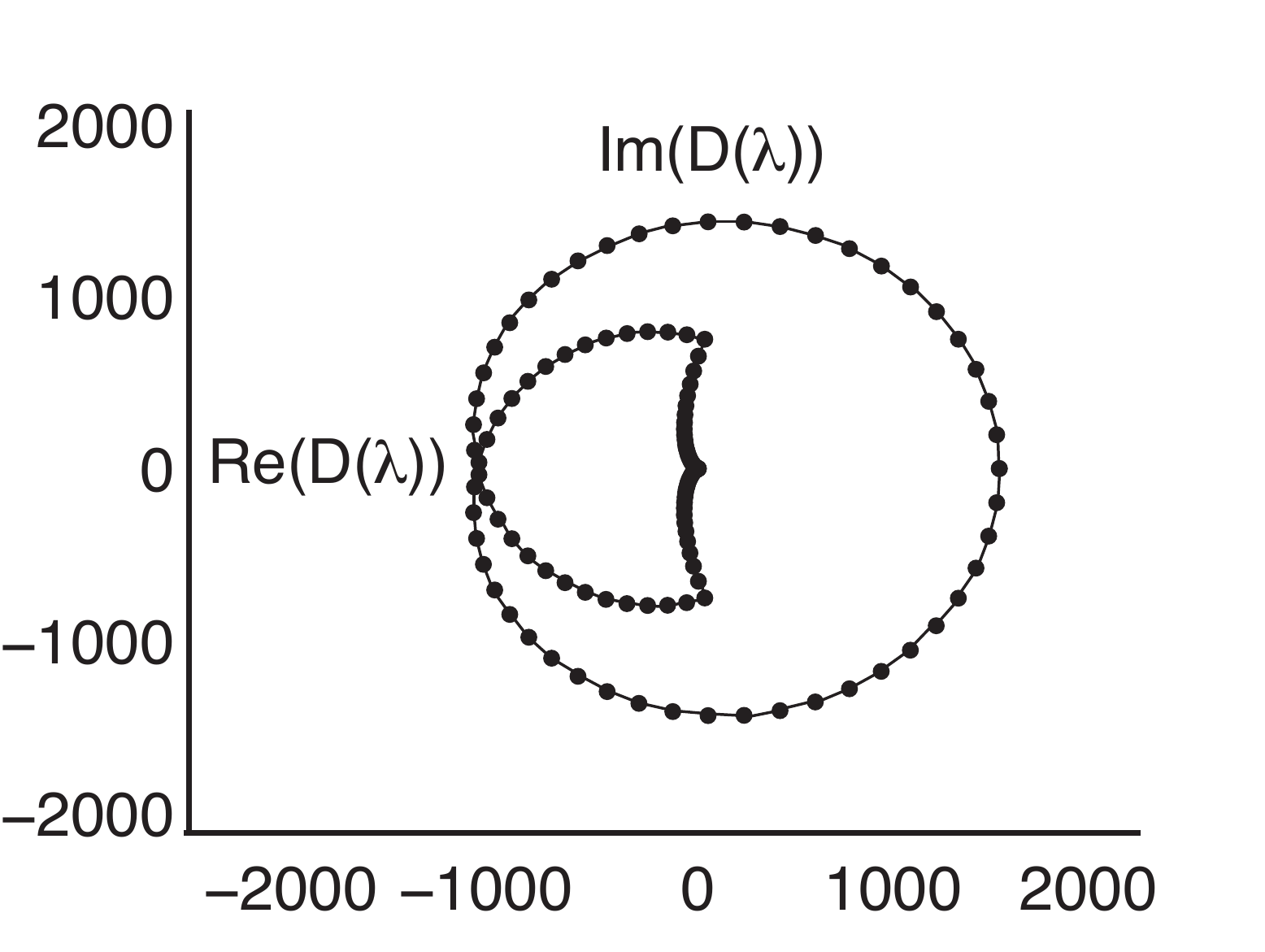}
\caption{
Sample Evans function output, demonstrating the calculation of $n(\xi,\Gamma_0)$ in Step 2 of the above numerical protocol.
We view in (a) a domain contour of form $\Gamma_0$ defined in Step 2 above on which we evaluate the Evans function, 
and in (b) the associated range for $\xi = 0$. Here, $\varepsilon = 0$, $q = 5$, $X = 6.3$, and $\delta = 1.4627$.  
In particular, we see that $n(0,\Gamma_0)=0$ in this case.
}
\label{f:contours}
\end{center}
\end{figure}

\br\label{r:degenerate}\textup{
The set $\mathcal{D}$ introduced in Step 0 above is precisely the parameter set in 
which hypothesis (H2) is valid; see \cite{NR2,JNRZ1} for details.  In particular, Step 0 is a useful
test for co-periodic instability corresponding to bifurcation of periodic solution to a new branch. 
From a numerical standpoint, it should
be noted that Step 3 of the above procedure will fail, producing an infinite loop in the algorithm,
on the set $\mathcal{D}$.  In general, given a numerical tolerance $0<\gamma\ll 1$, the algorithm
is only guaranteed to be well-conditioned off of a $\gamma$-neighborhood of $\mathcal{D}$ since this
guarantees we are away from a degenerate boundary case where 
$\d_\lambda^2D(0,0)$ 
is too small. This will produce sharp stability  boundaries in the region where 
$\d_\lambda^2D(0,0)$ 
is small. Such degenerate boundaries, while seemingly not occurring in the present case, do occur in the non-conservative
Swift-Hohenberg equation studied in Appendix \ref{a:sh}.
}
\er

\br\label{r:lfchoice}
\textup{
It is clear that some care must be taken in choosing the low-frequency cutoff $k_0$ in the above procedure, as it must
be chosen small enough to ensure that the conditions on $\alpha_j$ and $\beta_j$ in Step 4 are sufficient to guarantee
that $\Re(\lambda_j(\xi))<0$ for $|\xi|<k_0$; that is, the cubic order remainder term must be small enough
for $|\xi|<k_0$ as to not dominate the $\mathcal{O}(|\xi|^2)$ terms determined by Taylor expansions \eqref{e:lexp}.
To analyze how small $k_0$ must be, assume $|\xi|<k_0$ and note that Taylor's theorem implies 
\[
\lambda_j(\xi)=\alpha_j\xi+\beta_j\xi^2+\frac{\xi^3}{2}\int_0^1 (1-s)^2\lambda_j'''(s\xi)ds.
\]
Letting $K\gg k_0$ and setting\footnote{In our numerics, we utilized the method of moments to determine the number $M_j$;
see Section \ref{meth:mom} below for details.} $M_j=\max_{|\zeta|=K}|\lambda_j(\zeta)|$, where $k_0$ is given in Step 3 above, it follows by basic interior estimates that for all $s\in[0, 1]$
\[
|\lambda_j'''(s\xi)|=\left|\frac{3!}{2\pi i}\int_{\partial B(0,K)}\frac{\lambda_j(\zeta)}{(\zeta-s\xi)^4}d\zeta\right|\leq \frac{6M_jK}{(K-k_0)^4}.
\]
In particular, upon verifying that $\Re(\alpha_j)=0$ and $\Re(\beta_j)<0$ for $j=1,2$,
we can guarantee $\Re(\lambda_j(\xi))<0$ for $|\xi|<k_0$ provided that
$k_0$ satisfies the inequality
\be\label{e:lfselection}
k_0<\frac{(K-k_0)^4}{K}\cdot\min_{j=1,2}\frac{|\Re(\beta_j)|}{M_j}.
\ee
This serves as a consistency check in the above numerical procedure: if the chosen $k_0$ does not satisfy the stated
bound then one must choose a new $k_0$ which does satisfy the bound and repeat Steps 1--4 with this new $k_0$.
}
\er
\subsubsection{Numerical determination of Taylor coefficients}
Concerning Step 4 in the above procedure, we now explain how the coefficients $\alpha_j$ and $\beta_j$ in \eqref{e:lexp} can be
obtained from the Evans function.  First, assuming 
that
Steps 1-3 have been verified, the analyticity of the Evans function implies
that $D(\cdot,\cdot)$ can be Taylor expanded to third order as
\[
D(\lambda,\xi)=a_0\lambda^2+a_1\lambda\xi+a_2\xi^2+a_3\lambda^3+a_4\lambda^2\xi+a_5\lambda\xi^2+a_6\xi^3+\mathcal{O}(|\lambda|^4+|\xi|^4)
\]
where the coefficients 
$a_0,\ldots,a_6\in\CM$ 
can be calculated from the Evans function via the formula
\be\label{e:evans-doubint}
\partial_\lambda^r\partial_\xi^sD(0,0)=\frac{r!s!}{4\pi^2r!}\oint_{\partial B(0,h)}\oint_{\partial B(0,h)}D(\mu,k)\mu^{-r-1}k^{-s-1}d\mu~dk
\ee
where $h>0$ is sufficiently small.
Substituting the expansions in \eqref{e:lexp} into the equation $D(\lambda_j(\xi),\xi)=0$ and equating second-order terms, we find that the $\alpha_j$
are roots of the polynomial 
$a_0z^2+a_1z+a_2=0$. 
In particular, we can take
\be\label{e:alpha}
\alpha_j=\frac{-a_1+(-1)^{j+1}\sqrt{a_1^2-4a_0a_2}}{2a_0}
\ee
noting that (strong) hyperbolicity of the associated Whitham averaged system is equivalent to the discriminant condition
\be\label{e:disc}
a_1^2-4a_0a_2<0,
\ee
which is easily checked.  Furthermore, equating third-order terms in  $D(\lambda_j(\xi),\xi)=0$ implies that
\be\label{betaj}
\beta_j=-\frac{a_3\alpha_j^3+a_4\alpha_j^2 +a_5\alpha_j+a_6}
{2a_0\alpha_j+a_1}
\ee
where we note that the denominator is non-zero by the definition of the $\alpha_j$ and the discriminant condition \eqref{e:disc}.

\subsubsection{\label{meth:mom}Method of moments}

Finally, we describe how to numerically obtain bounds on the numbers $M_j:=\max_{|\xi|=K}|\lambda_j(\xi)|$.  From the above discussion, the computation of these quantities is important to determine the sign of eigenvalues in the neighborhood of the origin. Here, the idea is to use the Evans function to track the spectral curves $\lambda_j(k)$ as functions of $k$ by using what we refer to as the method of moments. We need to locate two eigenvalues: to this end, we define the first and second moments of the spectral curves $\lambda_j(\xi)$ as
\[
m_1(\xi):=\sum_{j=1,2}\lambda_j(\xi)\quad m_2(\xi):=\sum_{j=1,2}\lambda_j(\xi)^2.
\]
Note that these moments can be directly computed through the Evans function via the formula
$$
m_l(\xi)
=\frac{1}{2\pi i}\oint_{\partial B(0,R)}\frac{\lambda^l\partial_\lambda D(\lambda,\xi)}{D(\lambda,\xi)}~d\lambda.
$$
for $R>0$ sufficiently large so as $\lambda_j(\xi)\in B(0,R)$ for all $|\xi|=K$.  Once the moments $m_l(\xi)$ are computed for a fixed $\xi$, by using a simple Stieltjes trapezoid method, for example,
the corresponding numbers $\lambda_j(\xi)$ can be easily recovered via the formula
\[
\lambda_j(\xi)=\frac{m_1(\xi)+(-1)^j\sqrt{2m_2(\xi)^2-m_1(\xi)^2}}{2},\quad j=1,2
\]
from which the numbers $M_j$ can easily be determined. This method of moments readily extends to an arbitrary number of eigenvalues; see \cite{BZ,BJNRZ1}).
 In order to avoid the ambiguity for the choice of $\lambda_1$ and $\lambda_2$ above, we substitute \eqref{e:lfselection} by
$$
\displaystyle
k_0<\frac{(K-k_0)^4}{KM}\cdot\min_{j=1,2}|\Re(\beta_j)|.
$$
\noindent
where $M=\max_{j=1,2}\max_{|\xi|=K}|\lambda_j(\xi)|$.


\subsubsection{Numerical results}

The above numerical procedures provide a direct and well-conditioned way to numerically verify the spectral stability
hypotheses (H3) and 
(D1)--(D3). 
We now apply this procedure to two cases.
\begin{center}
\mathversion{bold}
{\bf First case: ${\bf \gamma=1,\eps=0}$}\\
 \mathversion{normal}
 \end{center}
\noindent
The case $\eps=0$ corresponds to the ``classic" Kuramoto-Sivashinsky equation, which has received a great amount of numerical and analytical attention over the years.  In particular, stability boundaries are well known for this equation.  We state our results first in this special case for 
verification
of our method against 
these
previously known results. Let us now carefully describe the implementation of the numerical procedure for  the special case of \eqref{e:cks1} for the periodic traveling waves parameterized by $q=5$, and $X=6.3$.  For numerical efficiency, we proceed in our explicit example by
slightly modifying both the content and order 
of the steps of
the above numerical procedure.  As a first step, we notice that when $\eps=0$
and $q=5$, a periodic orbit of \eqref{e:ckssystem} with period $X=6.3$ exists when $\delta=1.4627$, and so we fix $\delta$ at this
value throughout this discussion.  The corresponding solution $\bar{u}$ of \eqref{e:cks1} then satisfies
$\|\bar{u}\|_{L^\infty(\RM)}<4.6509$ and $\|\bar{u}'\|_{L^\infty(\RM)}<6.4856$ 
so that, appealing to Lemma \ref{l:hfbds}, we find
that any unstable spectra of the associated linearization $L$ about $\bar{u}$ lies in the ball $B(0,\tilde{R})$ with
\[
\tilde{R}=\frac{1}{2}\left(\|f''(\bar{u})\bar u'\|_{L^\infty(\RM)}+\|f'(\bar{u})\|_{L^\infty(\RM)}^2+\delta^2+\frac{(1+2\eps^2)^2}{2}\right)
\]
%
Next, we set 
$R_0=15.478>\tilde{R}$ 
%
and $r_0=1$ and define $\Gamma_0$ and $\Gamma_1$ as in \eqref{e:contours}.
Using 153 evenly spaced points in $\lambda$, thereby assuring relative error between successive points varies by less than $0.2$,
we compute that $n(\xi;\Gamma_0)=0$ for 1000 evenly spaced points in $\xi$ in the interval $\xi\in[-\pi/6.3,\pi/6.3]$,
where here we use the scaled lifted polar coordinate method 
(described in Appendix \ref{s:alg}) 
to compute the Evans function $D(\cdot,\xi)$; see \cite{BJNRZ1}.

Now, for numerical efficiency we deviate 
slightly from the basic numerical method described above
and proceed to compute the Taylor expansions \eqref{e:lexp} of the critical
modes near the origin.  Notice that as there are generically two such branches of spectra bifurcating from the origin at $\xi=0$, there
is in general no problem with making this assumption at this point: at the end of our numerical test, we will do a consistency check
to verify that there were indeed only two branches to start with.  Using the 
unscaled lifted polar method 
(described in Appendix \ref{s:alg}) 
to compute the Evans function, we obtain the values of $\alpha_j$ and $\beta_j$ which are recorded
in the $q=5$ row of Table \ref{stabstudy:0eps}. Here, we have used \eqref{e:evans-doubint} with 300 points on the $\lambda$-contour, yielding relative error between
successive values of $\alpha_j$ and $\beta_j$ less than 0.1 and 0.01, respectively, and 1000 points on the $k$-contour, yielding
relative error less than 0.2 in successive steps. From the above Taylor coefficients $\beta_j$, we can use Remark \ref{r:lfchoice} to determine
an appropriate low-frequency cutoff $k_0$.  Here, we choose $K=0.5$, $R=2$ and choose $k_0$ such that
the inequality in \eqref{e:lfselection} holds; a brief calculation shows that $k_0=0.018365$ suffices.
With these choices, we use 427 points in the $\lambda$-contour, assuring relative
error between successive points varies less than 0.2 and verify that $n(\xi;\Gamma_1)=0$ for 1000 evenly spaced
values of $\xi$ satisfying 
$k_0\leq|\xi|\leq\pi/6.3$.

\begin{table}[!th]
\begin{tabular}{|c|c|c|c|c|c|}
\hline
$q$&$\delta$&$\alpha_1$&$\alpha_2$&$\beta_1$&$\beta_2$\\
\hline
4.5 & 1.41 & 0.695+1.34e-05i & -0.695+1.33e-05i & -6.32-0.000214i & -6.32+0.000209i \\ 
\hline
4.6 & 1.42 & 6.86e-09+0.267i & 7.83e-09-0.267i & -5.98-2.02e-07i & -5.98-3.91e-07i \\ 
\hline
4.7 & 1.43 & -4.66e-09-0.792i & -5.37e-09+0.792i & -5.66-2.05e-08i & -5.66+9.42e-08i \\ 
\hline
4.8 & 1.44 & 2.44e-09-1.09i & 1.89e-09+1.09i & -5.34-2.9e-07i & -5.34+2.15e-07i \\ 
\hline
4.9 & 1.45 & 2.38e-09-1.32i & 1.57e-09+1.32i & -5.04-1.15e-07i & -5.04+8.15e-08i \\ 
\hline
5 & 1.46 & -9.44e-09-1.52i & -9.9e-09+1.52i & -4.74+5.02e-07i & -4.74-2.85e-07i \\ 
\hline
5.1 & 1.47 & 4.21e-09-1.7i & 3.06e-09+1.7i & -4.45-3.9e-07i & -4.45+2.9e-07i \\ 
\hline
5.2 & 1.48 & 8.39e-10-1.86i & 2.29e-10+1.86i & -4.17-9.99e-09i & -4.17-7.39e-09i \\ 
\hline
5.3 & 1.49 & -1.81e-07+2i & 1.89e-07-2i & -3.9+1.95e-05i & -3.9+2.13e-05i \\ 
\hline
5.4 & 1.51 & -1.37e-07+2.14i & 1.67e-07-2.14i & -3.64-4.97e-05i & -3.64-4.72e-05i \\ 
\hline
5.5 & 1.52 & -4.29e-07-2.27i & 3.88e-07+2.27i & -3.38-2.82e-05i & -3.38-3.23e-05i \\ 
\hline
5.6 & 1.53 & -3.99e-09-2.4i & -4.91e-09+2.4i & -3.13+4.78e-07i & -3.13-3.34e-07i \\ 
\hline
5.7 & 1.54 & 2.48e-10-2.52i & -4.01e-10+2.52i & -2.88-3.05e-08i & -2.88+2.41e-08i \\ 
\hline
5.8 & 1.55 & -7.36e-10-2.63i & -1.31e-09+2.63i & -2.64+1.54e-07i & -2.64-5.01e-07i \\ 
\hline
5.9 & 1.56 & -1.79e-09-2.74i & -2.24e-09+2.74i & -2.41+6.27e-08i & -2.41-1.07e-07i \\ 
\hline
6 & 1.57 & -1.08e-09-2.84i & -1.79e-09+2.85i & -2.18+3.14e-09i & -2.18+1.26e-08i \\ 
\hline
6.1 & 1.58 & 3.79e-10-2.95i & 2.89e-10+2.95i & -1.95+3.54e-07i & -1.95-3.47e-07i \\ 
\hline
6.2 & 1.59 & 1.91e-10-3.05i & -7.21e-10+3.05i & -1.73-4.21e-08i & -1.73-2.22e-08i \\ 
\hline
6.3 & 1.61 & -1.14e-08-3.14i & 1.18e-08+3.14i & -1.52+2.47e-07i & -1.52-2.87e-08i \\ 
\hline
6.4 & 1.62 & -1.28e-08-3.24i & 1.37e-08+3.24i & -1.3+2.88e-07i & -1.3+2.36e-08i \\ 
\hline
6.5 & 1.63 & -3.76e-08-3.33i & 3.67e-08+3.33i & -1.09+6.92e-07i & -1.09-1.71e-07i \\ 
\hline
6.6 & 1.64 & -4.03e-08-3.42i & 3.99e-08+3.42i & -0.885+6.99e-07i & -0.885-2.95e-07i \\ 
\hline
6.7 & 1.65 & -4.33e-08-3.51i & 4.33e-08+3.51i & -0.681+7.4e-07i & -0.681-2.83e-07i \\ 
\hline
6.8 & 1.66 & -2.36e-09+3.6i & 1.73e-08-3.6i & -0.476-8.45e-05i & -0.476-8.36e-05i \\ 
\hline
6.9 & 1.68 & -2.66e-06+3.68i & 2.66e-06-3.68i & -0.281+0.000738i & -0.281+0.000671i \\ 
\hline
7 & 1.69 & -1.41e-08-3.77i & 1.36e-08+3.77i & -0.0849+1.26e-07i & -0.0843-2.38e-07i \\ 
\hline
7.1 & 1.7 & -1.5e-08-3.85i & 1.39e-08+3.85i & 0.109+1.6e-07i & 0.11-1.23e-07i \\ 
\hline
\end{tabular}
\caption{This table collects the coefficients of the Taylor expansions \eqref{e:lexp} of the critical modes
bifurcating from the origin at $\xi=0$ in the case $\eps=0$ and $X=6.3$.  Here, we vary the integration constant $q$ in \eqref{e:ckssystem}
and note then that $\delta=\delta(q)$ is determined from $q$ via the periodic boundary conditions.  This table
demonstrates a range of values of $q$ corresponding to spectrally stable 6.3-periodic traveling wave solutions
of \eqref{e:cks1}.  In these computations, the evaluation of the integrals \eqref{e:evans-doubint} was performed
using an absolute tolerance of $1e-10$ and relative tolerance of $1e-8$ for the integration in the Floquet
parameter, and an absolute tolerance of $1e-8$ with relative tolerance of $1e-6$ in the integration
in $\lambda$.  Furthermore, the value of the Evans function along the contour was found
with absolute tolerance of $1e-8$ with relative tolerance of $1e-10$.
}
\label{stabstudy:0eps}
\end{table}

Finally we must verify that there are indeed only two spectral branches bifurcating
at $\xi=0$ from the origin.  Following Step 3 of the above procedure, we must simply
find an $r_1\in(r_0,R_0)$ such that $n(\xi;\partial B(0,r_1))=2$ for all $|\xi|\leq k_0$.  We note, however, that
computing the Evans function $D(\lambda,\xi)$ for a given $\xi$ becomes increasingly more difficult and less
accurate as $\Re(\lambda)$ becomes increasingly more negative; see \cite{BJNRZ1} for more details.
As a result, we find it more useful to compute the winding number corresponding to the contour\footnote{Notice that
by the previously computed Taylor expansions, we expect
$\min_{|\xi|\leq k_0}\Re(\lambda_j(\xi))\geq -2\max_j|Re(\beta_j)|k_0^2=-0.0031727$.}
%
\[
\tilde{\Gamma}:=\partial\left\{\lambda\in\CM\ \middle|\ \Re(\lambda)\geq-2\max_{j=1,2}|\Re(\beta_j)|k_0^2\ \textrm{ and }\ \Big|\lambda-2\max_{j=1,2}|\Re(\beta_j)|k_0^2\Big|\leq 2r_0\right\},
\]
from which we find, as expected, $n(\xi;\tilde{\Gamma})=2$ for 1000 evenly spaced values of $\xi$
satisfying $|\xi|\leq k_0$ where again we use 427 points along the $\lambda$-contour to evaluate
the Evans function, ensuring relative error less than 0.2 between successive points.
This observation, together with the (now justified) Taylor expansions computed above, 
implies that
no unstable spectra can exist for any Bloch-frequency $\xi$.

From the above arguments, we conclude with great numerical certainty that the underlying periodic traveling wave
associated with $q=5$, $\eps=0$ and period $X=6.3$ is spectrally stable to small
localized perturbations, in the sense of 
(D1)--(D3) 
and (H3).  This procedure can then be repeated
for different values of $q$, holding $\eps=0$, the period  (X=6.3) and the wave speed ($c=0$) fixed, and the results of these
computations are tabulated in Table \ref{stabstudy:0eps}.  In particular, our numerics indicate
that, as expected, there is a band of spectrally stable periodic traveling wave solutions, corresponding
to different values of $q$, of \eqref{e:cks1} with $\eps=0$ and fixed period $X=6.3$.

\medskip
{\bf Comparison with previous results.}
As there is a well known approximate stability boundary in the case $\eps=0$ provided in \cite{FST}, it seems
worthwhile to compare the results of Table \ref{stabstudy:0eps} to theirs.
In \cite{FST}, the authors consider $2\pi$-periodic solutions of the equation
\be\label{e:fstform}
\partial_t u+u\partial_x u+\partial_x^2 u+\nu \partial_x^4 u=0,\quad\nu>0.
\ee
Going to Fourier space, it is clear that when $\nu>1$ all Fourier modes of the solution are linearly damped and hence
the solution converges to a spatially homogeneous state as $t\to\infty$.  However, when $\nu<1$  there will be a finite
number of excited modes leading to a non-uniform solutions.  Setting $\nu=1-\eta$ then, a stability analysis is conducted
for $2\pi$-periodic solutions of \eqref{e:fstform} with $0<\eta<0.69$.  In particular, the authors find that such solutions
are spectrally stable for $\eta\in(\eta_1,\eta_2)$ with $\eta_1\in(0.25,0.3)$ and $\eta_2\in(0.4,0.45)$  In order to compare these results
to those provided in Table \ref{stabstudy:0eps}, we notice that by setting $\alpha=\frac{6.3\nu^{2/7}}{2\pi}$ and introducing
the new variables $\bar{u}=\alpha^{-3} \nu^{-1/7}u$, $\bar{x}=\alpha\nu^{-2/7} x$, and $\tilde{t}=\alpha^4\nu^{-1/7}$ the $2\pi$-periodic
solutions of \eqref{e:fstform} correspond to 6.3-periodic solutions of the equation
\[
\partial_{\bar{t}}\bar{u}+\bar{u}\partial_{\bar{x}}\bar{u}+\rho\partial^2_{\bar{x}}\bar{u}+\partial_{\bar{x}}^4\bar{u}=0
\]
where $\rho=\Big(\frac{2\pi}{6.3}\Big)^2\nu^{-1}$.  Letting $\nu=1-\eta$ then, the results of \cite{FST} suggest
we should have stability for $\rho\in(\delta_1,\delta_2)$ where
$\delta_1\in(1.32623,1.42096)$ and $\delta_2\in(1.65778,1.80849)$ in our $\rho$-parametrization: these predictions
are consistent with the data in Table \ref{stabstudy:0eps} where we find $\delta_1\approx 1.411$ and $\delta_2\approx 1.701$.

\br\label{numrmk}
\textup{
It is worth mentioning that the numerical method of \cite{FST}
is completely different from ours, so that the close agreement between
results gives a useful check on both sets of numerics.
Specifically, for frequencies $(\xi,\lambda)$ bounded away
from the origin, they use a Galerkin method similar to that 
used by us to generate figures using SPECTRUW; that is, they 
compute the spectra of a truncated infinite-dimensional matrix.
In the numerically delicate small-frequency regime, they
approximate the spectra by perturbation expansion around
critical modes, as we do.  
However, they accomplish this in a different way, by direct spectral
perturbation expansion of the underlying linearized operator about
the wave.
They point out also the relation between this spectral expansion and
the formal modulation (i.e., Whitham) equations expected to govern
asymptotic behavior.
Though the numerics of \cite{FST} are well-conditioned,
and based on sound functional-analytic principles,
there is no attempt made there to estimate numerical computation error.
Indeed, for this type of method, this seems to us a complicated task.
By contrast, Evans-function/winding number methods come effectively
with built-in error bounds, given by the ODE tolerance and 
the requirements of Rouch\'es Theorem to guarantee validity of the winding 
number.
}
\er

\begin{table}[!th]
\begin{tabular}{|c|c|c|c|c|c|}
\hline
$q$&$\delta$&$\alpha_1$&$\alpha_2$&$\beta_1$&$\beta_2$\\
\hline
4.6 & 1.41 & 0.694+0.272i & -0.694+0.272i & -6.46-1.05i & -6.46+1.05i \\ 
\hline
4.7 & 1.42 & 5.44e-08+0.518i & -6.31e-09+0.0311i & -9.16+1.37e-06i & -3.11-1.56e-07i \\ 
\hline
4.8 & 1.43 & -2.91e-09-0.499i & -4.35e-09+1.05i & -4.86+9.51e-10i & -6.77+2.62e-07i \\ 
\hline
4.9 & 1.44 & 3.51e-09+1.35i & 2.21e-09-0.791i & -6.2+3.6e-07i & -4.82-1.93e-07i \\ 
\hline
5 & 1.45 & -2.02e-07+1.58i & 7.16e-08-1.02i & -5.78-1.52e-05i & -4.64-4.23e-06i \\ 
\hline
5.1 & 1.46 & 9.62e-10-1.21i & 9.17e-10+1.78i & -4.43-6.08e-08i & -5.42+1.06e-07i \\ 
\hline
5.2 & 1.47 & 1.55e-10-1.38i & -4.18e-10+1.96i & -4.2+1.05e-07i & -5.08-1.52e-07i \\ 
\hline
5.3 & 1.49 & 2.81e-09-1.54i & 2.59e-09+2.12i & -3.97-1.89e-07i & -4.77+2.67e-07i \\ 
\hline
5.4 & 1.5 & 1.43e-09-1.68i & 1.16e-09+2.27i & -3.74-1.73e-08i & -4.47-1.77e-08i \\ 
\hline
5.5 & 1.51 & 6.84e-10-1.81i & 3.05e-10+2.41i & -3.51-3.51e-08i & -4.19+2.86e-08i \\ 
\hline
5.6 & 1.52 & -9.72e-11-1.94i & -8.41e-10+2.54i & -3.28+2.7e-08i & -3.92-5.11e-08i \\ 
\hline
5.7 & 1.53 & -1.05e-07+2.67i & 5.65e-08-2.06i & -3.65+5.4e-06i & -3.06+2.24e-06i \\ 
\hline
5.8 & 1.54 & 2.55e-09-2.17i & 2.12e-09+2.79i & -2.84-9.26e-08i & -3.39+6.11e-08i \\ 
\hline
5.9 & 1.55 & -2.51e-09+2.9i & 5.59e-09-2.28i & -3.14-2.14e-06i & -2.63-1.01e-06i \\ 
\hline
6 & 1.56 & -4.8e-10-2.38i & -1.25e-09+3.01i & -2.42+6.1e-08i & -2.9-5.96e-08i \\ 
\hline
6.1 & 1.57 & -4.61e-07-2.48i & 6.86e-07+3.12i & -2.21-5.33e-05i & -2.66-0.000108i \\ 
\hline
6.2 & 1.58 & -1.66e-09-2.58i & 1.4e-09+3.22i & -2.01+2.63e-07i & -2.43+1.72e-07i \\ 
\hline
6.3 & 1.59 & -8.86e-10-2.67i & -1.57e-09+3.32i & -1.81+1.27e-07i & -2.2-1.13e-07i \\ 
\hline
6.4 & 1.6 & -4e-09+3.42i & 5.31e-10-2.76i & -1.98+4.16e-06i & -1.61+2.22e-06i \\ 
\hline
6.5 & 1.62 & -5.27e-10-2.85i & -7.42e-10+3.52i & -1.42+1.32e-07i & -1.76+3.66e-07i \\ 
\hline
6.6 & 1.63 & 1.34e-09-2.94i & 3.55e-10+3.61i & -1.23+2.29e-07i & -1.55-2.84e-07i \\ 
\hline
6.7 & 1.64 & -1.61e-09+3.7i & 1.98e-09-3.02i & -1.34+5.22e-08i & -1.05+7.35e-09i \\ 
\hline
6.8 & 1.65 & -6.87e-08+3.8i & 4.71e-08-3.1i & -1.13+4.11e-06i & -0.865+1.19e-06i \\ 
\hline
6.9 & 1.66 & -8.73e-07-3.18i & 1.39e-06+3.88i & -0.688-3.38e-05i & -0.925-6.86e-05i \\ 
\hline
7 & 1.67 & -9.68e-09+3.97i & 8.81e-09-3.26i & -0.72-6.88e-06i & -0.51-4.02e-06i \\ 
\hline
7.1 & 1.68 & -5.29e-09+4.06i & 4.25e-09-3.33i & -0.521+1.23e-05i & -0.337+6.76e-06i \\ 
\hline
7.2 & 1.7 & -2.46e-08-3.41i & 3.62e-08+4.14i & -0.167+3.29e-07i & -0.324-2.02e-07i \\ 
\hline
7.3 & 1.71 & -5.07e-07-3.48i & 6.83e-07+4.23i & 0.0037+0.000203i & -0.124+0.000359i \\ 
\hline
7.4 & 1.72 & -2.72e-08-3.55i & 4.14e-08+4.31i & 0.167+1.17e-06i & 0.0642+1.07e-06i \\ 
\hline
\end{tabular}
\caption{Similar to Table \ref{stabstudy:0eps}, this table collects the Taylor coefficients
of the critical modes in \eqref{e:lexp} in the case $\eps=0.2$ and $X=6.3$.  This table indicates
the existence of a range of values of $q$ corresponding to spectrally stable 6.3-periodic traveling
wave solutions of \eqref{e:cks1}.  These computations were preformed with the same absolute and relative
tolerances used for the computations discussed in Table \ref{stabstudy:0eps}.
}
\label{stabstudy:0.2eps}
\end{table}

\begin{center}
\mathversion{bold}
{\bf Second case: ${\bf \gamma=1,\eps=0.2}$}\\
 \mathversion{normal}
 \end{center}

\noindent
This choice corresponds to a generalized Kuramoto-Sivshinsky equation: the purpose here was to see whether we have the same qualitative picture of a band of spectrally stable periodic traveling waves. Here $X=6.3$ is fixed.  Employing
our above numerical procedure, in an analogous fashion 
as for the 
$\eps=0$ case considered above, we summarize
our findings in Table \ref{stabstudy:0.2eps}.  As before, we find a range of values of the integration constant $q$ associated
with periodic traveling wave solutions of \eqref{e:cks1} which are spectrally stable to small localized perturbations,
in the sense that 
(D1)--(D3) 
and (H3) hold; see also Figure \ref{f:orbits} (b).

\mathversion{bold}
\subsection{A full parameter study: $\gamma=\delta$, $\eps^2+\delta^2=1$}
\mathversion{normal}

\begin{table}[!th]
\begin{center}
\begin{tabular}{|c|c|c||c||c|c|c|}
\hline
$\delta$&$X_L$&$X_U$&\quad&$\delta$&$X_L$&$X_U$\\
\hline
\hline
0.1 & 9.03125 & 25.7188 &\quad&0.045 & 8.98047 & 25.9219 \\
\hline
0.095 & 9.03125 & 25.7188 &\quad&0.04 & 8.98047 & 26.1406 \\
\hline
0.09 & 9.03125 & 25.7188 &\quad&0.035 & 8.98047 & 26.1406 \\
\hline
0.085 & 9.03125 & 25.9062 &\quad&0.03 & 8.98047 & 26.1406 \\
\hline
0.08 & 9.03125 & 25.9062 &\quad&0.025 & 8.98047 & 26.1406 \\
\hline
0.075 & 9.03125 & 25.9062 &\quad&0.02 & 8.98047 & 26.1406 \\
\hline
0.06 & 8.98047 & 25.9219 &\quad&0.015 & 8.98047 & 26.1406 \\
\hline
0.055 & 8.98047 & 25.9219 &\quad&0.01 & 8.98047 & 26.1406 \\
\hline
0.05 & 8.98047 & 25.9219 &\quad&0.005 & 8.98047 & 26.1406 \\
\hline
\end{tabular}
\end{center}
\caption{Table of the periods $X_L$ and $X_U$ corresponding respectively to the lower and upper stability boundary of the generalized KS system, $\partial_t u +\partial_x(u^2/2)+\partial_x^3u+\delta(\partial_x^2 u+\partial_x^4 u)=0$  as $\delta \to 0^+$.
}
\label{t:stabbdry-0delta}
\end{table}

Finally, we conclude our spectral stability analysis by considering stability of periodic traveling wave solutions of the 
generalized Kuramoto-Sivashinsky equation, under the particular scaling
\be\label{e:kssing}
\partial_t u +\partial_x(u^2/2)+\eps\partial_x^3u+\delta(\partial_x^2 u+\partial_x^4 u)=0,
\ee
and $\eps^2+\delta^2=1$.  
It should be emphasized that  the singular limit $\delta\to 0$ of \eqref{e:kssing} 
arises naturally in the study of small amplitude roll-waves on the surface of a viscous 
liquid thin film running down an inclined plane. Indeed, one can  derive \eqref{e:kssing} 
either from shallow water equations with friction at the bottom or free surface Navier-Stokes 
equations, both at the transition to instability of constant steady states (fluid height and velocity 
are constant in the shallow water description) 
for small amplitude disturbances and in the small wave number regime; see \cite{W,YY} 
for more details on this derivation. In this case, periodic traveling waves solutions 
of \eqref{e:kssing} correspond to well-known hydrodynamical instabilities, 
known as
{\it roll-waves}; see \cite{NR1} and \cite{JZN} for the spectral and nonlinear stability analysis
of these roll-waves and for the analysis of their corresponding Whitham equations.

In the limit $\delta\to 0$, the governing equation \eqref{e:kssing} reduces to the integrable KdV equation
\be\label{e:kdv1}
\partial_t u+u\partial_x u+\partial_{x}^3 u=0
\ee
where it is known that all periodic traveling wave solutions are spectrally stable to perturbations in $L^2(\RM)$
and nonlinearly (orbitally) stable in $L^2_{\rm per}([0,nX])$ where $X$ denotes the period of the underlying elliptic
function solution; see \cite{BD}.
As discussed in detail in Appendix \ref{a:kdv}, the elliptic function solutions given in \eqref{e:kdvsoln} which
satisfy the selection principle \eqref{solvability} may be continued for $0<\delta\ll 1$
to a three-parameter family of periodic traveling
wave solutions of \eqref{e:cks1} parameterized by translation, period, and spatial mean over a period.

Concerning the stability of these ``near KdV profiles", we numerically observe that for $\delta>0$ sufficiently small
there exist numbers $X_L(\delta)$ and $X_U(\delta)$ such that the associated X-periodic traveling wave solution
of \eqref{e:kssing} is spectrally stable to perturbations in $L^2(\RM)$ 
provided $X_L(\delta)<X<X_U(\delta)$,
with $X_L\approx 8.49$ and $X_U\approx 26.17$;
see Table \ref{t:stabbdry-0delta} for details.  This numerical observation complements the well known result that the
``near KdV solitary wave" profiles of \eqref{e:kssing} are spectrally unstable to localized perturbations; see \cite{PSU,BJNRZ2}.
Analytical verification of these numerical 
observations, 
though outside the scope of the present work, 
would be a very interesting direction for future investigation.
See \cite{JNRZ4} for substantial progress in this direction.

\begin{figure}[htbp]
\begin{center}
\includegraphics[scale=.35]{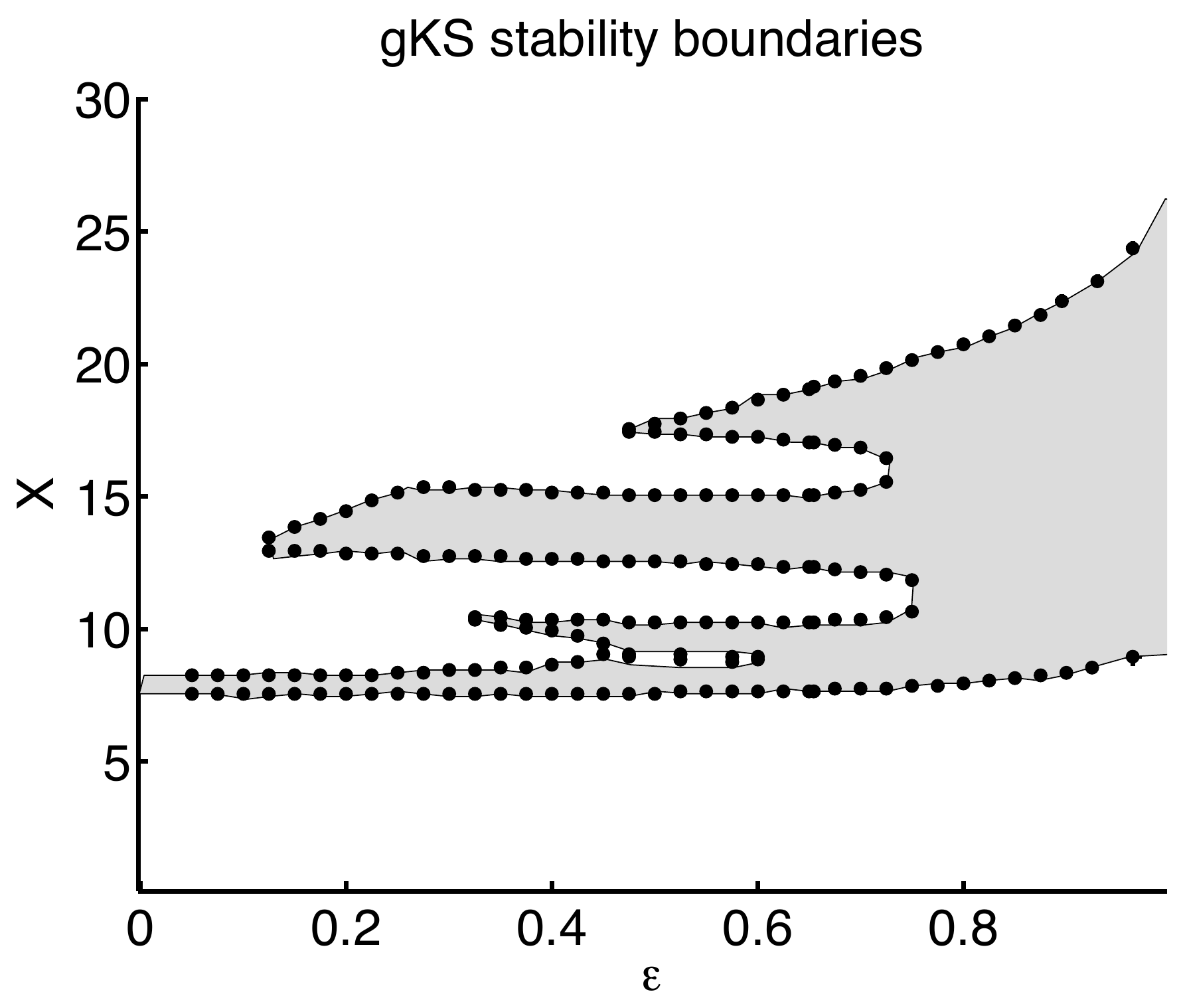}
\caption{Plot of the stability boundaries (in the period $X$) versus the parameter $\eps$. 
Here, $\delta=\sqrt{1-\eps^2}$ is fixed by
the choice of $\eps$ and the shaded regions correspond to spectrally stable periodic traveling waves.
In the limits $\eps\to 0$ and $\eps\to 1$, we see the existence of only a single
band of spectrally stable periodic traveling waves.  For intermediate values, however, several bands emerge
and the stability picture becomes much more complicated.
}\label{f:gKSstabislands}  
\end{center}
\end{figure}

At first sight, one may attempt to deduce from the previous computations $\gamma=1, \eps=0, 0.2$ or $\gamma=\delta\to 0, \eps=1$ that, generically, for a fixed set of parameters, there exists a band of spectrally stable periodic traveling waves. In fact, the situation is slightly more complicated and we carry out a spectral stability analysis for the full set of parameters $\eps^2+\delta^2=1$ and $\gamma=\delta$. Up to a time and space rescaling, this analysis includes all the particular cases treated previously: see Figure \ref{f:gKSstabislands}.  
This picture confirms that both for $\eps\approx 0$ or $\delta\approx 0$ 
there is a single band of spectrally stable periodic waves. However for 
intermediate values of $\eps$ and $\delta$, we clearly see that there a 
several bands of spectrally stable periodic waves and bounded bands of 
unstable periodic waves: this may be, among other things, connected to 
the various bifurcations of periodic waves occurring in this intermediate regime.  Moreover, throughout this intermediate regime
the nature of the transition to instability can be more complicated
than previously seen.  For example, it may happen that one diffusion coefficient $\beta_j$ becomes positive 
while the other stays strictly negative; see Figure \ref{f:SHeigpic2}.  This stands in contrast to the
numerics in Tables \ref{stabstudy:0eps} and \ref{stabstudy:0.2eps}.

\begin{figure}[htbp]
\begin{center}
(a) \includegraphics[scale=.35]{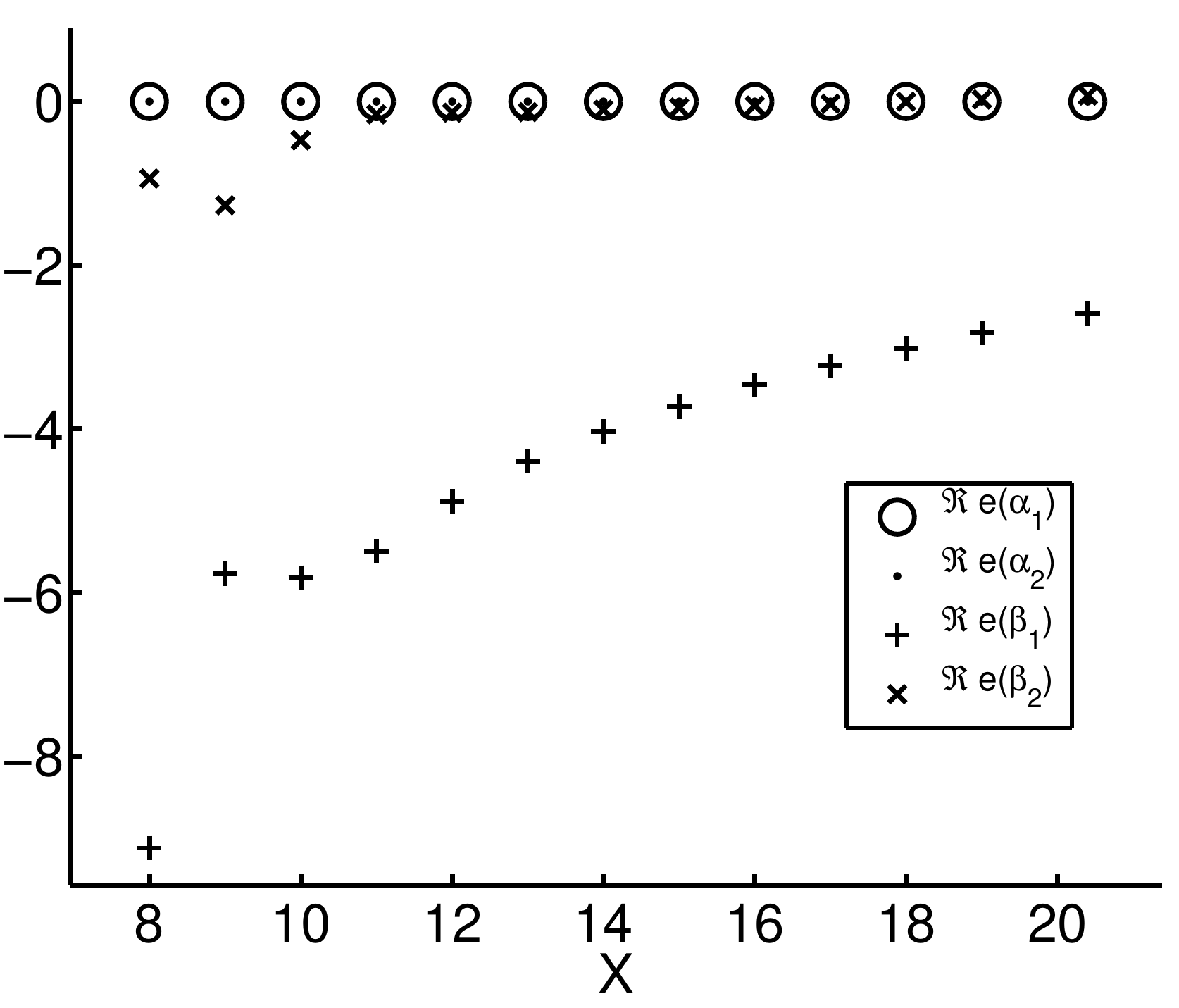}\quad(b) \includegraphics[scale=.35]{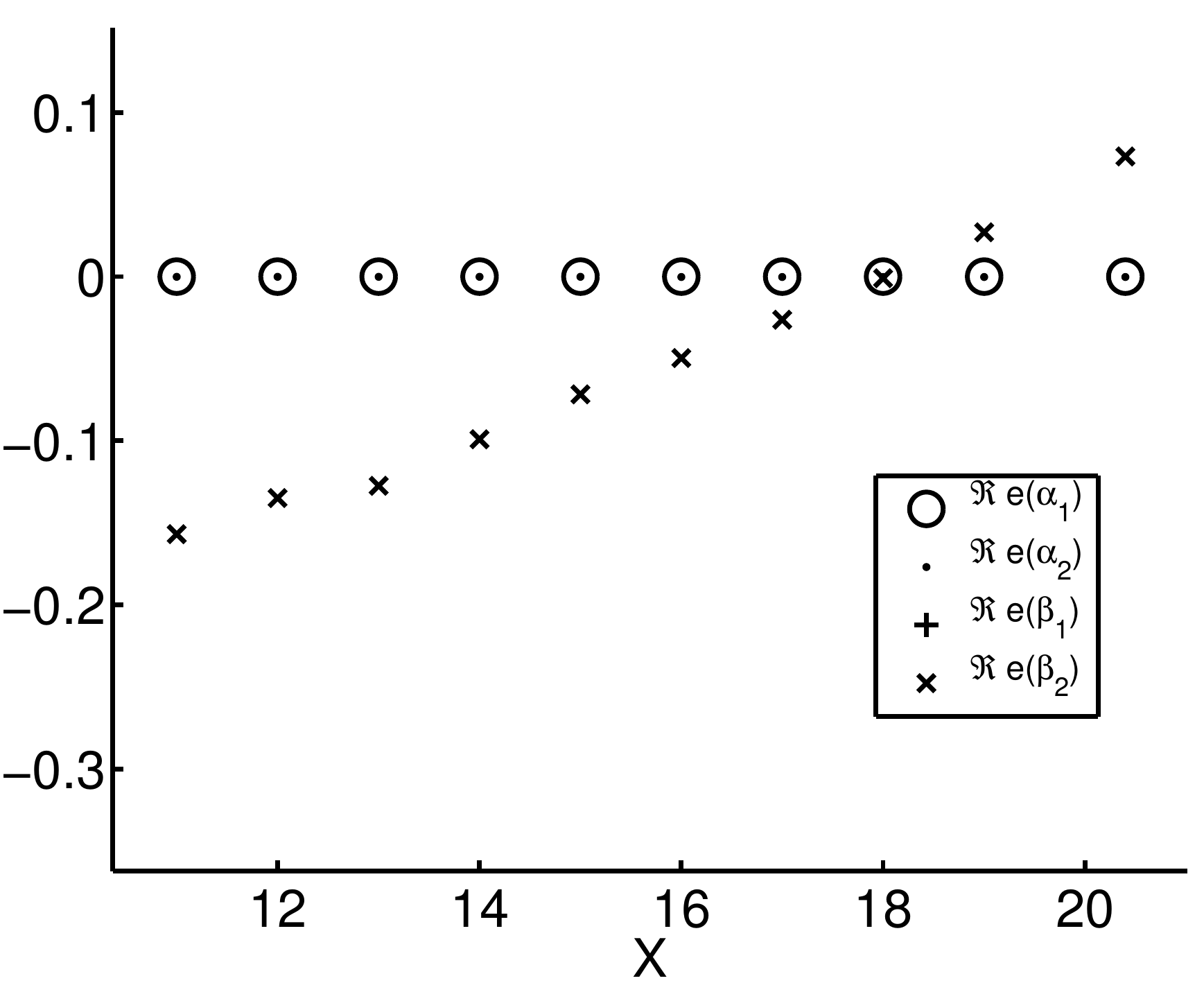}
\caption{Plot of the real part of $\alpha_j$ and $\beta_j$ from the spectral curves against the period $X$ for gKS when $\varepsilon = 0.8$ and $\delta = 0.6$. Figure (b) is a zoom of figure (a). 
}\label{f:SHeigpic2}  
\end{center}
\end{figure}

\medskip
{\bf Comparison with previous results.}
The results of our numerics agree very well with previous results of
\cite{BaN,CDK} obtained by completely different means.
In particular, our estimates for the limiting stability boundaries
$X_L$ and $X_U$ in the KdV limit are quite close
to predictions obtained in \cite{BaN} by formal singular perturbation
analysis.
Likewise, the global stability boundaries displayed in
Figure \ref{f:gKSstabislands} agree quite well with
numerical results displayed in
Figure 12 p. 316 of \cite{CDK}.
The methods used in \cite{CDK} are again Galerkin-type as described in Remark
\ref{numrmk}, quite different from the Evans function methods used here,
so give a useful independent check.

To connect our results to those of \cite{CDK}, we give
a brief lexicon between our paper and that of Chang et al.
The figure 12 in \cite{CDK} involves $\alpha$, $\delta$,
whereas our Figure is in terms of parameters $X$ and $\epsilon$.
One has period $X=2*\pi/\alpha$ and
coefficient $\epsilon=\delta/\sqrt(1+\delta^2)$.
On the right hand side of Figure 12 \cite{CDK} ($\delta=8/\epsilon=0.99$), 
corresponding to the KdV limit, one finds stability region
$\alpha\in(0.24,0.74)$ which (after translation to our coordinates)
agrees with the bounds found in this paper and in \cite{BaN}. 
The first island of instability (starting from $\epsilon=1$) is found at
$\delta=1.1$ which corresponds to $\epsilon=0.74$.
We note that $\delta=1.1$ in \cite{CDK} is
found to represent also the transition where new branches
of periodic wave trains are found as $\delta$ decreases below $1.1$.
There is also a nice discussion on the nature of the instabilities at this
and other singularities of the stability boundary curves (p.316 \cite{CDK}
as well),
which we recommend to the interested reader.

\subsection{The Whitham system and time evolution studies} \label{num:whitham}

In the previous section, we demonstrated numerically the existence of spectrally stable periodic traveling
wave solutions of \eqref{e:cks1} in various circumstances, both in the classically studied case $\eps=0$ (Kuramoto-Sivashinsky equation) and in the case $\eps\neq 0$ arising from general thin film flows.  In each scenario considered, we found that all constant solutions are spectrally
unstable and that, for a fixed period, an interval of the integration constant $q$ (alternatively, of $\delta$) corresponding to spectral stability; see Figure \ref{f:orbits}. 

The goal of this final subsection
is to shed light 
on
the dynamics of the underlying periodic traveling wave when subject to a small
integrable perturbation: in particular, we wish to connect the observed (numerical) behavior
of solutions to the theoretical predictions given by Theorem \ref{main}.  As we will see, the long-time behavior of
such solutions to low-frequency perturbations can be well-approximated by a formal second-order Whitham modulation
equation, with the dynamics intimately related to the properties of these derived amplitude equations.

We consider the behavior of a fixed periodic traveling wave solution $\bar{u}$ of \eqref{e:cks1},
corresponding, say, to $\alpha=0$ and 
$\beta=\bar\beta$ 
in \eqref{e:manifold}, when subject to perturbations with characteristic wavenumber $|\nu|\ll 1$, on space and time scales
$(X,T)=(x/\nu,t/\nu)$.  Using a nonlinear optics (WKB) approximation
\be\label{e:wkb}
u(x,t)=u^0\left(X,T,\frac{\psi(X,T)}{\nu}\right)+\nu u^1\left(X,T,\frac{\psi(X,T)}{\nu}\right)+\mathcal{O}(\nu^2)
\ee
where 
$y\mapsto u^j(X,T,y)$ are unknown 1-periodic functions, 
to find approximate solutions 
of \eqref{e:KS}, it follows
by substitution of the ansatz \eqref{e:wkb} into 
%
\eqref{e:KS}
%
in the re-scaled 
$(X,T)$-coordinates 
and collecting terms
of leading order that we can take
\[
u^0(X,T,y)=U\left(y-\alpha(X,T)-c(\beta(X,T))t;\beta(X,T)\right),
\]
where $\alpha(X,0)=0$ and 
$\beta(X,0)=\bar\beta$.
%
For simplicity, we specialize now the discussion to the case where $\beta=(k,c)$, with $k$ the wave number and $c$ the wave speed.
%

As described in a
more general setting in \cite{Se},
continuing the above calculation to higher 
orders,
integrating over one period with respect to the fast variable $y$ and noting
by periodicity that integrals of perfect derivatives vanish, 
we find at first order in $\nu$ the 
modulation system
\ba\label{e:whitham}
M(k,c)_t+ F(k,c)_x&=0,\\
k_t + (ck)_x&=0
\ea
where $k=\psi_X$ is the local frequency, $c(X,T)=-\psi_T(X,T)/\psi_X(X,T)$ denotes the wave speed,
$M(k,c)$ denotes the mean of $u^0$ over
one period, and $F(k,c)$ is the mean of $f(u^0)$.
Here, the second equation represents simply equality of mixed partial
derivatives of $\psi$.

%
In the more specific setting of \eqref{e:cks1}, $F(k,c)$ is the mean of $(u^0)^2/2$, and 
using the Galilean invariance $x\to x-ct$, $u\to u+c$, we may reduce \eqref{e:whitham} to
\be\label{e:whitham2}
\begin{aligned}
(m(k)+c)_T+\left(H(k)-m(k)c\right)_X&=0\\
k_{T}+(ck)_X&=0,
\end{aligned}
\ee
where $H(k)=F(k,0)$ and $m(k)$ denotes the mean over one period for a zero-speed wave of frequency $k$. In the classical situation, considered in \cite{FST}, where moreover $\eps=0$, symmetry of the governing equation ensures $m(k)=0$ and \eqref{e:whitham}
reduces to
\ba\label{ksw}
c_t+ (H(k))_x&=0,\\
k_t + (ck)_x&=0.
\ea

Linearizing 
the latter 
about the constant $(\bar k,0)$, corresponding to a background wave $\bar{u}=U(\cdot;(\bar k,0))$,
yields the linear \emph{scalar} wave equation
\be\label{e:wave}
k_{TT}\ -\ \bar k H'(\bar k)\ k_{XX}=0
\ee
in the local wave number $k$.
In particular, the critical spectral curves $\lambda_j(\xi)$ bifurcating from the origin at the $\xi=0$ state
agree to first-order with the dispersion relation of \eqref{e:wave}.  Clearly then, hyperbolicity
of \eqref{e:wave} is a necessary requirement for the spectral stability of the underlying periodic traveling wave
$\bar{u}$.

Continuing to consider the classic case $\eps=0$
and a zero-speed wave, we may introduce 
similarly as in \cite{NR2}, higher order corrections to the WKB approximation \eqref{e:wkb}
to find that the critical spectral curves $\lambda_j(\xi)$ actually
agree to \emph{second} order with the dispersion relation of the viscoelastic wave 
equation\footnote{The absence of terms like $k_{XXX}$ originates in the symmetry of the governing equation under $(x,c)->(-x,-c)$.}
%
\be\label{e:viscowave}
k_{TT}\ -\ \bar k H'(\bar k)\ k_{XX}=d(\bar k)~k_{TXX}
\ee
for some $d(\bar k)$ depending only on the underlying wave $U(.;\bar \beta)$. 
As a result, we find that the diffusive spectral stability conditions 
(D1)--(D3) 
are equivalent to 
$\bar k H'(\bar k)>0$, corresponding to hyperbolicity of the 
linearized first-order Whitham averaged system \eqref{e:wave}, and $\Re(d(\bar k))<0$, 
corresponding to diffusivity of the second-order linearized Whitham
averaged system \eqref{e:viscowave}. Thus, the periodic traveling waves below the lower stability boundary ($\delta\approx 1.411$)
in Table \ref{stabstudy:0eps} correspond to a loss of hyperbolicity in the first-order Whitham system, while those
above the upper stability boundary ($\delta\approx 1.701$) correspond to a ``backward damping" effect corresponding
to the amplification of the local wave-number 
$k(X,T)$ 
on time-scales of order $T=t/\nu$. These observations are discussed in more detail in \cite{FST} 
and generalizations to cases $\eps\neq0$ (and $\bar c$ arbitrary) may be found in \cite{NR2}.

Equation  \eqref{e:viscowave}
recovers the formal prediction of ``viscoelastic behavior'' of
modulated waves carried out in \cite{FST} and elsewhere, or
``bouncing'' behavior of individual periodic cells. 
Put more concretely, \eqref{e:viscowave} predicts that the
maximum of a perturbed periodic solution should behave approximately
like point masses connected by viscoelastic springs.
However, we emphasize that {\it such qualitative behavior}- in particular,
the fact that the modulation equation is of second order-  
{\it does not derive only
from Galilean or other invariance of the underlying equations},
as might be suggested by early literature on the subject, 
{\it but rather from the more general structure of conservative} (i.e.,
divergence) {\it form} 
\cite{Se,JZ1}.\footnote{As discussed further in \cite{Z3},
conservation of mass lies outside the usual Noetherian formulation.}
Indeed, for any choice of $f$, $\lambda_j(\xi)$ may be seen to
agree to second order with the dispersion relation for an appropriate
diffusive correction of \eqref{e:whitham}, a generalized
viscoelastic wave equation.  See \cite{NR1,NR2} for further discussion
of 
Whitham averaged equations and their derivation.

\begin{figure}[htbp]
\begin{center}
(a)\includegraphics[scale=.3]{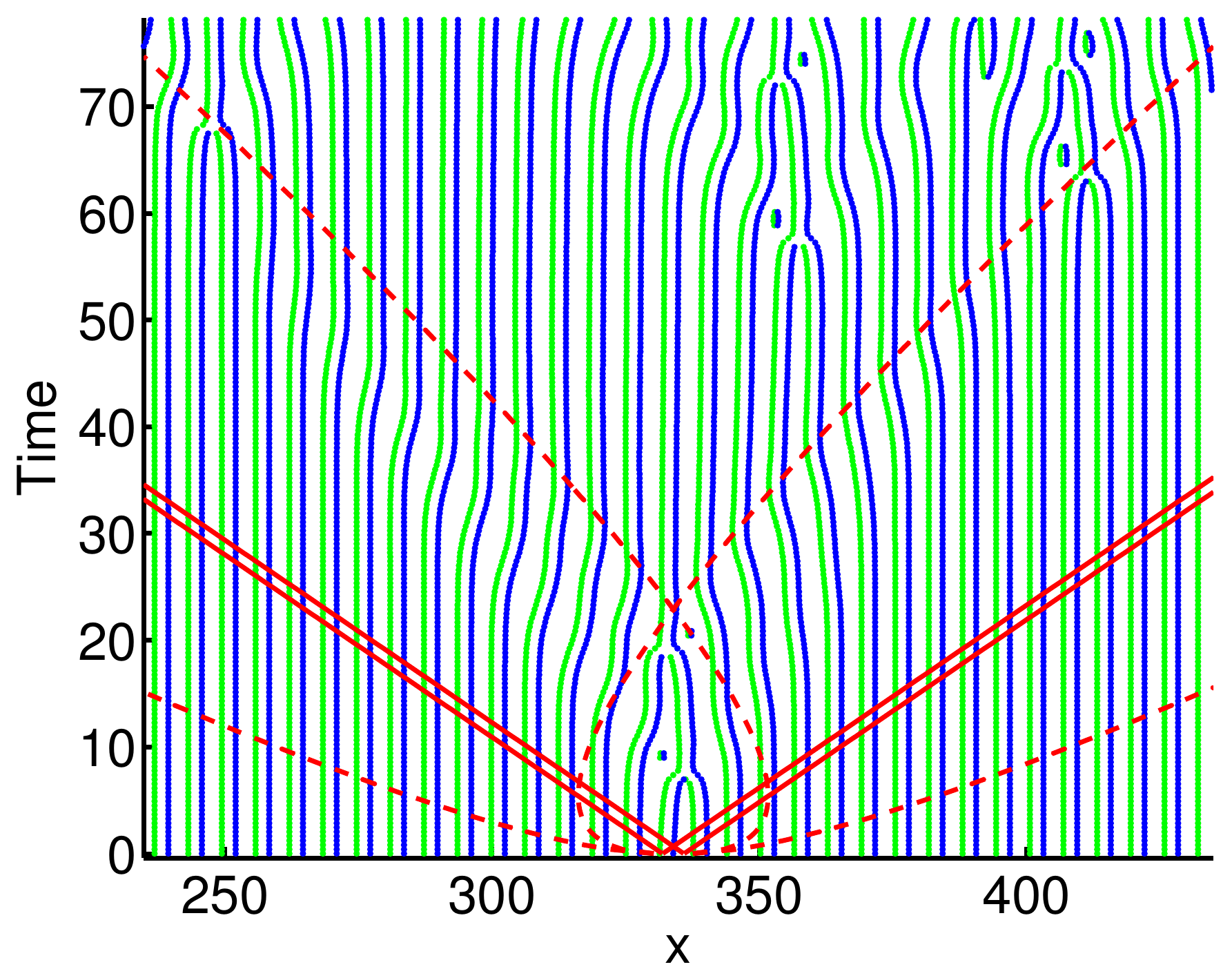}\quad(b)\includegraphics[scale=.3]{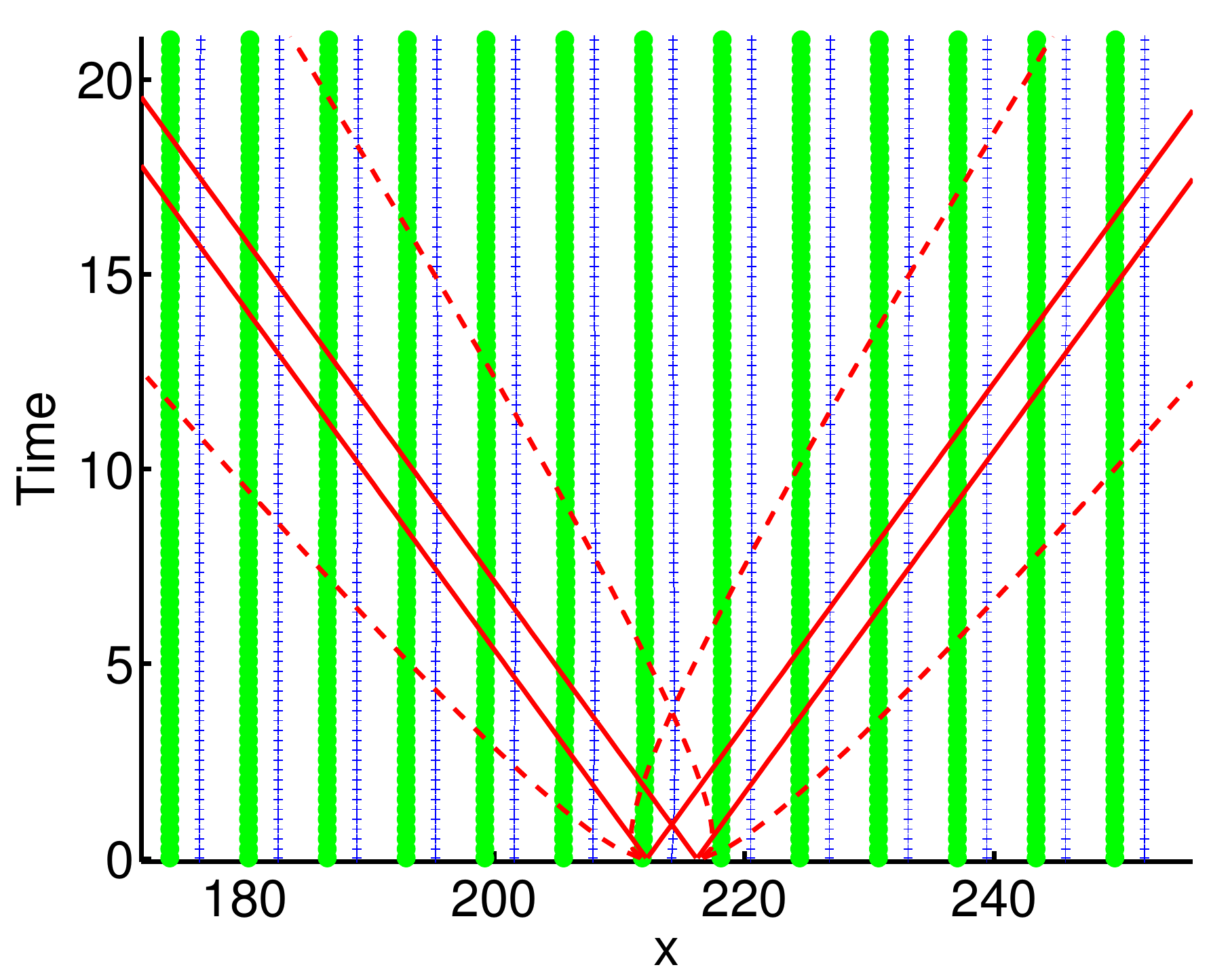}
      \quad(c)\includegraphics[scale=.3]{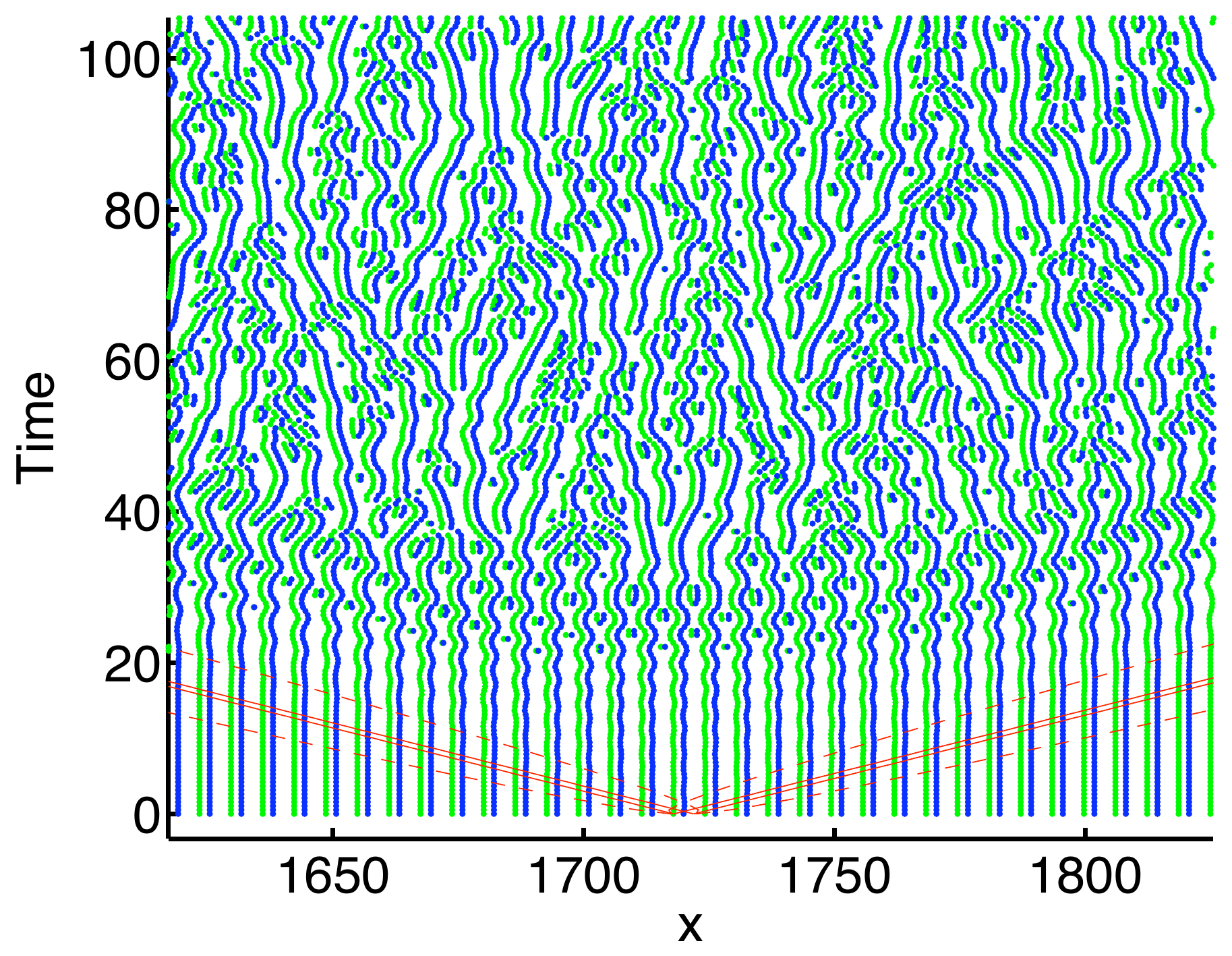}\quad(d)\includegraphics[scale=.3]{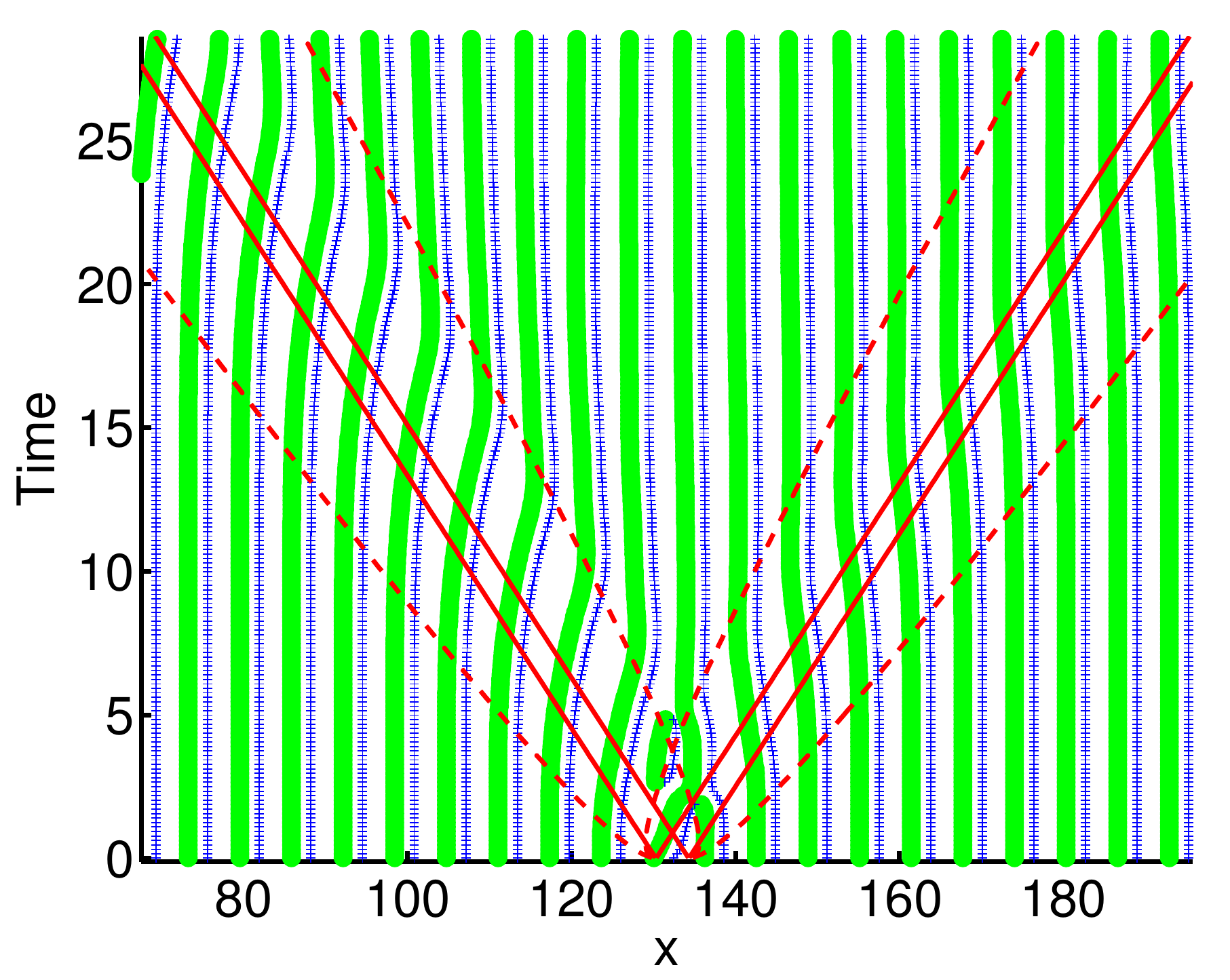}
\end{center}
\caption{Here, we present the results of four time-evolution studies for the case when $\eps=0$.
In Figures (a)-(c), we start with a small ``Gaussian" type perturbation
of the underlying 6.3-periodic wave (taken to be stationary by change of coordinates) and evolve the perturbation over time, with the vertical
lines corresponding to the ``peaks" and ``troughs" between the waves.  The wave-train in (a) corresponds to a wave below the lower stability
boundary 
($q=3$),
while the wave-train in (c) corresponds to a wave above the upper stability boundary ($q=10$).  The wave-train in (b) is spectrally
stable and corresponds to $q=5.5$.
The \qut{peaks} are plotted with thick green lines and the \qut{troughs} are plotted with thin blue lines. 
In (a)-(d), the solid and dotted red lines indicate respectively the first order and second order approximations of bounds on the characteristics.
They originate from a region enclosing three standard deviations from the mean of the perturbation.  Finally, Figure (d) is the same
as Figure (b) except for the initial perturbation is multiplied by a factor of 10.  It is interesting to note that
even when we subject the $q=5.5$ wave to a large perturbation, lying well outside our stability theory, 
we observe a similar time asymptotic stability with analogous phase description as we developed for weak perturbations.
}\label{f:tevol}
\end{figure}



We now  wish to connect the formal predictions of the above WKB analysis to the numerically observed time-evolution
of a perturbed periodic traveling wave solution of \eqref{e:cks1}.  Keeping with the above theme, we fix $\eps=0$ throughout
and consider waves with period $X=6.3$ which, after an appropriate change of coordinates, have been taken to be initially stationary.
In Figure \ref{f:tevol} we have fixed three periodic traveling wave solutions of \eqref{e:cks1}.  The wave in Figure \ref{f:tevol}(a)
corresponds to a periodic traveling wave below the lower stability boundary so that the associated linearized Whitham averaged system
formally describing long-time behavior fails to be locally well-posed (hyperbolic).  The resulting instability seems
considerably more drastic than that observed for the wave in Figure \ref{f:tevol}(c), corresponding to a periodic traveling
wave above the upper stability boundary.  In this latter case, the first-order linearized Whitham system is locally well-posed
but second-order diffusion coefficient, i.e. the coefficients $\beta_j$ in \eqref{e:lexp}, have real part with positive sign.
This results in a type of ``backward diffusion" where the amplitude of the local-wavenumber $\psi$ grows with time,
resulting in the forced visco-elastic behavior between the individual peaks and valleys in Figure \ref{f:tevol}(c).
For more details on the instabilities 
in Figure \ref{f:tevol}(c), see Figure \ref{f:tevolq10}.

\begin{figure}[htbp]
\begin{center}
(a)\includegraphics[scale=.25]{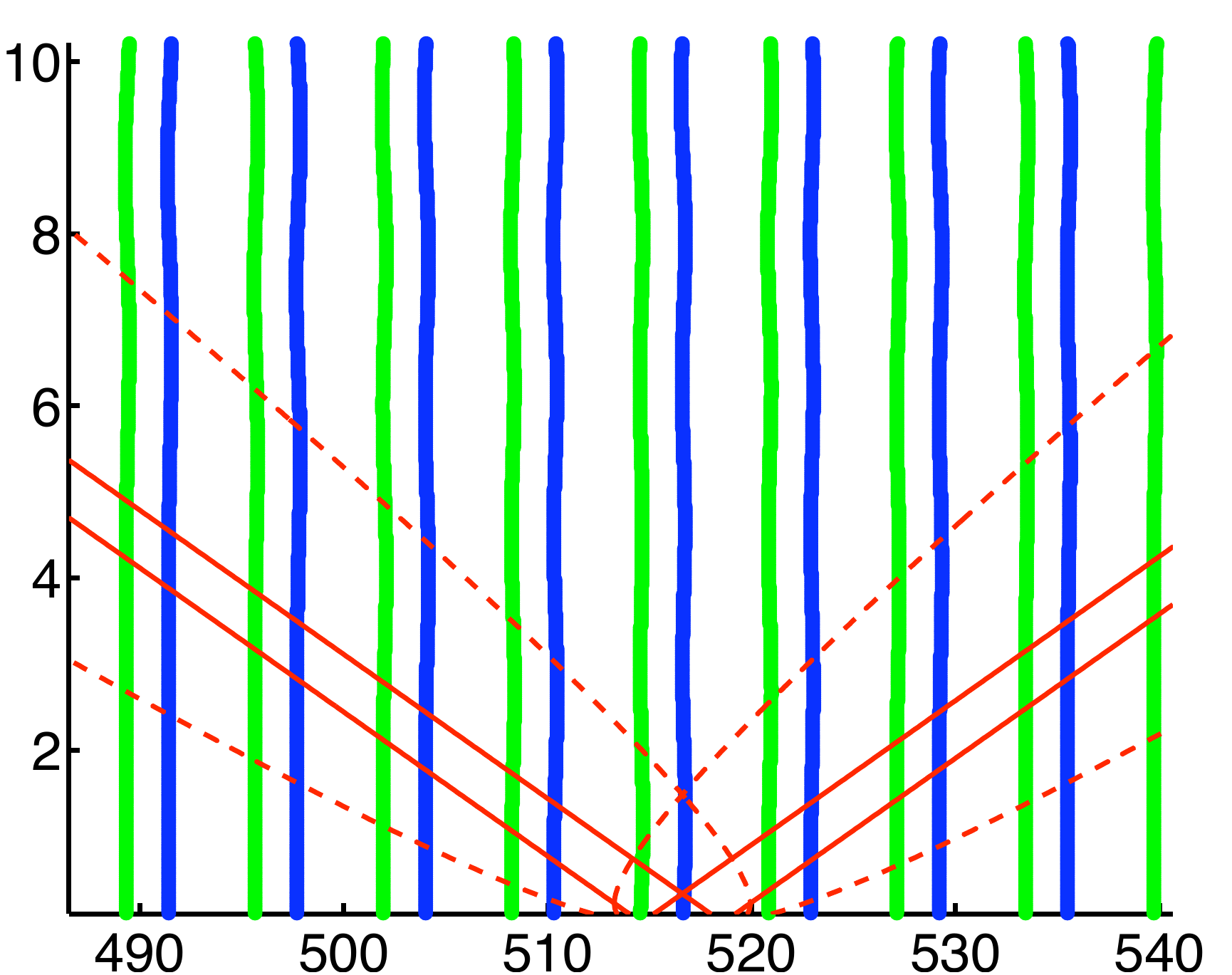}\quad(b)\includegraphics[scale=.25]{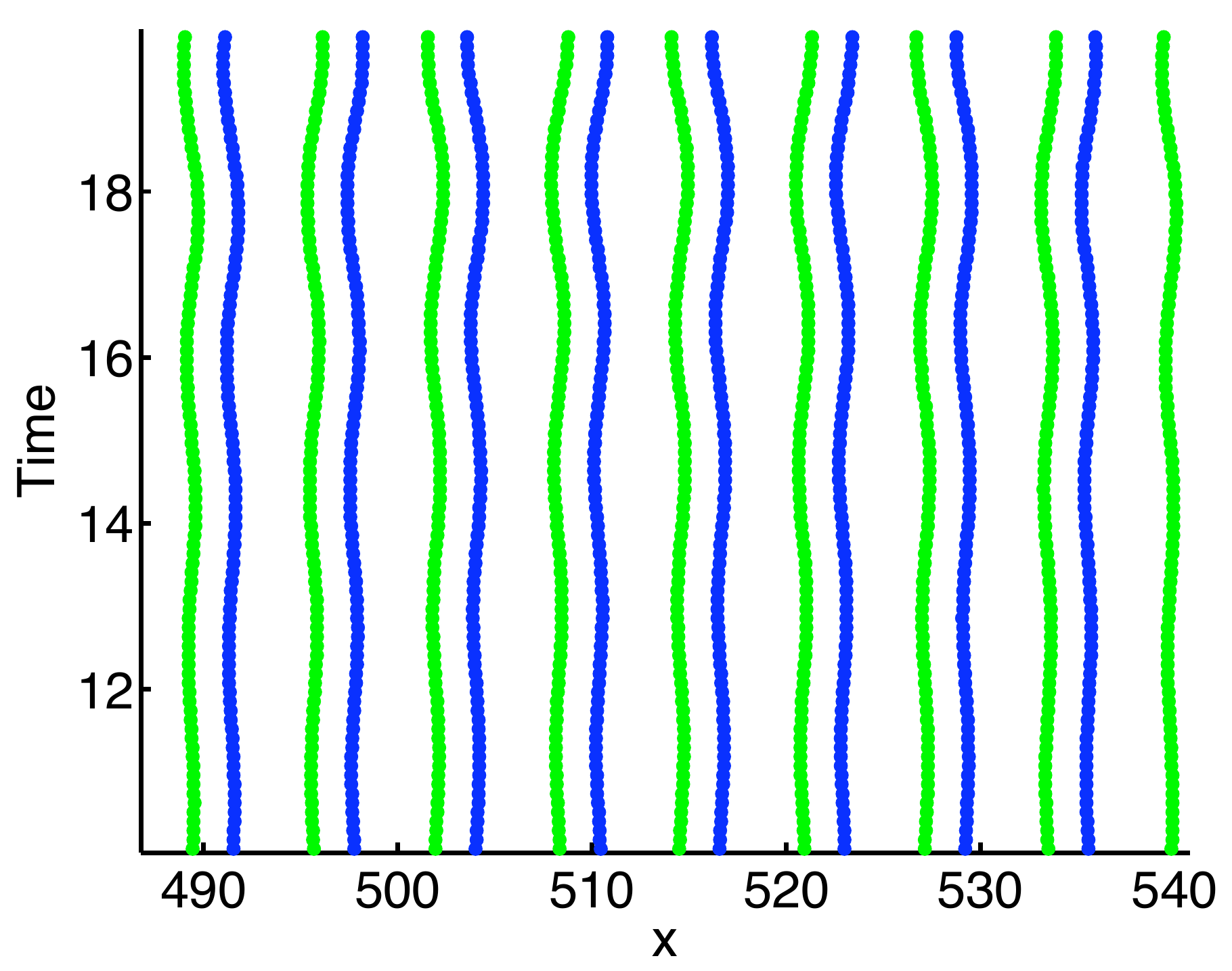}\\
(c)\includegraphics[scale=.25]{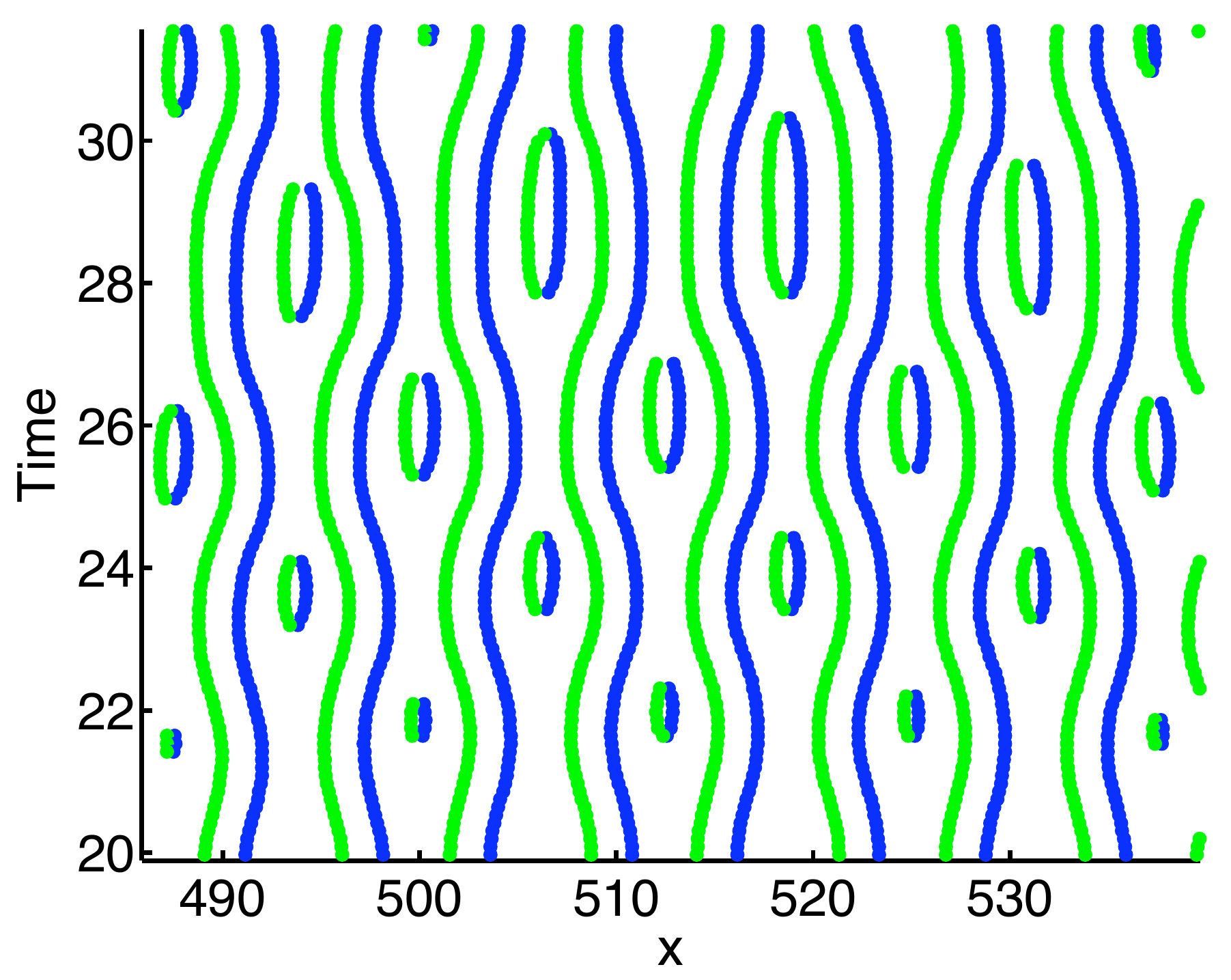}\quad (d)\includegraphics[scale=.25]{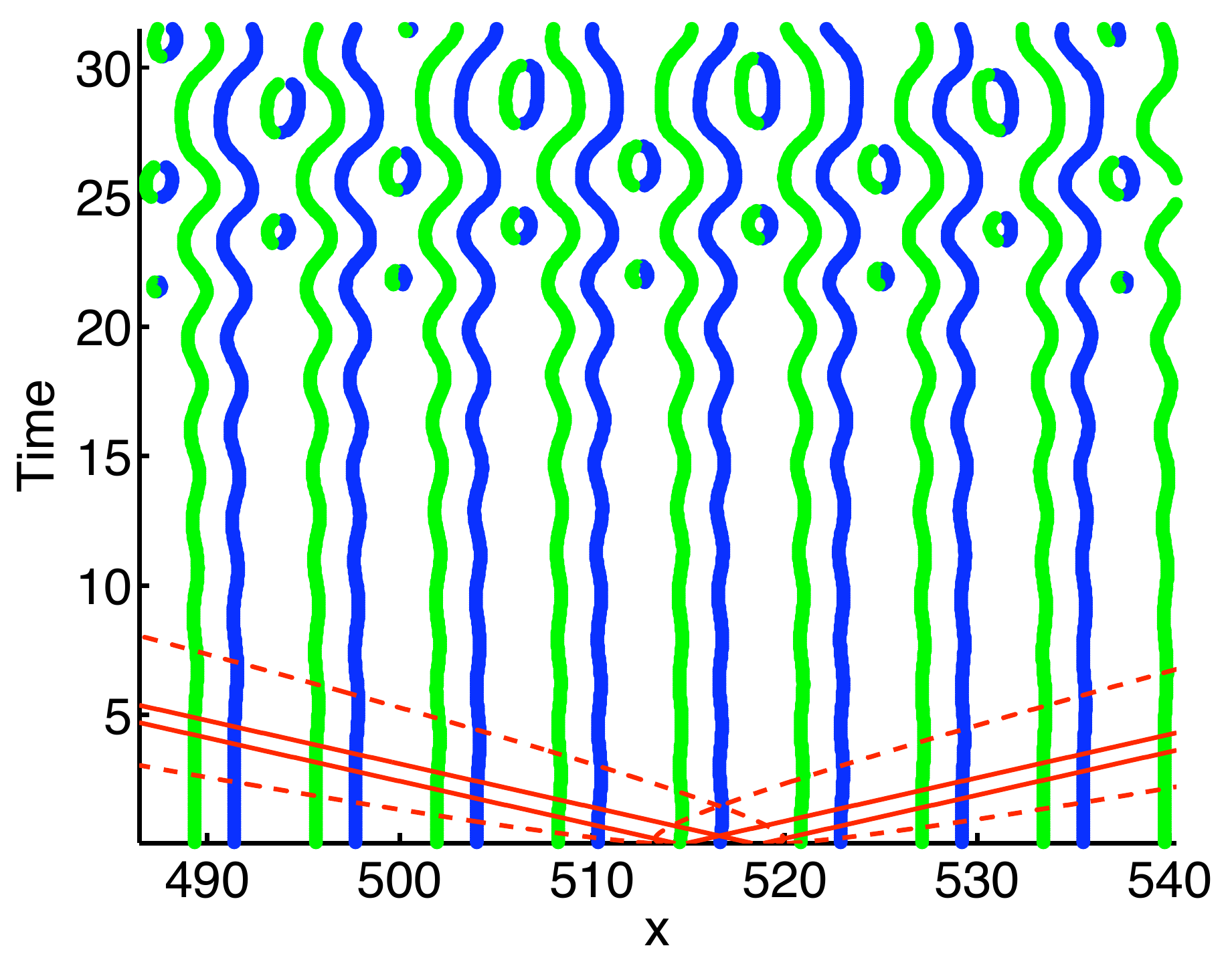}\\
(e)\includegraphics[scale=.25]{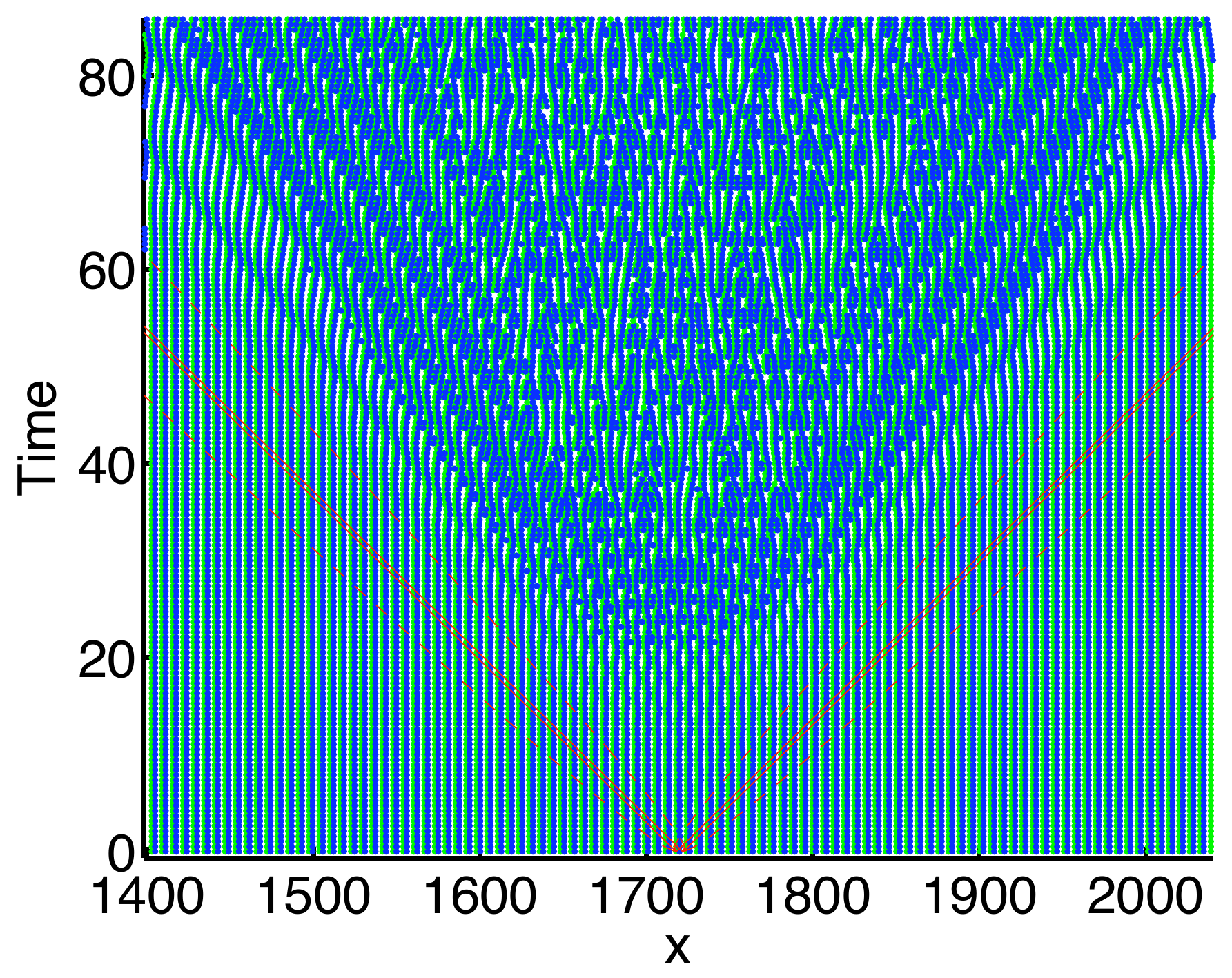}\\
\end{center}
\caption{Here, we continue the time evolution study from Figure \ref{f:tevol}(c) (corresponding to $\eps=0$, $X=6.3$, and $q=10$).
In particular, we highlight various aspects of the instability by zooming in on the evolution over different time intervals.
Furthermore, due to the hyperbolicity of the first order Whitham equation \eqref{e:wave} in this case 
%
the time scale on which the modulational instability (corresponding to Bloch frequencies $|\xi|\ll 1$) 
is observed is expected to be 
$\sim \eta^{-2}$, where $\eta$ is the difference between $X$ and the
stability boundary $X_*$; see Appendix \ref{s:behavior}.
}\label{f:tevolq10}
\end{figure}

Finally, the wave in Figure \ref{f:tevol}(b) is spectrally stable and the resulting time-evolution of the perturbed wave
at first sight appears 
to be of the form
\[
u(x,t)=u(x+\psi(x,t))
\]
where $\psi(\cdot,t)$ resembles the sum of well-separated Gaussian waves propagating in opposite directions.  More precisely,
it appears at first sight that our small initial perturbation simply divides its mass
into two traveling Gaussian packets which convect in opposite directions
and satisfies $\lim_{t\to\infty}u(x,t)=\bar{u}(x)$ for each $x\in\RM$.  In particular, this description would suggest that
\[
\|u(\cdot,t)-\bar{u}\|_{L^p(\RM)}\approx\|\psi(\cdot,t)\|_{L^p(\RM)}\lesssim (1+t)^{-\frac{1}{2}(1-1/p)},\quad p\geq 2.
\]
However, while this intuition seems reasonable from our numerical
experiments it is not correct: rather, the above ``convecting Gaussian" description applies to the \emph{wave number} $\psi_x$ \emph{and
not the phase} $\psi$.  As a result, $\psi$ should roughly be described by the \emph{integral} of Gaussian packets convecting in
opposite directions, i.e. a small amplitude compactly supported (for each $t>0$) sum of
error functions of algebraically growing mass with $\lim_{t\to\infty}\psi(x,t)\neq 0$.
In particular, we expect then that
\[
\|u(\cdot,t)-\bar{u}(\cdot-\psi(\cdot,t))\|_{L^p(\RM)}\approx \|\psi_x(\cdot,t)\|_{L^p(\RM)}\lesssim (1+t)^{-\frac{1}{2}(1-1/p)}
\]
for all $p\geq 2$, and that
\[
\|u(\cdot,t)-\bar{u}\|_{L^p(\RM)}\approx\|\psi(\cdot,t)\|_{L^p(\RM)}\lesssim 1.
\]
It is this observation that small localized perturbations of an underlying periodic traveling wave solution of \eqref{e:cks1} behave time-asymptotically as localized shifts of the original wave that drives our stability analysis in the next section, even in the general case when $\eps\neq 0$; see
Proposition \ref{GIdecay1}, Corollary \ref{e:Gbdsfinal}, and Lemma \ref{l:cancel1} below.

Finally, we remark that if one considers the more general case when $\eps\neq 0$ it follows by similar considerations that the critical
spectral curves $\lambda_j(\xi)$ agree to second-order with the dispersion relation of a second-order Whitham modulation
equation: see \cite{NR2} for details of this derivation.  However, the associated second order Whitham averaged system
consists of a coupled system which, while sharing many of the properties of that of
\eqref{e:whitham2}, is considerably more complicated to analyze directly.
Nevertheless, when considering for definiteness again the case $\eps=0.2$,
it is found that 6.3-periodic traveling wave solutions below the lower stability boundary ($\delta\approx 1.43$)
in Table \ref{stabstudy:0.2eps} correspond to a loss of hyperbolicity of the first-order Whitham system \eqref{e:whitham2}
while  those corresponding to waves above the upper stability boundary ($\delta\approx 1.719$) correspond to a
``backward damping" effect as described above.

\br 
\textup{As noted previously, the modulation $\psi$ in the above discussion is not the same modulation 
presented in Theorem \ref{main}, although the two are very closely related; see Remark \ref{nat:coordinates}.}
\er 

\br
\textup{
For rigorous justification of the Whitham equations at the nonlinear
level, see \cite{JNRZ1}.
}
\er

\section[Proof of Theorem 1.1]{Proof of Theorem \ref{main}}\label{s:proof}

In this section, we prove Theorem \ref{main}, showing 
that spectral stability of a given
periodic traveling wave solution of \eqref{e:gen} implies both linearized and nonlinear stability to
small localized perturbations. The proof
closely follows the analysis of \cite{JZ1}, concerning the analogous problem in the context of strictly parabolic
second-order conservation laws. The main difference lies in the fact that the linear
operator $L$ defined in \eqref{e:lin} is a fourth-order differential operator rather than second order. Nevertheless,
the principles of \cite{JZ1} readily extend to the present higher-order case with little modification.

\subsection{Spectral preparation}\label{prep}

Recall that in the 
statement 
of Theorem \ref{main}, we normalize $\bar X$ to 1.
%
Our assumptions (H1)-(H2) and (D3) imply that the generalized kernel of the operator $L_0$, defined on $L^2_{\rm per}([0,1])$, is of dimension $2$.  
A crucial part of our analysis of the linearized solution operator $e^{Lt}$ relies on understanding how 
this 
zero-eigenspace bifurcates from this neutral state.
As expected, the existence of a Jordan block at $\lambda=0$ for the operator $L_0$, guaranteed by hypothesis (H4), greatly complicates
matters as compared to the degenerate case when (H4) fails; see \cite{JZ3}.
We thus begin our linearized stability analysis with a careful study of the Bloch perturbation expansion of these critical
eigenvalues and associated eigen-projections near $\xi=0$. 

\begin{lemma}\label{blochfacts}
Assuming (H1), (H2), 
and (D3), there exist constants $\xi_0\in(0,\pi)$, $\varepsilon_0>0$ and two continuous curves, $j=1,2$, $\lambda_j: [-\xi_0,\xi_0]\to B(0,\varepsilon_0)$ such that, when $|\xi|\leq\xi_0$, 
\be\label{critical_spec}
\sigma(L_\xi)\cap B(0,\varepsilon_0)\ =\ \left\{\lambda_1(\xi),\lambda_2(\xi)\right\}.
\ee
Moreover these two critical curves are differentiable at $0$ and can be expanded as
\be\label{lin_group_velocity}
\displaystyle
\lambda_j(\xi)=\ -ia_j\xi\ +o(\xi),\quad j=1,2
\ee
as $\xi\to 0$.
Assuming also (H3), the curves $\lambda_j$ are analytic in a neighborhood of $\xi=0$.  
Thus, up to a possible change of $\xi_0$, there exist, for $0<|\xi|\leq\xi_0$, dual right and left eigenfunctions 
$\{q_j(\xi, \cdot)\}_{j=1,2}$ and $\{\tilde q_j(\xi, \cdot)\}_{j=1,2}$ of $L_\xi$ associated with $\lambda_j(\xi)$, of form
\begin{align*}
q_j(\xi,\cdot)&=(i\xi)^{-1}\beta_{j,1}(\xi)\,v_1(\xi,\cdot)+\beta_{j,2}(\xi)\,v_2(\xi,\cdot)\\[1em]
\tilde q_j(\xi,\cdot)&= i\xi\tilde\beta_{j,1}(\xi)\,\tilde v_1(\xi,\cdot)+\tilde\beta_{j,2}(\xi)\,\tilde v_2(\xi,\cdot)
\end{align*}
for $j=1,2$, where
\begin{itemize}
\item for $j=1,2$, the functions $v_j:[-\xi_0,\xi_0]\to L^2_{\rm per}([0,1])$ and $\tilde v_j:[-\xi_0,\xi_0]\to L^2_{\rm per}([0,1])$ 
are analytic functions such that, when $|\xi|\leq\xi_0$,  $\{v_j(\xi,\cdot)\}_{j=1,2}$ and 
$\{\tilde v_j(\xi,\cdot)\}_{j=1,2}$ are dual bases of the total
eigenspace of $L_\xi$ associated with spectrum $\sigma(L_\xi)\cap B(0,\varepsilon_0)$, chosen to satisfy
$$
v_1(0,\cdot)=\bar u_x,\quad{\rm and }\quad\tilde v_2(0,\cdot)\equiv 1;
$$
\item for $j=1,2$ and $k=1,2$, the functions $\beta_{j,k}:\ [-\xi_0,\xi_0]\to\C$ and $\tilde\beta_{j,k}:\ [-\xi_0,\xi_0]\to\C$ are analytic.
\end{itemize}
Finally, assuming in addition that (D2) holds, the spectral curves $\lambda_j$ can be expanded as
\begin{equation}\label{e:surfaces2}
\lambda_j(\xi)=-ia_j\xi-b_j\xi^2+\mathcal{O}(|\xi|^3),\quad j=1,2
\end{equation}
for some $a_j\in\RM$ and $b_j>0$.
\end{lemma}

\begin{proof}(following \cite{JZ1,JZN,NR2})
First, as already mentioned in Remark \ref{r:linphasecouple}, 
assumptions (H1), (H2) and (D3) ensures the possibility of a parametrization 
$\beta=(k,M)$ by wave number and mean 
We adopt such a parametrization. This provides $L_0\bar u_x=0$ and 
$L_0\d_cU(\cdot;\bar \beta)=-\d_Mc(\bar \beta)\bar u_x$ 
with $\langle \tilde v,\bar u_x\rangle_{L^2_{\rm per}([0,1])}=0$ 
and 
$\langle \tilde v, \d_MU(\cdot;\bar \beta)\rangle_{L^2_{\rm per}([0,1])}=1$, 
where $\tilde v\equiv1$.

Since $L_\xi$ has dense domain $H^4_{\rm per}([0,1])$ compactly embedded in $L^2_{\rm per}([0,1])$, its spectrum consists of isolated eigenvalues of finite multiplicity \cite{He}. As $0$ is separated from the rest of the spectrum of $L_0$, 
assumption (D3) and 
standard spectral theory for perturbations by relatively compact operators (see \cite{K}) yields constants
$\xi_0\in(0,\pi)$, $\varepsilon_0>0$ and \emph{continuous} functions $\lambda_1$, $\lambda_2$ such that, for $|\xi|<\xi_0$, $\sigma(L_\xi)\cap B(0,\varepsilon_0)\ =\ \left\{\lambda_1(\xi),\lambda_2(\xi)\right\}$. Moreover this also yields analytic dual right and left spectral projectors associated to spectrum in $B(0,\varepsilon_0)$. 
Analytic dual bases of the right and left eigenspaces may then be obtained by projecting dual bases for spectral spaces of the spectrum of $L_0$ in $B(0,\varepsilon_0)$. By the first paragraph of the proof and the conservation law structure of the governing equation, we may choose such bases in the form 
$\{\bar u_x,\d_MU(\cdot;\bar \beta)\}$ 
and $\{ *,\tilde v\}$ and obtain in this way the functions $\{v_j\}$ and $\{\tilde v_j\}$ of the lemma.

We have now reduced the 
infinite dimensional spectral perturbation problem for the operators $L_\xi$ to 
the spectral analysis of
$$
M_\xi\ = \left[\langle\tilde v_j(\xi,\cdot),L_\xi v_l(\xi,\cdot)\rangle_{L^2([0,1])}\right]_{j,l},
$$
a $2\times 2$ matrix perturbation problem. 
%
By direct calculation 
$M_0=\begin{pmatrix}0&-\d_Mc(\bar \beta)\\0&0\end{pmatrix}$. 
Below, however, we will scale $M_\xi$ to blow up at the double eigenvalue. To do so, we expand the operator $L_\xi$ as
\be \label{Lpert}
L_\xi=L_0 + i\xi L^{(1)}+(i\xi)^2L^{(2)}+(i\xi)^3L^{(3)}+(i\xi)^4L^{(4)}
\ee
and note, specifically, that
\ba \label{Ls}
L^{(1)}=-(\bar{u}-\bar c)-3\eps\partial_x^2-2\delta\partial_x-4\partial_x^3.
\ea
By either a direct calculation or by first scaling the parameterization then differentiating the profile
equation with respect to $k$ (see \cite{NR1,NR2}), we find that $\left<\tilde{v}_2(0,\cdot),L^{(1)}v_1(0,\cdot)\right>=0$.
%
Therefore, $M_\xi$ can be expanded as $\xi\to 0$ as
\[
M_\xi=\begin{pmatrix}0&-\d_Mc(\bar \beta)\\0 & 0\end{pmatrix}
+i\xi\begin{pmatrix} * & * \\0 & *\end{pmatrix}
+\mathcal{O}(|\xi|^2)
\]
so that, in particular, the scaling
\be\label{rescale}
\check M_\xi:= (i\xi)^{-1} S(\xi)M_\xi S(\xi)^{-1},\qquad
S(\xi):=\bp i\xi & 0\\0 & 1\\\ep,
\ee
preserves smoothness in $\xi$ at $\xi=0$.  Since the eigenvalues $m_j(\xi)$ of $\check M_\xi$ are
$(i\xi)^{-1}\lambda_j(\xi)$, 
their continuity implies the  differentiability of the functions $\lambda_j(\xi)$ at $\xi=0$.
%
Assuming, in addition to above, that assumption (H3) holds, it follows from the fact that $m_j(0)=-a_j$ that the eigenmodes
of $\check M_\xi$ are analytic in $\xi$ in a neighborhood of $\xi=0$.   Undoing the scaling finishes the proof, 
up to the observation that \eqref{e:surfaces2} follows from (D2) and $\overline{\lambda_j(\xi)}=\lambda_j(-\xi)$.
\end{proof}

\subsection{Linearized bounds}\label{s:linbds}

We begin our stability analysis by deriving decay rates for the linearized solution operator $e^{Lt}$ of the linearized
equation \eqref{e:lin}.  Recalling the inverse Bloch transform representation \eqref{IBFT} of the linearized solution operator, 
we first introduce, as in the proof of Lemma \ref{blochfacts}, for each $\xi\in(-\xi_0,\xi_0)$ the spectral projection $P(\xi)$, analytic in $\xi$, onto the total eigenspace associated with spectrum $\sigma(L_\xi)\cap B(0,\varepsilon_0)=\left\{\lambda_1(\xi),\lambda_2(\xi)\right\}$ of the Bloch operator $L_\xi$. We also choose a smooth cutoff function $\phi:[-\pi,\pi)\to[0,1]$ such that
\[
\phi(\xi)=\left\{
            \begin{array}{ll}
              1, & \textrm{ if }|\xi|\leq \xi_0/2 \\
              0, & \textrm{ if }|\xi|\geq \xi_0
            \end{array}
          \right.
\]
and split the solution operator $S(t):=e^{Lt}$ into its 
low Floquet-number critical part
%
\begin{equation}\label{solnlf}
S^I(t)g(x):=\int_{-\pi}^\pi e^{i\xi x}\phi(\xi)[P(\xi)e^{L_{\xi}t}\check{g}(\xi,\cdot)](x)d\xi
\end{equation}
and 
exponentially-stable part
%
\begin{equation}\label{solnhf}
S^{II}(t)g(x):=\int_{-\pi}^\pi e^{i\xi x}[\left(1-\phi(\xi)P(\xi)\right)e^{L_{\xi}t}\check{g}(\xi,\cdot)](x)d\xi.
\end{equation}

As the 
analysis of the critical part is considerably more delicate, we being by deriving
%
$L^p$ 
bounds on $S^{II}(t)$.
By standard sectorial bounds \cite{Pa,He}, 
the fact that $L^2$ and $H^r$ spectra coincide 
and the spectral separation of the $\lambda_j(\xi)$ from the remaining spectrum
of $L_{\xi}$ we 
have bounds
%
\begin{equation}\label{sgbds}
\begin{aligned}
\|L_{\xi}^me^{L_{\xi}t}\left(1-\phi(\xi)P(\xi)\right)g\|_{H_{\rm per}^r([0,1]),\xi}&\leq Ct^{-m}e^{-\theta t}\|g\|_{H_{\rm per}^r([0, 1]),\xi}
\end{aligned}
\end{equation}
for some constants $\theta,C>0$, with $\|g\|_{H_{\rm per}^r([0,1]),\xi}^2=\sum_{j=0}^r\|(\d_x+i\xi)^jg\|_{L_{\rm per}^2([0,1])}^2$.  Using the fact that $L_{\xi}$ is 
a relatively compact perturbation of the fourth order differential operator $(-1-(\d_x+i\xi)^4)$, 
in conjunction with \eqref{iso}, we immediately have 
the $H^r$ bounds of the following Proposition.

\begin{prop}\label{p:hfbounds}
Under assumptions (H1)--(H4) and (D1)--(D3), there 
exist 
constants $C,\theta>0$ such that for all $2\leq p\leq\infty$,
$0\leq 4l_1+l_2=l_3\leq K+1$, $0\leq 4m_1+m_2+m_3\leq K$, $K$ as in (H1), $r=0,1$ and $t>0$
\begin{align*}
\left\|\d_t^{l_1}\d_x^{l_2} S^{II}(t)\d_x^{l_3}g\right\|_{H^r(\RM)}&\leq Ct^{-\frac{4l_1+l_2+l_3}{4}}e^{-\theta t}\|g\|_{H^r(\RM)},\\
\left\|\d_t^{m_1}\d_x^{m_2} S^{II}(t)\d_x^{m_3}g\right\|_{L^p(\RM)}&\leq Ct^{-\frac{1}{4}\left(\frac{1}{2}-\frac{1}{p}\right)-\frac{4m_1+m_2+m_3}{4}}e^{-\theta t}\|g\|_{L^2(\RM)},\\
\left\|S^{II}(t)g\right\|_{L^p(\RM)}&\leq Ce^{-\theta t}\|g\|_{H^1(\RM)}.
\end{align*}
\end{prop}

\begin{proof}
By the above discussion, all that is left is to explain how to verify the stated $L^p$ bounds.
These follow from $L^2$ and $H^1$ bounds and the Sobolev embedding inequality 
$\|g\|_{L^p(\RM)}\leq C\|g\|_{L^2(\RM)}^{1-\left(1/2-1/p\right)}\|\d_xg\|_{L^2(\RM)}^{1/2-1/p}$.
%
\end{proof}

Next, we analyze the 
critical part 
of the solution operator $S(t)$.
For this purpose, it is convenient to introduce the (critical)
Green kernel
\[
G^I(x,t;y):=S^I(t)\delta_y(x)
\]
associated with $S^I$ and
\[
\left[G^I_{\xi}(\cdot,t;y)\right](x):=\phi(\xi)P(\xi)e^{L_{\xi}t}\left[\delta_y\right](x)
\]
the corresponding integral kernel appearing within the Bloch-Fourier representation of $G^I$,
where $\left[\cdot\right]$ denotes the 
1-periodic 
extension of the given function onto the whole real line.
We first prove the following lemma yielding a spectral representation of $G^I$ incorporating
the results of Lemma \ref{blochfacts}.

\begin{lemma}\label{l:G1rep}
Under the hypothesis (H1)--(H4) and (D1)--(D3), we have
\begin{equation}\label{e:G1rep}
\begin{aligned}
\left[G_\xi^I(\cdot,t;y)\right](x)&=\phi(\xi)\sum_{j=1}^2e^{\lambda_j(\xi)t}q_j(\xi,x)\tilde{q}_j(\xi,y)^*,\\
G^I(x,t;y)&=\int_{-\pi}^\pi e^{i\xi(x-y)}\left[G_\xi^I(\cdot,t;y)\right](x)d\xi\\
&=\int_{-\pi}^\pi e^{i\xi(x-y)}\phi(\xi)\sum_{j=1}^2e^{\lambda_j(\xi)t}q_j(\xi,x)\tilde{q}_j(\xi,y)^*d\xi,
\end{aligned}
\end{equation}
where $*$ denotes the matrix adjoint, or complex conjugate transpose, and $q_j(\xi,\cdot)$ and $\tilde{q}_j(\xi,\cdot)$
are right and left eigenfunctions of $L_{\xi}$ associated with the eigenvalues $\lambda_j(\xi)$ defined in  
\eqref{critical_spec}, 
normalized so that 
$\left<\tilde{q}_j(\xi,\cdot),q_j(\xi,\cdot)\right>_{L_{\rm per}^2([0,1])}=1$.
\end{lemma}

\begin{proof}
The first equality follows by the spectral
decomposition of $e^{L_\xi t}$, and the spectral
description of Lemma \ref{blochfacts}.   The second equality follows by the inverse Bloch transform formula
\eqref{solnlf}
and Fourier transform manipulations using the fact that both discrete and continuous
transforms of the 
centered 
$\delta$-function are 
constant equal to $(2\pi)^{-1}$ to get
$$
\check{\delta_y}(\xi,x)\ =\ \frac{1}{2\pi}e^{-i\xi y}\sum_{l\in\ZM}e^{2i\pi l(x-y)}
\ =\ e^{-i\xi y}[\delta_y](x).
$$
The third quality now follows by substitution; see 
\cite{OZ4} 
for further details.
\end{proof}

Continuing, we point out that it seems unlikely that the low-frequency Green function $G^I$ will satisfy
$L^p\to L^p$ bounds which are suitable for our purposes.  To see this, notice
from Lemma \ref{l:G1rep} and Lemma \ref{blochfacts} that we have the representation
\[
G^I(x,t;y)=\int_{-\pi}^\pi e^{i\xi(x-y)}\phi(\xi)\sum_{j}^2e^{\lambda_j(\xi)t}
\Big((i\xi)^{-1}\beta_{j,1}(\xi)\tilde{\beta}_{j,2}(\xi)v_1(\xi,x)\tilde{v}_2(\xi,y)^*
+\mathcal{O}(1)\Big)d\xi
\]
of the 
critical 
Green kernel. From assumption (D2) 
%
we expect, for example,
\[
\left\|G^I(\cdot,t;y)\right\|_{L^\infty(\RM)}\approx \left\|\xi\mapsto|\xi|^{-1}e^{-\theta|\xi|^2t}\phi(\xi)\right\|_{L^1([-\pi,\pi))},
\]
where the right hand side is interpreted in the principal value sense, which is merely bounded and hence
does not decay in time.  In order to compensate for this lack of decay arising from the Jordan block associated with the translation
mode at $\xi=0$, we separate out the bounded translation mode from the faster-decaying ``good" part of the Green kernel.
To this end, notice that by defining the function
\begin{equation}\label{e1def}
\tilde{e}(x,t;y):=\int_{-\pi}^\pi
e^{i\xi(x-y)}\phi(\xi)\sum_{j=1}^2e^{\lambda_j(\xi)t}(i\xi)^{-1}\beta_{j,1}(\xi)\tilde{q}_j(\xi,y)^*d\xi
\end{equation}
we have
\begin{equation}\label{e:G1decomp}
\begin{aligned}
G^I(x,t;y)&-\bar u_x(x)\tilde{e}(x,t;y)\\
&=\int_{-\pi}^\pi e^{i\xi(x-y)}\phi(\xi)\sum_{j=1}^2e^{\lambda_j(\xi)t}
\beta_{j,2}(\xi)v_2(\xi,x)\tilde{q}_j(\xi,y)^*d\xi\\
&+\int_{-\pi}^\pi e^{i\xi(x-y)}\phi(\xi)\sum_{j=1}^2e^{\lambda_j(\xi)t}\beta_{j,1}(\xi)
\frac{v_1(\xi,x)-v_1(0,x)}{i\xi}\tilde{q}_j(\xi,y)^*d\xi
\end{aligned}
\end{equation}
Now since all quantities involved in \eqref{e:G1decomp} are
$C^1$
in $\xi$, from (D2) we deduce
\begin{equation}\label{e:Gbdsinfty}
\left\|G^I(\cdot,t;y)-\bar{u}_x(\cdot)\tilde{e}(\cdot,t;y)\right\|_{L^\infty(\RM)}
\leq C\left\|\xi\mapsto e^{-\theta|\xi|^2t}\phi(\xi)\right\|_{L^1([-\pi,\pi))}\leq C(1+t)^{-1/2}
\end{equation}
yielding algebraic decay of 
$G^I-\bar{u}_x\tilde{e}$ 
in time.  The more general effect of this regularization is the content of the following
proposition.

\begin{prop}\label{GIdecay1}
Under the assumptions (H1)--(H4) and (D1)--(D3), the 
critical 
Green kernel $G^I(x,t;y)$ of \eqref{e:lin} may be decomposed
as
\[
G^{I}(x,t;y)=\bar{u}_x(x)\tilde{e}(x,t;y)+\widetilde{G}^I(x,t;y)
\]
where for all 
$t\geq0$, 
$1\leq q\leq2\leq p\leq\infty$, 
and $1\leq r\leq 4$ the residual $\widetilde{G}^I(x,t;y)$ satisfies
\begin{equation}\label{e:G1bds}
\begin{aligned}
\left\|\int_\RM\widetilde{G}^I(\cdot,t;y)g(y)dy\right\|_{L^p(\RM)}&\leq C\left(1+t\right)^{-\frac{1}{2}\left(\frac1q-\frac1p\right)}\|g\|_{L^q(\RM)},\\
\left\|\int_\RM\partial_y^r\widetilde{G}^I(\cdot,t;y)g(y)dy\right\|_{L^p(\RM)}&\leq C\left(1+t\right)^{-\frac{1}{2}\left(\frac1q-\frac1p\right)-\frac{1}{2}}\|g\|_{L^q(\RM)},\\
\left\|\int_\RM\partial_t\widetilde{G}^I(\cdot,t;y)g(y)dy\right\|_{L^p(\RM)}&\leq C\left(1+t\right)^{-\frac{1}{2}\left(\frac1q-\frac1p\right)-\frac{1}{2}}\|g\|_{L^q(\RM)}.
\end{aligned}
\end{equation}
Furthermore, for all 
$t\geq0$, 
$1\leq q\leq2\leq p\leq\infty$, 
$0\leq j,l$, $j+l\leq K$, 
and $1\leq r\leq 4$ we have
\begin{equation}\label{e:ebds1}
\left\|\int_\RM\partial_x^j\partial_t^l\partial_y^r \tilde{e}(\cdot,t;y)g(y)dy\right\|_{L^p(\RM)}\leq C(1+t)^{-\frac{1}{2}\left(\frac1q-\frac1p\right)-\frac{(j+l)}{2}}\|g\|_{L^q(\RM)}
\end{equation}
and, if in addition $j+l\geq 1$, or 
($j=l=0$, $p=\infty$ and $q=1$),
%
\begin{equation}\label{e:ebds2}
\left\|\int_\RM\partial_x^j\partial_t^l\tilde{e}(\cdot,t;y)g(y)dy\right\|_{L^p(\RM)}\leq C(1+t)^{\frac{1}{2}-\frac{1}{2}\left(\frac1q-\frac1p\right)-\frac{(j+l)}{2}}\|g\|_{L^q(\RM)}.
\end{equation}
\end{prop}

\begin{proof}
To begin, let 
$\widetilde{G}^I(x,t;y):=G^I(x,t;y)-\bar{u}_x(x)\tilde{e}(x,t;y)$ 
where the function $\tilde{e}$ is defined in \eqref{e1def}.
By interpolation, it is sufficient to consider only the cases $q=1$ and $q=2$.

{\it (i) Case $q=1$.} To prove bounds for $q=1$ we rely on Triangle Inequality
$$
\left\|\int_\RM F(\cdot,t;y)g(y)dy\right\|_{L^p(\RM)}\ 
\leq\ \|g\|_{L^1(\RM)}\sup_{y\in\RM}\|F(\cdot,t;y)\|_{L^p(\RM)}.
$$
Now, to generalize \eqref{e:Gbdsinfty} and prove the first part of \eqref{e:G1bds} in the case $q=1$, we use \eqref{hy} to get for any $(t,y)$
$$
\left\|\tilde G^I(\cdot,t;y)\right\|_{L^p(\RM)}
\leq C\left\|\xi\mapsto e^{-\theta|\xi|^2t}\phi(\xi)\right\|_{L^{p'}([-\pi,\pi))}\leq C(1+t)^{-\frac12(1-1/p)}
$$
with $p'$ the H\"older conjugate of $p$, $1/p+1/p'=1$. The second part of \eqref{e:G1bds} comes 
in a similar way by observing that for $j=1,2$ and $r\geq 1$, the fact that $\tilde q_j(0,y)=\tilde\beta_{j,2}(0)$ implies 
\be\label{constant_q}
\d_y^r \tilde q_j(\xi,y)\ =\ \d_y^r \tilde q_j(\xi,y)-\d_y^r \tilde q_j(0,y)=\mathcal{O}(\xi),
\ee
and thus
$$
\left\|\d_y^r\tilde G^I(\cdot,t;y)\right\|_{L^p(\RM)}
\leq C\left\|\xi\mapsto |\xi|e^{-\theta|\xi|^2t}\phi(\xi)\right\|_{L^{p'}([-\pi,\pi))}\leq C(1+t)^{-\frac12(1-1/p)-\frac12}.
$$
Likewise the last part of \eqref{e:G1bds} stems from the fact that the time derivative brings down factors $\lambda_j(\xi)=\mathcal{O}(\xi)$.

Using again \eqref{constant_q}, we obtain for $r\geq1$
$$
\d_x^j\d_t^l\d_y^r \tilde{e}(x,t;y)=\int_{-\pi}^\pi
e^{i\xi(x-y)}\phi(\xi)\sum_{j'=1}^2e^{\lambda_{j'}(\xi)t}(i\xi)^{j}\lambda_j(\xi)^l\beta_{j',1}(\xi)\frac{\d_y^r\tilde{q}_{j'}(\xi,y)^*-\d_y^r\tilde{q}_{j'}(0,y)^*}{i\xi}d\xi
$$
and another use of \eqref{hy} provides
$$
\left\|\d_x^j\d_t^l\d_y^r \tilde{e}(\cdot,t;y)\right\|_{L^p(\RM)}
\leq C\left\|\xi\mapsto |\xi|^{j+l}e^{-\theta|\xi|^2t}\phi(\xi)\right\|_{L^{p'}([-\pi,\pi))}\leq C(1+t)^{-\frac12(1-1/p)-\frac12(j+l)}.
$$
This proves \eqref{e:ebds1} in the case $q=1$. Inequality 
\eqref{e:ebds2}
follows similarly when $j+l\geq1$.

Finally, we must obtain 
\eqref{e:ebds2} in the critical case $j=l=0$, $p=\infty$ and $q=1$.
Recalling \eqref{e:surfaces2} we introduce $\check\lambda(\xi):=-ia_j\xi-b_j\xi^2$ and expand
\ba\label{e:princpart}
\tilde{e}(x,t;y)&=\int_{-\pi}^\pi
e^{i\xi(x-y)}\phi(\xi)\sum_{j=1}^2e^{\lambda_j(\xi)t}\frac{\beta_{j,1}(\xi)\tilde{q}_j(\xi,y)^*-\beta_{j,1}(0)\tilde{q}_{j}(0,y)^*}{i\xi}d\xi\\
&+\int_{-\pi}^\pi e^{i\xi(x-y)}\phi(\xi)\sum_{j=1}^2e^{\check\lambda_j(\xi)t}\frac{e^{(\lambda_j(\xi)-\check\lambda_j(\xi))t}-1}{i\xi}\beta_{j,1}(0)\tilde\beta_{j,2}(0)d\xi\\
&+\int_{-\pi}^\pi e^{i\xi(x-y)}\sum_{j=1}^2e^{\check\lambda_j(\xi)t}\frac{\phi(\xi)-1}{i\xi}\beta_{j,1}(0)\tilde\beta_{j,2}(0)d\xi\\
&-\int_{\RM\setminus[-\pi,\pi]} e^{i\xi(x-y)}\sum_{j=1}^2e^{\check\lambda_j(\xi)t}\frac{1}{i\xi}\beta_{j,1}(0)\tilde\beta_{j,2}(0)d\xi\\
&+\textrm{P.V.}\int_{\RM} e^{i\xi(x-y)}\sum_{j=1}^2e^{\check\lambda_j(\xi)t}\frac{1}{i\xi}\beta_{j,1}(0)\tilde\beta_{j,2}(0)d\xi.
\ea
All but the last term in \eqref{e:princpart} may be bounded using Haussdorff-Young estimates, 
either using the classical one or \eqref{hy}. To complete the proof then, we must derive an appropriate $L^\infty$ 
bound on the last integral in \eqref{e:princpart}. Since
\[
\frac{1}{2\pi }\int_{\RM}
e^{i\xi \cdot (x-y)}
e^{\check\lambda_j(\xi)t} d\xi
=\frac{e^{-(x-y-a_jt)^2/4b_jt}}{\sqrt{4\pi b_jt}},\quad j=1,2,
\]
the final principal value integral in \eqref{e:princpart} is recognized to be
\[
2\pi\sum_{j=1}^2\beta_{j,1}(0)\tilde\beta_{j,2}(0)\,\errfn\left(\frac{x-y-a_j t}{\sqrt{4b_jt}}\right),
\]
hence is bounded in $L^\infty$ as claimed.

{\it (ii) Case $q=2$.} To prove bounds with $q=2$ we directly apply \eqref{hy} to Bloch formulations. For instance expand
$$
\begin{aligned}
\int_\RM \tilde G^I(x,t;y)&g(y)dy=\int_{-\pi}^\pi e^{i\xi x}\phi(\xi)\sum_{j=1}^2e^{\lambda_j(\xi)t}
\beta_{j,2}(\xi)v_2(\xi,x)\langle\tilde{q}_j(\xi,\cdot),\check{g}(\xi,\cdot)\rangle_{L^2_{\rm per}([0,1])} d\xi\\
&+\int_{-\pi}^\pi e^{i\xi x}\phi(\xi)\sum_{j=1}^2e^{\lambda_j(\xi)t}\beta_{j,1}(\xi)
\frac{v_1(\xi,x)-v_1(0,x)}{i\xi}\langle\tilde{q}_j(\xi,\cdot),\check{g}(\xi,\cdot)\rangle_{L^2_{\rm per}([0,1])}d\xi.
\end{aligned}
$$
The generalized Haussdorff-Young estimate \eqref{hy} thus yields
$$
\begin{aligned}
\left\|\int_\RM\tilde{G}^I(\cdot,t;y)g(y)dy\right\|_{L^p(\RM)}
&\leq C\left\|\xi\mapsto e^{-\theta|\xi|^2t}\|\check{g}(\xi,\cdot)\|_{L^2_{\rm per}([0,1])}\right\|_{L^{p'}([-\pi,\pi])}\\
&\leq C\left\|\xi\mapsto e^{-\theta|\xi|^2t}\right\|_{L^{p''}([-\pi,\pi])}\|\check{g}\|_{L^2([-\pi,\pi),L^2_{\rm per}([0,1]))}\\
&\leq C(1+t)^{-\frac12\left(\frac12-\frac1p\right)}\|g\|_{L^2(\RM)}
\end{aligned}
$$
where $p'$ is such that $1/p+1/p'=1$, and $p''$ such that $1/p'=1/2+1/p''$. 
The other bounds on $\tilde G^I$ follows in the same way, observing 
that space derivatives $\d_y$ allow the use of \eqref{constant_q} and thus provide an extra $\mathcal{O}(\xi)$ factor
in the appropriate integrals.  The bounds on $\tilde e$ can be obtained in an analogous way, completing the proof.
\end{proof}

Finally, we combine the above various 
exponentially-stable and critical 
bounds to obtain decay estimates on the Green function
\[
G(x,t;y)=S(t)\delta_y(x)
\]
associated with the full solution operator $S(t)=e^{Lt}$.
Prior to that, 
we let 
$\chi:[0,\infty)\to[0,1]$ be a smooth real valued cutoff function
such that
\[
\chi(t)=\left\{
         \begin{array}{ll}
           0, & \textrm{ if }0\leq t\leq 1 \\
           1, & \textrm{ if }t\geq 2
         \end{array}
       \right.
\]
and define
\begin{equation}\label{enew}
e(x,t;y):=\chi(t)\tilde{e}(x,t;y);
\end{equation}
the purpose of the time 
cutoff 
function $\chi$ will be made clear 
%
in Remark \ref{time-cutoff} below.
The following result immediately follows from 
Lemma \ref{l:G1rep} and Proposition \ref{GIdecay1}.

\begin{corr}\label{e:Gbdsfinal}
Under assumptions (H1)--(H4) and (D1)--(D3), the Green function $G(x,t;y)$ of \eqref{e:lin} decomposes
as
\[
G(x,t;y)=\bar{u}_x(x)e(x,t;y)+\widetilde{G}(x,t;y)
\]
where, for some $C,\theta>0$ and all $t>0$, $2\leq p\leq\infty$ and $1\leq r\leq 4$,
%
\begin{align}
\left\|\int_\RM\partial_y^r\widetilde{G}(\cdot,t;y)g(y)dy\right\|_{L^p(\RM)}&\leq 
\min \begin{cases}
Ce^{-\theta t}\|\d_x^rg\|_{H^1(\RM)}
+C\left(1+t\right)^{-\frac{1}{2}\left(\frac{1}{2}-\frac{1}{p}\right)-\frac{1}{2}}\|g\|_{L^2(\RM)}\\
Ct^{-\frac{1}{4}\left(\frac{1}{2}-\frac{1}{p}\right)-\frac{r}{4}}
\left(1+t\right)^{-\frac{1}{4}\left(\frac{3}{2}-\frac{1}{p}\right)-\frac12+\frac{r}{4}}\|g\|_{L^1(\RM)\cap L^2(\RM)}
\end{cases}
\label{finalGbdyder}\\
\left\|\int_\RM\widetilde{G}(\cdot,t;y)g(y)dy\right\|_{L^p(\RM)}&\leq 
C\left(1+t\right)^{-\frac{1}{2}\left(1-\frac{1}{p}\right)}\|g\|_{L^1(\RM)\cap H^1(\RM)}
\label{finalGbd1}
\end{align}
Furthermore,  $e(x,t;y)\equiv 0$ for $0\leq t\leq 1$ and
there exists a constant $C>0$ such that for all 
$t\geq0$, $1\leq q\leq 2\leq p\leq\infty$, $0\leq j,l$, $j+l\leq K$, and $1\leq r\leq 4$
we have
\begin{equation}\label{finalebds1}
\begin{aligned}
\left\|\int_\RM\partial_x^j\partial_t^l\partial_y^r e(\cdot,t;y)g(y)dy\right\|_{L^p(\RM)}&\leq C\left(1+t\right)^{-\frac{1}{2}\left(\frac{1}{q}-\frac{1}{p}\right)
              -\frac{(j+k)}{2}}\|g\|_{L^q(\RM)}.
\end{aligned}
\end{equation}
and, if in addition $j+l\geq 1$, or ($j=l=0$, $p=\infty$ and $q=1$),
\begin{equation}\label{finalebds2}
\begin{aligned}
\left\|\int_\RM\partial_x^j\partial_t^l e(\cdot,t;y)g(y)dy\right\|_{L^p(\RM)}&\leq C\left(1+t\right)^{\frac{1}{2}-\frac{1}{2}\left(\frac{1}{q}-\frac{1}{p}\right)
              -\frac{(j+k)}{2}}\|g\|_{L^q(\RM)}\\
\end{aligned}
\end{equation}
\end{corr}

\br\label{Green-Bloch}
\textup{
As may be clear from the proofs of various linear estimates discussed above, except for the critical bounds on $e$ or $\tilde e$ in the case $j=l=0$, $p=\infty$ and $q=1$, the introduction of the critical Green kernel was a pure presentation device.  Indeed, we adopted here 
it to mark proximity with \cite{JZ1}.  Yet, since almost all the above linear bounds were proved using Haussdorff-Young type estimates, there may be some gain in clarity and efficiency in presentation in keeping all descriptions at the level of semigroups and Bloch symbols.
This latter approach was recently adopted in 
\cite{JNRZ1,JNRZ2,JNRZ3}, 
where further decompositions of the critical part of the 
solution operator are needed.
}
\er

\subsection{Nonlinear preparations}\label{s:pert}

Given the linearized bounds on the linear solution operator $S(t)=e^{Lt}$ derived in the previous section, we are now in position to consider
the effect of the small nonlinear terms
that were omitted in obtaining the linearized equation \eqref{e:lin}.
Our first task is to explain how to implement at the nonlinear level the separation of critical phase-shift contribution. 
To this end, let $\tilde{u}(x,t)$  be a solution of
\begin{equation}\label{e:kstravel}
\partial_t u-c\partial_x u+\partial_x^4u+\eps\partial_x^3u+\delta\partial_x^2u+\partial_x f(u)=0
\end{equation}
and define the spatially modulated function $u(x,t):=\tilde{u}(x+\psi(x,t),t)$, where 
$\psi:\RM\times\R_+\to\RM$ 
is a function to be determined later.  Moreover, let $\bar{u}(x)$ be a stationary periodic solution of \eqref{e:kstravel}
and define the nonlinear perturbation function
\be\label{pertvar}
v(x,t):=u(x,t)-\bar{u}(x).
\ee

\begin{lem}\label{l:cancel1}
For $v$ as above, we have
\be\label{veq}
(\d_t-L)(v-\psi\bar u_x)=\cN,\quad \cN=\d_x\cQ+\d_x \cR+\d_t \cS
\ee
where
$$
\begin{aligned}
\cQ&=-\left(f(\bar u+v)-f(\bar u)-df(\bar{u})v\right)\\
\cR&=\psi_t v-\delta\frac{-\psi_x}{1+\psi_x}v_x
-\eps\frac{-\psi_x}{1+\psi_x}\d_x\left(\frac{1}{1+\psi_x}v_x\right)
-\eps\d_x\left(\frac{-\psi_x}{1+\psi_x}v_x\right)\\
&-\frac{-\psi_x}{1+\psi_x}\d_x\left(\frac{1}{1+\psi_x}\d_x\left(\frac{1}{1+\psi_x}v_x\right)\right)
-\d_x\left(\frac{-\psi_x}{1+\psi_x}\d_x\left(\frac{1}{1+\psi_x}v_x\right)\right)\\
&-\d_x^2\left(\frac{-\psi_x}{1+\psi_x}v_x\right)-\delta\frac{\psi_x^2}{1+\psi_x}\bar u_x
-\eps\frac{-\psi_x}{1+\psi_x}\d_x\left(\frac{-\psi_x}{1+\psi_x}\bar u_x\right)
-\eps\frac{\psi_x^2}{1+\psi_x}\bar u_{xx}\\
&-\eps\d_x\left(\frac{\psi_x^2}{1+\psi_x}\bar u_x\right)
-\frac{-\psi_x}{1+\psi_x}\d_x\left(\frac{-\psi_x}{1+\psi_x}\d_x\left(\frac{1}{1+\psi_x}\bar u_x\right)\right)
-\frac{-\psi_x}{1+\psi_x}\d_x^2\left(\frac{-\psi_x}{1+\psi_x}v_x\right)\\
&-\frac{\psi_x^2}{1+\psi_x}\bar u_{xxx}
-\d_x\left(\frac{-\psi_x}{1+\psi_x}\d_x\left(\frac{-\psi_x}{1+\psi_x}\bar u_x\right)\right)
-\d_x\left(\frac{\psi_x^2}{1+\psi_x}\bar u_{xx}\right)-\d_x^2\left(\frac{\psi_x^2}{1+\psi_x}\bar u_x\right)\\
\cS&=-v\psi_x.
\end{aligned}
$$
\end{lem}

\begin{proof}
From the definition of $u$ above, it follows that
$$
\begin{aligned}
(1+\psi_x)&\d_t u-(\psi_t+c)\d_x u+\d_x f(u)+\delta\d_x\left(\frac{1}{1+\psi_x}\d_xu\right)\\
&+\eps\d_x\left(\frac{1}{1+\psi_x}\d_x\left(\frac{1}{1+\psi_x}\d_xu\right)\right)\\
&+\d_x\left(\frac{1}{1+\psi_x}\d_x\left(\frac{1}{1+\psi_x}\d_x\left(\frac{1}{1+\psi_x}\d_xu\right)\right)\right)=0.
\end{aligned}
$$
Hence, subtracting the profile equation $-c\bar u_x+\bar u_{xxxx}+\eps\bar u_{xxx}+\delta\bar u_{xx}+(f(\bar u))_x=0$
implies
$$
\begin{aligned}
(\d_t-L)v&-(\psi\bar u_x)_t=\d_x\cQ-\psi_xv_t+\psi_tv_x-\delta\d_x\left(\frac{-\psi_x}{1+\psi_x}(\bar u+v)_x\right)\\
&-\eps\d_x\left(\frac{-\psi_x}{1+\psi_x}\d_x\left(\frac{1}{1+\psi_x}(\bar u+v)_x\right)\right)
-\eps\d_x^2\left(\frac{-\psi_x}{1+\psi_x}(\bar u+v)_x\right)\\
&-\partial_x\left(\frac{-\psi_x}{1+\psi_x}\d_x\left(\frac{1}{1+\psi_x}\d_x\left(\frac{1}{1+\psi_x}(\bar u+v)_x\right)\right)\right)\\
&-\d_x^2\left(\frac{-\psi_x}{1+\psi_x}\d_x\left(\frac{1}{1+\psi_x}(\bar u+v)_x\right)\right)
-\d_x^3\left(\frac{-\psi_x}{1+\psi_x}(\bar u+v)_x\right)
\end{aligned}
$$
or, equivalently,
$$
\begin{aligned}
(\d_t-L)v-(\psi\bar u_x)_t&=\d_x\cQ+\d_x\cR+\d_t\cS-\delta\d_x(-\psi_x\bar u_x)-\eps\d_x(-\psi_x\bar u_{xx})-\eps\d_x^2(-\psi_x\bar u_x)\\
&-\d_x(-\psi_x\bar u_{xxx})-\d_x^2(-\psi_x\bar u_{xx})-\d_x^3(-\psi_x\bar u_x).
\end{aligned}
$$
Using the fact that $\psi L\bar u_x=0$ completes the proof.  
\end{proof}

Using Lemma \ref{l:cancel1} and applying Duhamel's principle, recalling that $\tilde{e}(x,t;y)=0$  for $0\leq t\leq 1$, we obtain the implicit
integral representation
$$
v(x,t)=\bar{u}_x(x)\psi(x,t)+\int_\RM G(x,t;y)v(y,0)dy+\int_0^t\int_\RM G(x,t-s;y)\cN(y,s)dy~ds
$$
for the nonlinear perturbation variable $v$ and some function $\psi$ still to be determined.  
Recalling Corollary \ref{e:Gbdsfinal}, it follows that by
defining $\psi$ implicitly via the integral formula
\begin{equation}\label{psi1}
\psi(x,t)=-\int_\RM e(x,t;y)v(y,0)dy-\int_0^t\int_\RM e(x,t-s;y)\cN(y,s)dyds
\end{equation}
the perturbation variable $v$ must satisfy the integral equation
\begin{equation}\label{vimplicit}
v(x,t)=\int_\RM \widetilde{G}(x,t;y)v(y,0)dy+\int_0^t\int_\RM \widetilde{G}(x,t-s;y)\cN(y,s)dy~ds.
\end{equation}
Furthermore, recalling that $e(x,s;y)=0$ for $0\leq s\leq 1$, we find by differentiating \eqref{psi1} that
\begin{equation}\label{psider}
\partial_t^j\partial_x^k\psi(x,t)=-\int_\RM\partial_t^j\partial_x^k e(x,t;y)v(y,0)dy-\int_0^t\int_\RM\partial_t^j\partial_x^ke(x,t-s;y)\cN(y,s)dy~ds.
\end{equation}
for $0\leq j\leq 1$ and $0\leq k\leq K+1$. 

\br\label{time-cutoff}
\textup{
The purpose of the time cutoff function $\chi$ introduced in \eqref{enew} is precisely to ensure that $\psi(\cdot,0)\equiv 0$, corresponding
to the fact that we are dealing with perturbations that are initially localized in space.  Yet, in this periodic context
it is reasonable to consider also perturbations which affect the phase or wave number of the underlying periodic profile, which
corresponds to perturbations which are \emph{not} initially localized in space.  Stability of periodic wave trains to such
non-localized perturbations have been the focus of much recent study; see \cite{SSSU,JNRZ2,JNRZ3} for analysis in the context
of reaction diffusion systems, and \cite{JNRZ1} for analysis in systems with a conservative structure.
}
\er

\br\label{nat:coordinates2}
\textup{
As discussed in Remark \ref{nat:coordinates}, the implicit $\psi$-dependent change of variables performed above 
has the effect of introducing at the linear level the critical non-decaying translational mode $\psi\bar u_x$, 
%
while at the same time ensuring that only derivatives of $\psi$, which decay in $L^p(\RM)$, appear
in the nonlinear terms $\mathcal{N}$ in Lemma \ref{l:cancel1}, as opposed to non-decaying terms involving $\psi$ itself.
}
\er

Equations \eqref{vimplicit} and \eqref{psider} together form a complete system in the variables 
$(v,\psi_t,\psi_x)$ 
and from the solution of this system, if it exists, we may recover the 
phase 
modulation $\psi$ through \eqref{psi1}.
Furthermore, the short-time existence and continuity with respect to $t$ of solutions 
$(v,\psi_t,\psi_x)\in H^K(\RM)\times H^K(\RM)\times H^{K+1}(\RM)$ 
follows from \eqref{veq} together with \eqref{psider} and a standard contraction-mapping argument based on \eqref{psi1}, 
\eqref{finalebds1}, \eqref{finalebds2},
and the following nonlinear damping estimate.

\begin{prop}\label{p:damping}
Assuming (H1), there exist positive constants $\theta$, $C$ and $\eps_0$ such that if $v$ and $\psi$ solve \eqref{veq} on $[0,T]$ for some $T>0$ and 
$$
\sup_{t\in[0,T]}\|(v,\psi_x)(t)\|_{H^K(\RM)}+\sup_{t\in[0,T]}\|\psi_t(t)\|_{H^{K-1}(\RM)}\leq\eps_0
$$ 
then, for all $0\leq t\leq T$,
\be\label{e:damping}
\begin{aligned}
\|v(t)\|^2_{H^K(\RM)}
&\leq Ce^{-\theta t}\|v(0)\|_{H^K(\RM)}^2\\
&+C\int_0^t e^{-\theta(t-s)}
\left(\|v(s)\|^2_{L^2(\RM)}+\|\psi_x(s)\|_{H^{K+1}(\RM)}^2+\|\psi_t(s)\|_{H^{K-2}(\RM)}^2\right)ds.
\end{aligned}
\ee
\end{prop}
\begin{proof}
First rewrite \eqref{veq} as
$$
\begin{aligned}
v_t&+\d_x^4v\ =\ -\eps \d_x^3v-\delta\d_x^2v+c\d_xv-\d_x(f(\bar u+v)-f(\bar u))-\frac{-\psi_x}{1+\psi_x}\d_x(f(\bar u+v))\\
&+\frac{\psi_t-c\psi_x}{1+\psi_x}(\bar u_x+v_x)
-\delta\frac{-\psi_x}{1+\psi_x}\d_x\left(\frac{1}{1+\psi_x}(\bar u_x+v_x)\right)-\delta\d_x\left(\frac{-\psi_x}{1+\psi_x}(\bar u_x+v_x)\right)\\
&-\eps\frac{-\psi_x}{1+\psi_x}\d_x\left(\frac{1}{1+\psi_x}\d_x\left(\frac{1}{1+\psi_x}(\bar u_x+v_x)\right)\right)
-\eps\d_x\left(\frac{-\psi_x}{1+\psi_x}\d_x\left(\frac{1}{1+\psi_x}(\bar u_x+v_x)\right)\right)\\
&-\eps\d_x^2\left(\frac{-\psi_x}{1+\psi_x}(\bar u_x+v_x)\right)
-\frac{-\psi_x}{1+\psi_x}\d_x\left(\frac{1}{1+\psi_x}\d_x\left(\frac{1}{1+\psi_x}\d_x\left(\frac{1}{1+\psi_x}(\bar u_x+v_x)\right)\right)\right)\\
&-\eps\d_x\left(\frac{-\psi_x}{1+\psi_x}\d_x\left(\frac{1}{1+\psi_x}\d_x\left(\frac{1}{1+\psi_x}(\bar u_x+v_x)\right)\right)\right)\\
&-\eps\d_x^2\left(\frac{-\psi_x}{1+\psi_x}\d_x\left(\frac{1}{1+\psi_x}(\bar u_x+v_x)\right)\right)
-\eps\d_x^3\left(\frac{-\psi_x}{1+\psi_x}(\bar u_x+v_x)\right).
\end{aligned}
$$
By taking scalar product against $\sum_{j=0}^K(-1)^{j}\d_x^j v$ and using integration by parts and Sobolev embedding $\|g\|_{L^\infty(\RM)}\leq C\|g\|_{H^1(\RM)}$, we obtain under a smallness assumption as in the statement of the proposition
$$
\frac12\frac{d}{dt}\left(\|v\|_{H^K(\RM)}^2\right)(t)+\frac12\|v(t)\|_{H^{K+2}(\RM)}^2
\ \leq\ C(\|v(t)\|_{H^{K+1}(\RM)}^2+\|\psi_x(t)\|_{H^{K+1}(\RM)}^2+\|\psi_t(t)\|_{H^{K-2}(\RM)}^2).
$$
Using now the Sobolev inequality
$$
\|g\|_{H^{K+1}(\RM)}^2\ \leq\ \eta\,\|\d_x^{K+2}g\|_{L^2(\RM)}^2+C_\eta\|g\|_{L^2(\RM)}^2
$$
with a sufficiently small $\eta$ reduces the result to a simple integration.
\end{proof}

\subsection{Nonlinear iteration}

With the above preparations in hand, we are now prepared to state the main technical lemma leading to the proof of Theorem \ref{main}.
For this purpose, associated with the solution $(v,\psi_t,\psi_x)$ of the integral system \eqref{vimplicit} and \eqref{psider} considered
in the previous section, define
\begin{equation}\label{eta}
\eta(t):=\sup_{0\leq s\leq t}\left\|(v,\psi_t,\psi_x,\psi_{xx})(s)\right\|_{H^K(\RM)}(1+s)^{1/4}.
\end{equation}
By standard short-time $H^K(\RM)$ existence theory, 
$\eta$ 
is continuous so long as it remains sufficiently small. Using the linearized estimates of Section \ref{s:linbds} we now
prove that if $\eta(0)$ is small then $\eta(t)$ remains small for all $t>0$.

\begin{lemma}\label{l:iteration}
Under assumptions (H1)--(H4) and (D1)--(D3), there exist positive constants $C$ and $\varepsilon_0$ such that if $v(0)$ is such that 
\[
E_0:=\|v(0)\|_{L^1(\RM)\cap H^K(\RM)} \leq \varepsilon_0\quad\textrm{and}\quad\eta(T) \leq \varepsilon_0
\] 
for some $T>0$, then for all $0\leq t\leq T$ we have
\be\label{eq:sclaim}
\eta(t)\le C(E_0+\eta(t)^2).
\ee
\end{lemma}
\begin{proof}
First, note that by Lemma \ref{l:cancel1}, we have under the smallness assumption 
on $\eta$ that
\begin{equation}\label{soursebd1}
\left\|(\cQ,\cR,\cS)(t)\right\|_{L^1(\RM)\cap L^2(\RM)}
\leq C\eta(t)^2(1+t)^{-1/2}
\end{equation}
for some constant $C>0$. The contribution of $\d_t\cS$ needs some special care. 
For this purpose, note that, integrating by parts in critical contribution, 
we obtain from previous linear bounds and $\cS(0)=0$
$$
\begin{aligned}
\left\|\int_0^t\int_\RM\tilde G(\cdot,t-s;y)\d_s\cS(y,s)dy~ds\right\|_{L^p(\RM)}&
\leq C\|\cS(t)\|_{L^2(\RM)}+C\int_0^t e^{-\theta(t-s)}\|\d_t\cS(s)\|_{H^1(\RM)}ds\\
&\quad+C\int_0^t (1+t-s)^{-\frac12\left(1-\frac1p\right)-\frac12}\|\cS(s)\|_{L^1(\RM)}ds.
\end{aligned}
$$ 
Now, observe that, using \eqref{veq} to bound $v_t$, we have under the same smallness assumption the estimate
\begin{equation}\label{soursebd2}
\left\|\cS(t)\right\|_{H^1(\RM)}\leq C\eta(t)^2(1+t)^{-1/2}.
\end{equation}
Using then the second part of \eqref{finalGbdyder} and \eqref{finalGbd1} to bound other terms, we obtain from \eqref{vimplicit}
the estimate
\begin{equation}\label{vintbd1}
\begin{aligned}
\left\|v(t)\right\|_{L^p(\RM)}&\leq C\left(1+t\right)^{-\frac{1}{2}\left(1-\frac{1}{p}\right)}(E_0+\eta(t)^2)\\
&\quad +C\eta(t)^2\int_0^t\left(t-s\right)^{-\frac{1}{4}\left(\frac{1}{2}-\frac{1}{p}\right)-\frac{1}{4}}
               \left(1+t-s\right)^{-\frac{1}{4}\left(\frac{3}{2}-\frac{1}{p}\right)-\frac{1}{4}}(1+s)^{-\frac{1}{2}}ds\\
&\leq C_p\left(E_0+\eta(t)^2\right)(1+t)^{-\frac{1}{2}\left(1-\frac{1}{p}\right)}
\end{aligned}
\end{equation}
and similarly, using \eqref{finalebds1} and \eqref{finalebds2} with $q=1$, we obtain from \eqref{psider} the estimate
\begin{equation}\label{psiintbd1}
\begin{aligned}
\left\|(\psi_t,\psi_x)\right\|_{W^{K+1,p}(\RM)}&\leq C(1+t)^{-\frac{1}{2}(1-\frac{1}{p})}E_0\\
&\quad+C\eta(t)^2\int_0^t\left(1+t-s\right)^{-\frac{1}{2}(1-1/p)-\frac12}(1+s)^{-1/2}ds\\
&\leq C_p\left(E_0+\eta(t)^2\right)(1+t)^{-\frac{1}{2}\left(1-\frac{1}{p}\right)},
\end{aligned}
\end{equation}
valid for all\footnote{Notice that the above integral estimates fail in the case $p=\infty$ due to a term of size $\log(1+t)$
arising from integrating on $\left[\frac{t}{2},t\right]$.} $2\leq p<\infty$.
The assumed smallness guarantees that we may apply Proposition~\ref{p:damping} and get
\begin{align*}
\|v(t)\|_{H^K(\RM)}^2&\leq Ce^{-\theta t}E_0^2+C(E_0+\eta(t)^2)^2\int_0^t e^{-\theta(t-s)}(1+s)^{-1/2}ds\\
&\leq C\left(E_0+\eta(t)^2\right)^2(1+t)^{-1/2}.
\end{align*}
Together with \eqref{psiintbd1}, this completes the proof.
\end{proof}

\noindent
{\bf Proof of  Theorem \ref{main}.} 
We are free to assume the constant $C$ in \eqref{eq:sclaim} is larger than $1$. Since $\eta$ is continuous and $\eta(0)=\|v(0)\|_{H^K(\RM)}\leq E_0$, it follows by continuous induction that if $4C^2E_0<1$ then $\eta(t)\leq 2CE_0$ for all $t\geq0$. We may then use \eqref{vintbd1} and \eqref{psiintbd1} to get uniform bounds for $p\in[2,4]$
$$
\|v(t)\|_{L^p(\RM)}+\|(\psi_t,\psi_x)(t)\|_{W^{K+1,p}(\RM)}
\leq CE_0(1+t)^{-\frac12\left(1-\frac1p\right)}.
$$
Next, we rewrite $\cR$ in form 
$$
\begin{aligned}
\cR&=\cR_1+\d_x\cR_2+\d_x^2\cR_3+\d_x^3\cR_4,\\
\cR_j(x,t)&=\cR^{(0)}_j(x,\psi_x(x,t),\psi_{xx}(x,t),\psi_{xxx}(x,t),\psi_{xxxx}(x,t))\\
&+\cR^{(1)}_j(x,\psi_x(x,t),\psi_{xx}(x,t),\psi_{xxx}(x,t),\psi_{xxxx}(x,t))v(x,t)
\end{aligned}
$$
so that
$$
\left\|(\cQ,\cR_1,\cR_2,\cR_3,\cR_4,\cS)(t)\right\|_{L^2(\RM)}
\leq C(\|v(t)\|_{L^4(\RM)}^2+\|(\psi_t,\psi_x)(t)\|_{W^{K+1,4}(\RM)}^2)
\leq CE_0(1+t)^{-\frac{3}{4}}
$$
and
$$
\left\|(\d_x\cQ,\d_x\cR_1,\d_x^2\cR_2,\d_x^3\cR_3,\d_x^4\cR_4,\d_t\cS)(t)\right\|_{H^1(\RM)}
\leq C E_0(1+t)^{-\frac{1}{2}}.
$$
Using now the estimate
$$
\begin{aligned}
\left\|\int_0^t\int_\RM\tilde G(\cdot,t-s;y)\d_s\cS(y,s)dy~ds\right\|_{L^p(\RM)}&\leq C\|\cS(t)\|_{L^2(\RM)}+C\int_0^t e^{-\theta(t-s)}\|\d_t\cS(s)\|_{H^1(\RM)}ds\\
&\quad+C\int_0^t (1+t-s)^{-\frac12\left(\frac12-\frac1p\right)-\frac12}\|\cS(s)\|_{L^2(\RM)}ds,
\end{aligned}
$$
together with the first part of \eqref{finalGbdyder} and \eqref{finalGbd1}, we obtain for all $p\in[2,\infty]$
$$
\begin{aligned}
\|v(t)\|_{L^p(\RM)}&\leq C(1+t)^{-\frac12\left(1-\frac1p\right)}E_0
+C E_0\int_0^t\left(1+t-s\right)^{-\frac12\left(\frac12-\frac1p\right)-\frac12}(1+s)^{-\frac34}ds\\
&\leq C E_0(1+t)^{-\frac12\left(1-\frac1p\right)}
\end{aligned}
$$
and, similarly, using 
\eqref{finalebds1} with $q=2$ and \eqref{finalebds2} both with $q=1$ and $q=2$ 
we obtain the estimate
$$
\begin{aligned}
\|(\psi_t,\psi_x)(t)\|_{W^{K+1,p}(\RM)}&
\leq C(1+t)^{-\frac12\left(1-\frac1p\right)}E_0
+C E_0\int_0^t\left(1+t-s\right)^{-\frac12\left(\frac12-\frac1p\right)-\frac12}
\left(1+s\right)^{-\frac34}ds\\
&\leq C E_0(1+t)^{-\frac12\left(1-\frac1p\right)}.
\end{aligned}
$$
This completes the proof of \eqref{eq:smallsest}.

To establish \eqref{eq:stab}, note that \eqref{finalebds1} and \eqref{finalebds2} imply
$$
\left\|\psi(t)\right\|_{L^\infty(\RM)}
\leq CE_0+CE_0^2\int_0^t\left(1+t-s\right)^{-\frac{1}{2}\left(1-\frac{1}{p}\right)}(1+s)^{-\frac{1}{2}}ds
\leq CE_0.
$$
Finally, notice that by definition we have that
\[
\tilde{u}(x,t)-\bar{u}(x)=v(x,t)+\left(\tilde{u}(x,t)-\tilde{u}(x+\psi(x,t),t)\right),
\]
hence 
$$
\|\tilde{u}(t)-\bar{u}\|_{L^\infty(\RM)}\leq \|v(t)\|_{L^\infty(\RM)}
+\|\bar u_x\|_{L^\infty([0,1])}\|\psi(t)\|_{L^\infty(\RM)}
\leq C E_0.
$$
We thus obtain the final $L^1\cap H^K\to L^\infty$ bound \eqref{eq:stab} of Theorem \ref{main}.
This completes the proof of the main theorem, hence establishing the nonlinear $L^1\cap H^K\to L^\infty$ stability of the underlying periodic traveling wave $\bar{u}$ under the structural and spectral assumptions (H1)--(H4) and (D1)--(D3).

\medskip
{\bf Acknowledgement}:
The numerical stability computations in this paper were carried out
using the STABLAB package developed by Jeffrey Humpherys and
the first and last authors and the SpectrUW package
developed by Bernard Deconinck and collaborators;
see \cite{BHZ2,CDKK} for documentation. 
We gratefully acknowledge their contribution.
We thank also
Indiana University Information Technology Service for
the use of the Quarry computer with which some of
our computations 
were 
carried out.

\appendix


\section{Appendix: Survey of existence theory for generalized KS \label{a:survey}}

In this appendix we give a brief survey of the existence theory for periodic traveling wave solutions
of the generalized Kuramoto-Sivashinsky equation \eqref{e:KS} with nonlinearity $f(u)=3u^2$.  For the forthcoming
discussion we will also consider the particular scaling where $\gamma=\delta$, i.e. the equation
\be\label{e:cks}
u_t+6uu_x+\varepsilon \partial_x^3 u +\delta(\partial_x^2 u+\partial_x^4u)=0.
\ee
Under this scaling, there are three situations to consider: the rather generic case where $\varepsilon=1$ and $\delta=\mathcal{O}(1)$; the classic Kuramoto-Sivashinsky limit $\delta=1$ and $|\eps|\ll 1$ and the (integrable) Korteweg-de Vries (KdV) limit $\eps=1$ and $0<\delta\ll 1$.  In the first two cases, we obtain only existence of small amplitude periodic wave trains through a normal form approach whereas we obtain existence of large amplitude periodic traveling waves in the KdV limit, by using Fenichel theory.

\subsection{The generic case: Hopf bifurcation analysis}\label{a:hopf}

Traveling waves $u(x,t)=U(x-ct)$ of \eqref{e:cks} are then readily seen to be solutions of the ODE
\be\label{e:pr}
\delta(U'''+U')+\eps U''+3U^2-cU=q,
\ee
where $q\in\RM$ is a constant of integration.
We begin by considering the generic situation where $\eps>0$ and $\delta\neq 0$ are fixed and arbitrary.
In this case, so long as $c^2+12 q>0$, \eqref{e:pr} possess two stationary solutions
$U_{-}(c,q)<U_{+}(c,q)$ such that $3U_{\pm}^2-cU_{\pm}=q$.  Linearizing \eqref{e:pr} about a constant state $U=U_{\pm}$
thus yields 
\begin{equation}\label{ep_lin}
(6U_{\pm}-c)\tilde{U}+\eps\tilde{U}''+\delta\big(\tilde{U}'+\tilde{U}'''\big) = 0,
\end{equation}
which is seen to undergo a Hopf bifurcation (necessarily) at $U_+$ when $c_{\rm Hopf}(q)=6U_{+}-\eps$.
Denoting $V=U-U_+$ and $6U_+-c=\eps+\mu$ with $|\mu|\ll 1$, equation \eqref{e:pr} then reads
\be\label{e:hopf_bifeqn}
(3V+\eps+\mu)V+\eps V''+\delta(V'+V''')=0,
\ee
from which we find, for sufficiently small $\mu$ on one side of zero, a family $V(\mu)$ of periodic
orbits with amplitude $\mathcal{O}(|\mu|^{1/2})$ and frequency $k(\mu)$, with $k(0)=1$ and $k'(0)\neq 0$.
As a result, in the neighborhood of the Hopf bifurcation we find a three dimensional manifold of small amplitude
periodic traveling waves parameterized by translation,  the wavenumber $k$, and the integration constant $q$.

\br\textup{
Note 
that when $c^2+12q=0$, the equilibrium states coincide.  The nature of the bifurcation occurring at this
degenerate point requires a more delicate normal form analysis, which we postpone to the next section.
}
\er

\noindent
Finally, notice that for $|\mu|\ll 1$ the solution $U$ of \eqref{e:pr} satisfies
\[
U=\left<U\right>+\mathcal{O}(|\mu|^{1/2}),
\]
where $\left<U\right>$ denotes the spatial mean of $U$ over a period,
so that, by the definition of $q$, we find
\[
3\left<U\right>^2-c\left<U\right>=q+\mathcal{O}(|\mu|^{1/2}).
\]
Setting $M(k,q)=\left<U\right>$ it follows that 
$\d_qM$ 
is non-zero %
by the definition of $\mu$. Here, we are using that $\eps>0$ is fixed.  In particular, in a neighborhood of the Hopf bifurcation,
we can switch from the $(k,q)$ parametrization of the local manifold of periodic traveling waves (here identified up to translation)
to the $(k,M)$ parametrization; see Remark \ref{r:linphasecouple}.  This 
provides
the following proposition.

\begin{proposition}
Let $M\in\R$ be fixed and $k<1$ be such that $1-k$ is small. Then there exist a unique $c(k,M)$ and a unique $q(k,M)$ such that there exists a $1$-periodic $U(\,\cdot\,;k,M)$ solution of
\begin{equation}
\label{epIk}
\displaystyle
k\left(3U^2-cU\right)+\eps k^3U''+\delta\left(k^2U'+k^4U'''\right)\ =\ q,\quad \left<U\right>=M\ .
\end{equation}
Moreover this solution is unique up to translation.
\end{proposition}

\subsection{The classic KS limit: a normal form analysis}\label{a:nf}

Here, we complement the above bifurcation analysis in the previous section with a normal form analysis near the degenerate case $|\eps|\ll 1$. In order to find periodic solutions, we necessarily work in the neighborhood of  $0<c^2+12q\ll 1$. Setting $V=U-U_+$ then we see that $V$ must satisfy \eqref{e:hopf_bifeqn} which we can rewrite here as the first order differential system
\be\label{edo3}
\left(\begin{array}{c}x'\\y'\\z'\end{array}\right)=\left(\begin{array}{ccc} 0 & 1 & 0\\-1 & 0 & 0\\0 & 0 & 0\end{array}\right)\left(\begin{array}{c}x\\y\\z\end{array}\right)+\mu(x+z)\left(\begin{array}{c}1\\0\\-1\end{array}\right)+\eps z\left(\begin{array}{c}1\\0\\-1\end{array}\right)+6(x+z)^2\left(\begin{array}{c}1\\0\\-1\end{array}\right)
\ee
where $x=2V+V''$, $y=V'$, $z=V+V''$, and $\eps>0$ and $\mu$ are small.

Next we compute the normal form of (\ref{edo3}).  By finding an appropriate change of variables of the form 
\[
(x,y,z)=(Id+\mu T_1+\eps T_2+Q)(\tilde{x},\tilde{y},\tilde{z})
\]
where the $T_i$ are linear operators and $Q$ a quadratic vector field, we obtain after a lengthy but straightforward computation
\be\label{nf}
\begin{aligned}
\tilde{x}'&=\tilde{y}+\mu\tilde{x}/2+\Gamma\tilde{x}\tilde{z}+\mathcal{O}(3)\\
\tilde{y}'&=-\tilde{x}+\mu\tilde{y}/2+\Gamma\tilde{y}\tilde{z}+\mathcal{O}(3)\\
\tilde{z}'&=-(\eps+\mu)\tilde{z}-\Gamma\big(\tilde{z}^2+(\tilde{x}^2+\tilde{y}^2)/2\big)+\mathcal{O}(3).
\end{aligned}
\ee
where $\mathcal{O}(3)$ denotes terms cubic order or higher terms.
Finally, using polar coordinates $r^2=\tilde{x}^2+\tilde{y}^2$ and $\theta=\cos^{-1}\left(\frac{\tilde{x}}{r}\right)$,
and dropping all resulting $\mathcal{O}(3)$ terms we obtain the normal form
\be\label{nf:cyl}
\begin{aligned}
r'&=\mu \frac{r}{2}+\Gamma rz\\
z'&=-(\eps+\mu)z-\Gamma\big(z^2+\frac{r^2}{2})\\
\theta'&=-1
\end{aligned}
\ee
valid in a neighborhood of $(\eps,\mu)=(0,0)$.

The dynamics of the normal form equation \eqref{nf:cyl} is quite easy to describe.  In the corresponding three dimensional phase
plane we find two fixed points in the r-z plane with $r=0$: namely, $(r,z)=(0,0)$ and $(r,z)=(0,-(\eps+\mu)/\Gamma)$.  By definition, these correspond
to the steady solutions $U_{\pm}$ of the original system \eqref{e:pr}.
The Jacobian matrix J at these points are, respectively,
\[
J(0,0)=\left(\begin{array}{cc} \mu/2 & 0\\ 0 & -(\eps+\mu)\end{array}\right)\textrm{  and  }
J\left(0,\frac{-(\eps+\mu)}{\Gamma}\right)\left(\begin{array}{cc} -\eps-\mu/2 & 0\\ 0 & \eps+\mu\end{array}\right).
\]
Furthermore, nontrivial stationary solutions of \eqref{nf:cyl} in the r-z plane exist only if $\mu^2+2\mu\eps$ is positive, in which case there is only
one possible stationary point $P$ given by
\[
(r,z)=\left(\sqrt{\frac{\mu^2+2\eps\mu}{2\Gamma^2}}\;,-\frac{\mu}{2\Gamma}\right),
\]
which
is a non-trivial stationary solution of \eqref{nf:cyl} corresponding to a periodic solution to the original system \eqref{e:pr}.
%
We readily find that the trace and determinant of the Jacobian matrix $J_P$ are
\[
{\rm tr}(J_P)=-\eps,\quad {\rm det}(J_P)=\frac{\mu^2+2\eps\mu}{2}
\]
so that the periodic point  $P$ undergoes a Hopf bifurcation (in the r-z plane) at $\eps=0$, which corresponds precisely to the Kuramoto-Sivashinsky equation. To determine whether
it is a sub or supercritical bifurcation, one would need to compute higher order terms in the normal
form, which is beyond the scope of this paper. These periodic orbits emerging from $P$ in the r-z plane
correspond to quasiperiodic solutions of the original system \eqref{e:pr}.

Finally, note that in \cite{CD} a similar normal form was derived in the case when $\eps=0$, corresponding to the ``classic" KS equation
of the form
\[
\tilde{r}'=-\tilde{\mu}\tilde{r}/2+2\tilde{r}\tilde{z},\quad \tilde{z}'=\tilde{\mu}\tilde{z}-2\tilde{z}^2-4\tilde{r}^2,\quad\theta'=1
\]
from which we can recover \eqref{nf:cyl} in this case by a simple rescaling. In both cases, if ${\rm det}(J_P)=\frac{\mu^2+2\eps\mu}{2}\neq 0$, a straightforward application of the implicit function theorem shows that the point $P$, corresponding to a  periodic wave train of \eqref{e:cks}, persists 
under 
higher order perturbations.\\


We further note that in \cite{CD} a full family of periodic solutions to (\ref{nf:cyl}) for the Poincar\'e return map around the point $P$ and ending with a solitary wave was found. However, it was not proved that such a family of quasi periodic solutions persists under higher order perturbations. This is not surprising since we are precisely at the Hopf bifurcation point and $\eps$ is the additional parameter needed in KS to carry out a codimension $2$ bifurcation analysis. In contrast when $0<|\eps|\ll 1$, there is a selection of the periodic orbit of the Poincar\'e return map and this structure persists for higher order perturbations. \\


Let us mention that the full bifurcation analysis for KS is far more complicated: indeed Kent and Elgin \cite {KE} proved the occurrence of a Shi'lnikov bifurcation which leads to  cascades of period doubling, period multiplying $k$-bifurcations and oscillatory homoclinic
as period is increased.
The computation of the bifurcation diagram was also investigated numerically for KS \cite{BKJ} and work is still in progress to carry out a similar program for gKS by using AUTO continuation software.

\subsection{The KdV limit: singular perturbation analysis}\label{a:kdv}

We finally conclude our survey of relevant existence results for periodic traveling wave solutions of \eqref{e:cks} by considering a particular singular limit arising 
in applications to pattern formation analysis. 
In this context, it is usually assumed that $\eps=1$ and $0<\delta\ll 1$ in which case \eqref{e:cks} can
be treated as a singular perturbation of the integrable KdV equation
\be\label{e:kdv}
u_t+uu_x+u_{xxx}=0.
\ee
As is well known, \eqref{e:kdv} admits a four parameter family of periodic traveling wave solutions of the form
\be\label{e:kdvsoln}
U_{\rm KdV}(x-c(u_0,\kappa,p)t;x_0,u_0,\kappa,p)=u_0+12p^2\kappa^2\cn^2\left(\kappa(x-x_0-c(u_0,\kappa,p)t),p\right)
\ee
with wave speed $c(u_0,\kappa,p):=8\kappa^2p^2-4\kappa^2+u_0$,
where $\cn(\cdot,p)$ denotes the standard Jacobi elliptic cosine function with elliptic modulus $p\in[0,1)$,
and $x_0$, $u_0$, and $\kappa$ are arbitrary real numbers corresponding, respectively, to the
translation, Galilean, and scaling symmetries of the KdV equation \eqref{e:kdv}.

Here, we are interested in the continuation of
these explicit solutions to the KdV equation in the singular limit $\delta\to 0^+$ in \eqref{e:cks}.
In this limit, it was shown in \cite{EMR}, by using Fenichel theory, that periodic solutions to \eqref{e:cks} remain close to the above elliptic function
solutions of the KdV and, in particular, an expansion of these solutions with respect to $\delta$
was obtained.  Furthermore, we point out 
that 
in \cite{BaN} a formal spectral analysis of these perturbed KdV waves was conducted. 
We now briefly describe the associated expansions of the periodic traveling waves in the 
limit $\delta\to 0^+$; see \cite{EMR,BaN,NR2,JNRZ4} for more details.

Without loss of generality, we can restrict to zero-mean solutions of \eqref{e:cks}.  We seek
an expansion of the associated periodic traveling wave solution of the form
\[
U^\delta(x,t)=U_{\rm KdV}(x,t;x_0,u_0^*,\kappa,p)+\delta\tilde{U}(x,t)+\mathcal{O}(|\delta|^2)
\]
where $u^*$ is chosen so that the mean of $U_{\rm KdV}$ over one spatial period vanishes.  Notice at
order $\delta^0$ the parameters $p$ and $\kappa$ are completely arbitrary, yielding a three parameter
family of periodic traveling wave solutions of \eqref{e:kdv} with zero mean.  When $0<\delta\ll 1$ however,
we expect only a \emph{two} parameter family to persist, hence we expect a selection principle
to manifest itself at the next order between the parameters $p$ and $\kappa$.

At the next order one finds that the first-order correction $\tilde{U}$ must satisfy the equation
\be\label{e:firstcorr}
\kappa^2\tilde{U}'''+(6U_{\rm KdV}\tilde{U}-c(u_0^*,\kappa,p)\tilde{U})'-\tilde{c}_1U_{\rm KdV}'+\kappa U_{\rm KdV}''+\kappa^3U_{\rm KdV}''''=0,
\ee
where $\tilde{c}_1$ denotes the first order $\delta$-correction to the wave speed of $U^\delta$.  As the linear
operator
\[
\mathcal{L}:=\kappa^2\frac{d^3}{dx^3}+\frac{d}{dx}\left((6U_{\rm KdV}-c(u_0^*,\kappa,p)).\right)
\]
has Fredholm index 0
and the kernel of its adjoint operator is spanned by constant functions and $U_{KdV}(\cdot;x_0,u_0^*,\kappa,p)\}$,  
it follows that \eqref{e:firstcorr} will have a solution precisely when the solvability condition
\be\label{solvability}
\Big<\Big(U_{\rm KdV}'\Big)^2\Big>=\kappa^2\Big<\Big(U_{\rm KdV}''\Big)^2\Big>
\ee
is satisfied. Condition \eqref{solvability} yields an explicit selection principle between
the elliptic modulus $p$ and the scaling parameter $\kappa$ and, as a result, we find
for $0<\delta\ll 1$ a two-parameter family of periodic traveling wave solutions of \eqref{e:cks}
parameterized by translation and the elliptic modulus $p$ or, equivalently, by translation
and period.

Coming back to the more general case where we allow $U_{\rm KdV}$ to have non-zero mean, we have
obtained for $0<\delta\ll 1$ a three dimensional manifold of periodic traveling wave solutions of \eqref{e:cks} parameterized
by translation, period, and spatial mean over a period. 

%
The numerical tests carried out here and in \cite{CDK} suggests that there are finite limit for lower and upper bounds of stability as $\delta\to 0$,
in agreement with results predicted by formal singular perturbation analysis in \cite{BaN}.
The purpose of the recent work \cite{JNRZ4}, although still finally relying on elliptic integrals numerical computations of \cite{BaN}, is precisely to go a step further towards a complete analytic proof of these observations.

\section{Appendix: The Swift-Hohenberg equation}\label{a:sh}

In this appendix, we demonstrate a streamlined proof that spectrally modulationally stable
periodic traveling wave solutions of the Swift-Hohenberg equation \eqref{e:sh} are nonlinearly
stable to small localized perturbations.  While this problem has been previously solved
by Schneider in \cite{Sc} by using a combination of weighted energy estimates, renormalization theory,
and ingenious nonlinear cancellation technique, all carried out in the Bloch frequency domain, our nonlinear analysis rather relies on spatial domain techniques developed in \cite{OZ4,JZ1,JZ3} in the context of systems of viscous conservation laws. We also carry out a numerical spectral stability analysis, demonstrating generality of our techniques. In particular, 
in a specific parameter regime we 
obtain nice agreement with stability curves found by Mielke \cite{M1}.

\subsection{Setup and main result}
The traveling wave solutions of \eqref{e:sh} are stationary solutions of the PDE
\begin{equation}\label{e:travsh}
\partial_t u-c\partial_x u+(1+\partial_x^2)^2u-ru+f(u)=0
\end{equation}
for some wave speed $c\in\RM$, i.e. they are solutions of the traveling wave ODE
\begin{equation}\label{e:twsh}
-cu'+u''''+2u''+(1-r)u+f(u)=0.
\end{equation}
Due to the presence of the non-conservative terms, this equation can not be further integrated, 
hence the orbits of \eqref{e:twsh} lie in the phase space $\RM^4$.  In particular, it follows that periodic orbits $u$ of
\eqref{e:twsh} correspond to values $(b,c,X)\in\RM^6$, where $X,c\in\RM$ denote the period and speed, respectively,
and the vector $b=(b_1,b_2,b_3,b_4)$ denotes the values of $(u,u',u'',u''')$ at $x=0$ such that
\[
(u,u',u'',u''')(X; b,c)=b,
\]
\noindent
where $(u,u',u'',u''')(\cdot; b,c)$ is the unique solution of \eqref{e:twsh} so that $(u,u',u'',u''')(0; b,c)=b$.
As usual, we make the following technical assumptions:
\begin{itemize}
\item[(H1')] $f\in C^{K}$, $K\ge 5$.
\item[(H2')] The map $H: \,
\R^6 \rightarrow \R^4$	
taking $(b,c,X) \mapsto (u,u',u'',u''')(X; b, c)-b$
is full rank at $(\bar b, \bar c,\bar X)$.
\end{itemize}

By the Implicit Function Theorem, conditions (H1')--(H2') imply that the set of periodic solutions 
of \eqref{e:sh} 
in the vicinity of 
$\bar u$, with parameters $(\bar b, \bar c,\bar X)$, 
forms a smooth $2$-dimensional manifold 
$$
\left\{\ (x,t)\mapsto U(x-\alpha-c(\beta)t;\beta)\ \middle|\ (\alpha,\beta)\in\RM\times\mathcal{I}\ \right\},
\quad
\hbox{\rm with $\mathcal{I}\subset \RM$}.
$$
In particular, in contrast to the Kuramoto-Sivashinsky equation \eqref{e:gen}, the lack of conservative
structure implies a loss in dimension of the periodic solution manifold.  As we will see, this has the effect
that there are no longer ``enough" periodic orbits around to make variations in wave speed an admissible
perturbation, hence the corresponding linearization does not have a (co-periodic) Jordan block at the origin.
%

\br\label{H2rmk:sh}
\textup{
As noted in \cite{JZN}, transversality, (H2'),
is necessary for our notion of spectral stability
hence there is no loss of generality in making this assumption.
}
\er

To begin our stability analysis we consider the linearization of \eqref{e:travsh} about a fixed 
$\bar X$-periodic 
traveling wave solution 
$\bar{u}=u(\cdot;\bar b,\bar c)$.
Here we assume 
without 
loss of generality $\bar X=1$.
To this end, consider a nearby solution of \eqref{e:travsh} of the form
$\bar{u}(x)+v(x,t)$ with $v$ small.
Directly substituting this into \eqref{e:travsh} and neglecting quadratic order terms in $v$ leads
us to 
the linearized equation $(\d_t-L)v=0$ with 
%
\begin{equation}\label{e:lin-sh}
Lv:=cv_x-v_{xxxx}-2v_{xx}+(r-1)v-f'(\bar{u})v.
\end{equation}
Introducing the one-parameter family of Bloch operators
\[
L_\xi:=e^{-i\xi x}Le^{i\xi x},\;\;\xi\in[-\pi,\pi)
\]
operating on $L^2_{\rm per}([0,1])$
parameterizes the spectrum of $L$ as
\[
\sigma_{L^2(\RM)}(L)=\bigcup_{\xi\in[-\pi,\pi)}\sigma_{L^2_{\rm per}([0,1])}\left(L_\xi\right).
\]
We assume the following spectral 
stability conditions:
\begin{itemize}
\item[(D1')] $\sigma(L) \subset \left\{ \lambda\in\CM \middle| \hbox{\rm Re}(\lambda)< 0  \right\}\cup\{0\}$.
\item[(D2')] $\sigma(L_\xi) \subset \left\{ \lambda\in\CM \middle| \hbox{\rm Re}(\lambda)\leq -\theta |\xi|^2 \right\}$, for some $\theta>0$ and any $\xi\in[-\pi,\pi)$.
\item[(D3')] $\lambda=0$ is a simple eigenvalue of $L_0$.
\end{itemize}
By standard spectral perturbation theory \cite{K}, (D3') implies that the critical eigenvalue $\lambda(\xi)$ bifurcating
from $\lambda=0$ at $\xi=0$ is analytic in $\xi$.  In particular, under assumption (D2')
and the symmetry $\lambda(\xi)=\overline{\lambda(-\xi)}$, the $\lambda(\xi)$ admits an expansion as $\xi\to 0$ of the form 
\begin{equation}\label{e:surfacessh}
\lambda(\xi)=-ia\xi-b\xi^2+\mathcal{O}(|\xi|^3),\quad a\in\RM,\quad b>0.
\end{equation}
%

\br\textup{
More generally, assume that a real valued, periodic coefficient differential operator $\mathcal{L}$ whose
essential spectrum near the origin admits a Bloch representation as a single analytic curve of the form $\lambda(\xi)=\sum_{j\in\NM} \alpha_j\xi^j$.
Then the symmetry $\lambda(\xi)=\overline{\lambda(-\xi)}$ implies that $\alpha_{2j}\in \RM$ and $\alpha_{2j+1}\in i\RM$ for each $j\in\mathbb{N}$. 
Finally, we note that this observation is related to properties of the associated Whitham modulation equations as well; see \cite{JNRZ1}
for details.
}\label{r:evenoddspec}
\er

With these preparations in hand, we now state the main theorem of this appendix.

\begin{theo}\label{main-sh}
Let $\bar u$ be any steady 1-periodic solution of \eqref{e:travsh} such that (H1')--(H2') and (D1')--(D3') hold. Then there exist constants $\varepsilon_0>0$ and $C>0$ such that for any $\tilde u_0$ with $\|\tilde{u}_0-\bar{u}\|_{L^1(\RM)\cap H^K(\RM)}\leq\varepsilon_0$, where $K$ is
as in assumption (H1'), there exist $\tilde{u}$ a solution of \eqref{e:travsh} satisfying $\tilde u(\cdot,0)=\tilde u_0$ 
and a function $\psi(\cdot,t)\in W^{K,\infty}(\RM)$ such that for all $t\geq 0$ and $2\leq p\leq\infty$
we have the estimates
\begin{equation}\label{eq:smallsest-sh}
\begin{aligned}
\left\|\tilde{u}(\cdot+\psi(\cdot,t),t)-\bar{u}(\cdot)\right\|_{L^p(\RM)}
         &\leq C\left(1+t\right)^{-\frac{1}{2}\left(1-\frac{1}{p}\right)-\frac12}\left\|\tilde{u}(\cdot,0)-\bar{u}(\cdot)\right\|_{L^1(\RM)\cap H^K(\RM)},\\
\left\|(\psi_t,\psi_x)(\cdot,t)\right\|_{L^p(\RM)}
         &\leq C\left(1+t\right)^{-\frac{1}{2}\left(1-\frac{1}{p}\right)-\frac12}\left\|\tilde{u}(\cdot,0)-\bar{u}(\cdot)\right\|_{L^1(\RM)\cap H^K(\RM)},\\
\left\|\tilde{u}(\cdot+\psi(\cdot,t),t)-\bar{u}(\cdot)\right\|_{H^K(\RM)}
         &\leq C\left(1+t\right)^{-\frac{3}{4}}\left\|\tilde{u}(\cdot,0)-\bar{u}(\cdot)\right\|_{L^1(\RM)\cap H^K(\RM)},\\
\left\|(\psi_t,\psi_x)(\cdot,t)\right\|_{H^K(\RM)}
         &\leq C\left(1+t\right)^{-\frac{3}{4}}\left\|\tilde{u}(\cdot,0)-\bar{u}(\cdot)\right\|_{L^1(\RM)\cap H^K(\RM)}.
\end{aligned}
\end{equation}
Moreover, we have the $L^1(\RM)\cap H^K(\RM)\to L^p(\RM)$, $2\leq p\leq\infty$, nonlinear stability estimate
\begin{equation}\label{eq:stab-sh}
\left\|\tilde{u}(\cdot,t)-\bar{u}(\cdot)\right\|_{L^p(\RM)},~\left\|\psi(\cdot,t)\right\|_{L^p(\RM)}\leq C\left(1+t\right)^{-\frac{1}{2}\left(1-\frac{1}{p}\right)}\left\|\tilde{u}(\cdot,0)-\bar{u}(\cdot)\right\|_{L^1(\RM)\cap H^K(\RM)}
\end{equation}
valid for all $t\geq 0$.
\end{theo}
The outline of the proof is as follows.  Given a fixed periodic solution $\bar{u}(x)$ of the traveling wave ODE \eqref{e:twsh}, nearby 
solutions of the traveling SH equation \eqref{e:travsh} $\tilde{u}$ are investigated by  defining a nonlinear perturbation variable
\begin{equation}\label{pertvar-sh}
v(x,t):=\tilde{u}(x+\psi(x,t),t)-\bar{u}(x),
\end{equation}
where $\psi:\RM\times\R_+\to\RM$ is a spatial-temporal phase modulation $\psi$ 
to be chosen later such that $\psi(\cdot,0)\equiv0$, corresponding to perturbations
which are initially spatially localized. 
By a direct calculation, we find that \eqref{e:travsh} is then written
\[
\left(\partial_t-L\right)(v-\bar{u}_x\psi)=\mathcal{N}[v,\psi]
\]
where $L$ is the linearized operator defined in \eqref{e:lin-sh} and $\mathcal{N}$ denotes a nonlinear remainder term depending on derivatives of $v$ and $\psi$: 
see \eqref{veq-sh} below. By an application of Duhamel's principle, it follows that \eqref{e:travsh} reads
\begin{equation}\label{vintegral-sh}
v(\cdot,t)-\bar{u}_x\psi(\cdot,t)=e^{Lt}v(\cdot,0)+\int_0^te^{L(t-s)}\mathcal{N}[v,\psi](\cdot,s)ds.
\end{equation}
Of course, we need then a detailed study of the linearized solution operator $e^{Lt}$. In the present periodic context, this is complicated
by the fact that the spectrum of $L$ agrees with the essential spectrum and contains a spectral curve touching the imaginary axis. We are thus led to separate the exponentially-stable spectrum from this low-Floquet critical curve for which, at best, one should only expect polynomial decay of small perturbations. A full separation at the linear level would lead to a separation of the nonlinear equation 
into a decay-critical equation and an exponentially slaved equation forced by nonlinear coupling terms.
Although we perform the linear separation only in an approximate way, the purpose of the introduction of $\psi$ is to ensure that $\psi \bar u_x$ contains at the linear level the main contribution of the algebraically decaying part of the semigroup, decay of $v$ being essentially slaved.

\subsection{Linearized estimates}\label{s:linearizedestimates-sh}

We begin our analysis deriving decay rates for the semigroup $e^{Lt}$ of the linearized operator $L$ defined in \eqref{e:lin-sh}. 
With the inverse Bloch transform representation \eqref{IBFT} of the solution operator in mind, we begin by making more precise the statements about the critical curve. 
Since $0$ is separated from the rest of the spectrum of $L_0$, we first note that by standard spectral perturbation theory \cite{K}, assumption (D3') implies that there exist $\xi_0\in]0,\pi[$, $\varepsilon_0>0$, an analytic curve $\lambda:\ [-\xi_0,\xi_0]\to B(0,\varepsilon_0)$ such that, when $|\xi|\leq\xi_0$, $\sigma(L_\xi)\cap B(0,\varepsilon_0)\ =\ \left\{\lambda(\xi)\right\}$. 
and, when $|\xi|\leq\xi_0$, dual right and left eigenfunctions $q(\xi, \cdot)$ and $\tilde q(\xi, \cdot)$ of $L_\xi$ associated with $\lambda(\xi)$, analytic in $\xi$ and such that $q(0,\cdot)=\bar u_x$. We may thus define the critical spectral projection $P(\xi)$ through 
$$
P(\xi)g=q(\xi)\langle \tilde q(\xi),g\rangle_{L_{\rm per}^2([0,1])}.
$$
Symmetry of the spectrum and (D2') imply that $\lambda$ satisfies \eqref{e:surfacessh}.

Now, we choose a smooth cutoff function $\phi:[-\pi,\pi)\to[0,1]$ such that
\[
\phi(\xi)=\left\{
            \begin{array}{ll}
              1, & \textrm{ if }|\xi|\leq \xi_0/2 \\
              0, & \textrm{ if }|\xi|\geq \xi_0
            \end{array}
          \right.
\]
and decompose the linearized solution operator $S(t)=e^{Lt}$ as $S(t)=S^I(t)+S^{II}(t)$, with
$$
\begin{aligned}
S^I(t)g(x)&:=\int_{-\pi}^\pi e^{i\xi x}\phi(\xi)[P(\xi)e^{L_{\xi}t}\check{g}(\xi,\cdot)](x)d\xi\\
&=\int_{-\pi}^\pi e^{i\xi x}\phi(\xi)e^{\lambda(\xi)t}q(\xi,x)\langle \tilde q(\xi,\cdot),\check{g}(\xi,\cdot)\rangle_{L_{\rm per}^2([0,1])} d\xi\\
S^{II}(t)g(x)&:=\int_{-\pi}^\pi e^{i\xi x}[\left(1-\phi(\xi)P(\xi)\right)e^{L_{\xi}t}\check{g}(\xi,\cdot)](x)d\xi.
\end{aligned}
$$
Note, by the way, that the Green kernel
\[
G^I(x,t;y):=S^I(t)\delta_y(x)
\]
is given by
\[
G^I(x,t;y)=\int_{-\pi}^\pi e^{i\xi(x-y)}\phi(\xi)e^{\lambda(\xi)t}q(\xi,x)\tilde{q}(\xi,y)^*d\xi;
\]
see the proof of Lemma \ref{l:G1rep} for some details. General considerations about semigroups lead for $S^{II}$ to the following bounds,
whose proof is essentially the same as Proposition \ref{p:hfbounds}.

\begin{prop}\label{p:hfbounds-sh}
Under assumptions (H1')-(H2') and (D1')--(D3'), there 
exist constants $C,\theta>0$ such that for all $2\leq p\leq\infty$,
$0\leq 4l_1+l_2=l_3\leq K+1$, $0\leq 4m_1+m_2+m_3\leq K$, $K$ as in (H1'), $s=0,1$ and $t>0$
\begin{align*}
\left\|\d_t^{l_1}\d_x^{l_2} S^{II}(t)\d_x^{l_3}g\right\|_{H^s(\RM)}&\leq Ct^{-\frac{4l_1+l_2+l_3}{4}}e^{-\theta t}\|g\|_{H^s(\RM)},\\
\left\|\d_t^{m_1}\d_x^{m_2} S^{II}(t)\d_x^{m_3}g\right\|_{L^p(\RM)}&\leq Ct^{-\frac{1}{4}\left(\frac{1}{2}-\frac{1}{p}\right)-\frac{4m_1+m_2+m_3}{4}}e^{-\theta t}\|g\|_{L^2(\RM)},\\
\left\|S^{II}(t)g\right\|_{L^p(\RM)}&\leq Ce^{-\theta t}\|g\|_{H^1(\RM)}.
\end{align*}
\end{prop}
Aiming at identifying up to remainders terms of form 
$S^I(t)g$ with terms of form $\bar u_x h$, 
we decompose further by expanding $q(\xi,\cdot)=\bar u_x+\mathcal{O}(\xi)$ and thus write $G^I(x,t;y)=\bar u_x(x) \tilde e(x,t;y)+\tilde G^I(x,t;y)$ with
$$
\begin{aligned}
\tilde e(x,t;y)&=\int_{-\pi}^\pi e^{i\xi(x-y)}\phi(\xi)e^{\lambda(\xi)t}\tilde{q}(\xi,y)^*d\xi\\
\tilde G^I(x,t;y)&=\int_{-\pi}^\pi e^{i\xi(x-y)}\phi(\xi)e^{\lambda(\xi)t}\left(q(\xi,x)-q(0,x)\right)\tilde{q}(\xi,y)^*d\xi.
\end{aligned}
$$
Since 
$$
\begin{aligned}
\int_\RM\tilde e(x,t;y)g(y) dy&=\int_{-\pi}^\pi e^{i\xi x}\phi(\xi)e^{\lambda(\xi)t}\langle\tilde{q}(\xi,\cdot),\check{g}(\xi,\cdot)\rangle_{L_{\rm per}^2([0,1])}d\xi\\
\int_\RM\tilde G^I(x,t;y) dy&=\int_{-\pi}^\pi e^{i\xi x}\phi(\xi)e^{\lambda(\xi)t}\left(q(\xi,x)-q(0,x)\right)
\langle\tilde{q}(\xi,\cdot),\check{g}(\xi,\cdot)\rangle_{L_{\rm per}^2([0,1])}d\xi,
\end{aligned}
$$
a direct application of Hausdorff-Young inequality \eqref{hy} yields the following bounds.

\begin{prop}\label{GIdecay1-sh}
Under the assumptions (H1')-(H2') and (D1)--(D3), the low-frequency Green function $G^I(x,t;y)$ of \eqref{e:lin-sh} may be decomposed
as
\[
G^{I}(x,t;y)=\bar{u}'(x)\tilde{e}(x,t;y)+\widetilde{G}^I(x,t;y)
\]
where for all $t\geq0$ and $1\leq q\leq2\leq p\leq\infty$,  $0\leq j,l,s,j+l\leq K+1$, the residual $\widetilde{G}^I(x,t;y)$ satisfies
\begin{equation}\label{e:G1bds-sh}
\left\|\int_\RM\partial_x^j\partial_t^l\partial_y^s\widetilde{G}^I(\cdot,t;y)g(y) dy\right\|_{L^p(\RM)}
\leq C \left(1+t\right)^{-\frac{1}{2}\left(\frac1q-\frac1p\right)-\frac{1}{2}-\frac{l}{2}}\|g\|_{L^q(\RM)}.
\end{equation}
Furthermore, for all $t\geq0$, $1\leq q\leq 2\leq p\leq\infty$, $0\leq j,l,s,j+l\leq K+1$,we have
\begin{equation}\label{e:ebds-sh}
\left\|\int_\RM\partial_x^j\partial_t^l\partial_y^s \tilde{e}(\cdot,t;y)\ g(y) dy\right\|_{L^p(\RM)}
\leq C(1+t)^{-\frac{1}{2}\left(\frac1q-\frac1p\right)-\frac{(j+l)}{2}}\|g\|_{L^q(\RM)}.
\end{equation}
\end{prop}

Finally, we combine the various 
exponential and algebraic 
bounds derived above to obtain decay estimates on the Green function
\be\label{e:lindelta-sh}
G(x,t;y)=S(t)\delta_y(x)
\ee
associated with the full solution operator $S(t)=e^{Lt}$. To this end, let $\chi$ be a smooth real valued 
cutoff 
function defined on $[0,\infty)$ such that 
$0\leq \chi\leq 1$ and 
\[
\chi(t)=\left\{
         \begin{array}{ll}
           0, & \textrm{ if }0\leq t\leq 1 \\
           1, & \textrm{ if }t\geq 2
         \end{array}
       \right.
\]
and define
\[
e(x,t;y):=\chi(t)\tilde{e}(x,t;y).
\]
Using the Hausdorff-Young inequality \eqref{hy} and the triangle inequality
\[
\left\|\int_{\RM}F(\cdot,t,y)g(y)dy\right\|_{L^p(\RM)}\leq C\|g\|_{L^1(\RM)}\sup_{y\in\RM}\left\|F(\cdot,t,y)\right\|_{L^p(\RM)},
\]
we obtain the following estimates.

\begin{cor}\label{e:Gbdsfinal-sh}
Under assumptions (H1')-(H2') and (D1')--(D3'), the Green function $G(x,t;y)$ of \eqref{e:lindelta-sh} decomposes
as
\[
G(x,t;y)=\bar{u}'(x)e(x,t;y)+\widetilde{G}(x,t;y)
\]
where, for some $C,\theta>0$ and all $t>0$, $2\leq p\leq\infty$, and $0\leq j,l,s,j+l\leq K+1$ we have
\begin{align}
\left\|\int_\RM\d_y^s\widetilde{G}(\cdot,t;y)g(y)dy\right\|_{L^p(\RM)}
&\leq
Ct^{-\frac{1}{4}\left(\frac{1}{2}-\frac{1}{p}\right)-\frac{s}{4}}
\left(1+t\right)^{-\frac{1}{4}\left(\frac{3}{2}-\frac{1}{p}\right)-\frac12+\frac{s}{4}}\|g\|_{L^1(\RM)\cap L^2(\RM)}
\label{finalGbdyder-sh}\\
\left\|\int_\RM\widetilde{G}(\cdot,t;y)g(y)dy\right\|_{L^p(\RM)}&\leq 
C\left(1+t\right)^{-\frac{1}{2}\left(1-\frac{1}{p}\right)-\frac12}\|g\|_{L^1(\RM)\cap H^1(\RM)}
\label{finalGbd1-sh}\\
\left\|\int_\RM\d_x^j\d_t^l\d_y^s e(\cdot,t;y)g(y)dy\right\|_{L^p(\RM)}&
\leq C\left(1+t\right)^{-\frac{1}{2}\left(1-\frac{1}{p}\right)-\frac{(j+k)}{2}}\|g\|_{L^1(\RM)}.\label{finalebds-sh}
\end{align}
and $e(x,t;y)\equiv 0$ for $0\leq t\leq 1$.
\end{cor}

%
%

\subsection{Nonlinear analysis}

Given the linearized bounds on the linearized solution operator $S(t)=e^{Lt}$ derived in the previous section, we are now in position to consider
the effect of the small nonlinear terms that were omitted in obtaining the linearized equation. For this purpose, we introduce $v(x,t)=\tilde u(x+\phi(x,t),t)-\bar{u}(x)$ as in \eqref{pertvar-sh}. Direct calculations similar to those detailed in the proof of Lemma \ref{l:cancel1} provide the following lemma.

\begin{lem}\label{l:cancel1-sh}
For $v$ as above, the equation is 
\be\label{veq-sh}
(\d_t-L)(v-\psi\bar u_x)=\cN,\quad \cN=\cQ+\d_x \cR+\d_t \cS+\cT
\ee
where
$$
\begin{aligned}
\cQ&=-\left(f(\bar u+v)-f(\bar u)-df(\bar{u})v\right)\\
\cR&=\psi_t v-2\frac{-\psi_x}{1+\psi_x}v_x
-\frac{-\psi_x}{1+\psi_x}\d_x\left(\frac{1}{1+\psi_x}\d_x\left(\frac{1}{1+\psi_x}v_x\right)\right)
-\d_x\left(\frac{-\psi_x}{1+\psi_x}\d_x\left(\frac{1}{1+\psi_x}v_x\right)\right)\\
&-\d_x^2\left(\frac{-\psi_x}{1+\psi_x}v_x\right)-2\frac{\psi_x^2}{1+\psi_x}\bar u_x
-\frac{-\psi_x}{1+\psi_x}\d_x\left(\frac{-\psi_x}{1+\psi_x}\d_x\left(\frac{1}{1+\psi_x}\bar u_x\right)\right)
-\frac{-\psi_x}{1+\psi_x}\d_x^2\left(\frac{-\psi_x}{1+\psi_x}v_x\right)\\
&-\frac{\psi_x^2}{1+\psi_x}\bar u_{xxx}
-\d_x\left(\frac{-\psi_x}{1+\psi_x}\d_x\left(\frac{-\psi_x}{1+\psi_x}\bar u_x\right)\right)
-\d_x\left(\frac{\psi_x^2}{1+\psi_x}\bar u_{xx}\right)-\d_x^2\left(\frac{\psi_x^2}{1+\psi_x}\bar u_x\right)\\
\cS&=-v\psi_x\\
\cT&=-\psi_x\left((1-r)v+f(\bar u+v)-f(v)\right).
\end{aligned}
$$
\end{lem}

By Duhamel's principle, we obtain the announced implicit representation \eqref{vintegral-sh}, so that we may express
the phase $\psi$ and the nonlinear perturbation variable $v$ implicitly as
\begin{equation}\label{psi1-sh}
\psi(x,t)=-\int_\RM e(x,t;y)v(y,0)dy-\int_0^t\int_\RM e(x,t-s;y)\cN(y,s)dyds
\end{equation}
and
\begin{equation}\label{vimplicit-sh}
v(x,t)=\int_\RM \widetilde{G}(x,t;y)v(y,0)dy+\int_0^t\int_\RM \widetilde{G}(x,t-s;y)\cN(y,s)dy~ds.
\end{equation}
Furthermore, recalling that $e(x,s;y)=0$ for $0\leq s\leq 1$, we find by differentiating \eqref{psi1-sh} that
\begin{equation}\label{psider-sh}
\partial_t^j\partial_x^k\psi(x,t)=-\int_\RM\partial_t^j\partial_x^k e(x,t;y)v(y,0)dy-\int_0^t\int_\RM\partial_t^j\partial_x^ke(x,t-s;y)\cN(y,s)dy~ds.
\end{equation}
for $0\leq j\leq 1$ and $0\leq k\leq K+1$. To apply a standard contraction-mapping argument and solve locally in time our Cauchy problem \eqref{vimplicit-sh}-\eqref{psider-sh} with $(v,\psi_t,\psi_x)\in H^K(\RM)\times H^K(\RM)\times H^{K+1}(\RM)$, we only need besides our linear bounds the following nonlinear damping energy estimate whose proof is entirely similar to the one of Proposition \ref{p:damping}.

\begin{prop}\label{p:damping-sh}
Assuming (H1'), there exist positive constants $\theta$, $C$ and $\eps_0$ such that if $v$ and $\psi$ solve \eqref{veq-sh} on $[0,T]$ for some $T>0$ and 
$$
\sup_{t\in[0,T]}\|(v,\psi_x)(t)\|_{H^K(\RM)}+\sup_{t\in[0,T]}\|\psi_t(t)\|_{H^{K-1}(\RM)}\leq\eps_0
$$ 
then, for all $0\leq t\leq T$,
\be\label{e:damping-sh}
\begin{aligned}
\|v(t)\|^2_{H^K(\RM)}
&\leq Ce^{-\theta t}\|v(0)\|_{H^K(\RM)}^2\\
&+C\int_0^t e^{-\theta(t-s)}
\left(\|v(s)\|^2_{L^2(\RM)}+\|\psi_x(s)\|_{H^{K+1}(\RM)}^2+\|\psi_t(s)\|_{H^{K-2}(\RM)}^2\right)ds.
\end{aligned}
\ee
\end{prop}

With the above preparations in hand, we are now prepared to state the main technical lemma leading to the proof of Theorem \ref{main-sh}.
To this end, associated with the solution $(u,\psi_t,\psi_x)$ of the integral system \eqref{vimplicit-sh} and \eqref{psider-sh} considered
in the previous section, define
\begin{equation}\label{eta-sh}
\eta(t):=\sup_{0\leq s\leq t}\left\|(v,\psi_t,\psi_x,\psi_{xx})(s)\right\|_{H^K(\RM)}(1+s)^{3/4}.
\end{equation}
By standard short-time $H^K(\RM)$ existence theory, the function 
$\eta$ 
is continuous so long as it remains sufficiently small.  Using the linearized estimates of Section \ref{s:linearizedestimates-sh} we now
prove that if $\eta(0)$ is 
sufficiently 
small then $\eta(t)$ remains small for all $t>0$.

\begin{lemma}\label{l:iteration-sh}
Under assumptions (H1')-(H2') and (D1')--(D3'), there exist positive constants $C$ and $\varepsilon_0$ such that if $v(0)$ is such that for some $T>0$
$$
E_0:=\|v(0)\|_{L^1(\RM)\cap H^K(\RM)} \leq \varepsilon_0\quad\textrm{and}\quad\eta(T) \leq \varepsilon_0
$$ 
then, for all $0\leq t\leq T$,
\be\label{eq:sclaim-sh}
\eta(t)\le C(E_0+\eta(t)^2).
\ee
\end{lemma}

\begin{proof}
First, 
note that under the above smallness assumption
%
\[
\left\|\cN(t)\right\|_{L^1(\RM)\cap L^2(\RM)}\leq C\eta(t)^2(1+t)^{-3/2}.
\]
Applying now the bounds of Corollary 
\ref{e:Gbdsfinal-sh} 
to representations \eqref{vimplicit-sh}--\eqref{psider-sh}, we obtain
for any $2\leq p\leq\infty$ the bounds
\begin{equation}\label{vintbd1-sh}
\begin{aligned}
\left\|v(t)\right\|_{L^p(\RM)}&\leq C\left(1+t\right)^{-\frac{1}{2}\left(1-\frac{1}{p}\right)-\frac{1}{2}}E_0\\
&\quad +C\eta(t)^2\int_0^t\left(t-s\right)^{-\frac{1}{4}\left(\frac{1}{2}-\frac{1}{p}\right)}
               \left(1+t-s\right)^{-\frac{1}{4}\left(\frac{3}{2}-\frac{1}{p}\right)-\frac{1}{2}}(1+s)^{-\frac{3}{2}}ds\\
&\leq C\left(E_0+\eta(t)^2\right)(1+t)^{-\frac{1}{2}\left(1-\frac{1}{p}\right)-\frac{1}{2}}
\end{aligned}
\end{equation}
and
\begin{equation}\label{psiintbd1-sh}
\begin{aligned}
\left\|(\psi_t,\psi_x)\right\|_{W^{K+1,p}(\RM)}&\leq C(1+t)^{-\frac{1}{2}(1-\frac{1}{p})-\frac{1}{2}}E_0\\
&\quad+C\eta(t)^2\int_0^t\left(1+t-s\right)^{-\frac{1}{2}(1-1/p)-\frac{1}{2}}(1+s)^{-3/2}ds\\
&\leq C\left(E_0+\eta(t)^2\right)(1+t)^{-\frac{1}{2}\left(1-\frac{1}{p}\right)-\frac{1}{2}}.
\end{aligned}
\end{equation}
Finally, since the smallness assumption guarantees that we can apply 
Proposition \eqref{p:damping-sh} it follows that
\begin{align*}
\|v(t)\|_{H^K(\RM)}^2&\leq Ce^{-\theta t}E_0^2+C\left(E_0+\eta(t)^2\right)^2\int_0^t e^{-\theta(t-s)}(1+s)^{-3/2}ds\\
&\leq Ce^{-\theta t}E_0^2+C\left(E_0+\eta(t)^2\right)^s(1+t)^{-3/2}\\
&\leq C\left(E_0+\eta(t)^2\right)^2(1+t)^{-3/2},
\end{align*}
which, together with \eqref{psiintbd1-sh} ,
completes the proof.
\end{proof}

Assuming without loss of generality that the constant $C$ in \eqref{eq:sclaim-sh} is larger than $1$, we obtain, since $\eta$ is continuous and $\eta(0)=\|v(0)\|_{H^K(\RM)}\leq E_0$, by continuous induction that if $4C^2E_0<1$ then $\eta(t)\leq 2CE_0$ for all $t\geq0$.
Recalling \eqref{vintbd1-sh}, \eqref{psiintbd1-sh} and the definition of $\eta$ yields \eqref{eq:smallsest-sh}.

Using Corollary \ref{e:Gbdsfinal-sh} again gives for $2\leq p\leq\infty$
\begin{align*}
\left\|\psi(t)\right\|_{L^p(\RM)}&\leq CE_0(1+t)^{-\frac{1}{2}\left(1-\frac{1}{p}\right)}+CE_0^2\int_0^t\left(1+t-s\right)^{-\frac{1}{2}\left(1-\frac{1}{p}\right)}(1+s)^{-\frac{3}{2}}ds\\
&\leq CE_0(1+t)^{-\frac{1}{2}\left(1-\frac{1}{p}\right)}.
\end{align*}
Now by definition we have that
\[
\tilde{u}(x,t)-\bar{u}(x)=v(x,t)+\left(\tilde{u}(x,t)-\tilde{u}(x+\psi(x,t),t)\right),
\]
hence 
$$
\|\tilde{u}(t)-\bar{u}\|_{L^p(\RM)}\leq \|v(t)\|_{L^p(\RM)}
+\|\bar u_x\|_{L^\infty([0,1])}\|\psi(t)\|_{L^p(\RM)}
\leq C E_0(1+t)^{-\frac{1}{2}\left(1-\frac{1}{p}\right)}.
$$
This completes the proof of 
Theorem \ref{main-sh}, establishing the nonlinear $L^1\cap H^K\to L^p$ asymptotic stability of the underlying
periodic traveling wave $\bar{u}$ under the structural and spectral assumptions (H1')-(H2') and (D1')--(D3').

%

\subsection{Numerical stability analysis}\label{a:shnum}

We demonstrate the numerical stability verification method described in Section \ref{sec:lowfreqanal} in the context of the Swift-Hohenberg equation.
To match with \cite{Sc}, we choose $f(u)=u^3$, set $\eps=\sqrt{r}$ and rewrite \eqref{e:sh} as
\be
\label{SHpde}
\partial_tu= (\varepsilon^2-1)u-2\partial_x^2u-\partial_x^4u-u^3.
\ee
Stationary traveling wave solutions of \eqref{SHpde} satisfy,
\be
(\varepsilon^2-1)u-2u''-u''''-u^3=0.
\ee
When $\varepsilon$ is small there is a three parameter family of 
stationary periodic 
solutions approximated by
\be
u_0(\omega,\phi,\varepsilon)[x]= 2\Re(\varepsilon(\sqrt{1-4\omega^2}/\sqrt{3})e^{i(1+\varepsilon\omega)x}e^{i\phi})
\ee
and 
with period $X = 2\pi/(1+\varepsilon\omega)$; see \cite{Sc}.
About $\bar u$, a periodic stationary wave solution of \eqref{SHpde}, the linearized evolution is described by $(\d_t-L)=0$ where
\be
\label{SHeigequ}
L u:= (\varepsilon^2-1)u-2u_{xx}-u_{xxxx}-3\bar u^2 u,
\ee

Note that the fact we deal with stationary solutions make $L$ and $L_0$ self-adjoint. In particular, assumption (D1') of the previous section is a consequence of (D2') and the spectrum of $L$ lies in $\RM$
so that any transition to stability/instability must be marked by an eigenvalue  passing through the origin for
some possibly non-zero Bloch frequency. Furthermore, in \cite{Eck,M1,MS} it was analytically verified that for $\varepsilon>0$ small in \eqref{SHpde} the solutions 
approximated by $u_0(\omega,\phi,\varepsilon)$ are spectrally stable, i.e. $\sigma_{L^2(\RM)}(L)\subset(-\infty,0]$, 
provided that
\begin{equation}\label{eckbdry}
\left|4\omega^2\right|<\frac{1}{3}+\mathcal{O}(\varepsilon).
\end{equation}
Below, we apply an appropriate modification of the numerical protocol introduced in Section \ref{sec:lowfreqanal} above to
analyze the stability of such small amplitude periodic wave trains of \eqref{SHpde}.

Modifying adequately the proof of Lemma \ref{l:hfbds}, we find that $\sigma_{L^2(\RM)}(L)\subset (-\infty,\varepsilon^2]$. 
%
We first check that there is only one spectral branch of the associated Bloch operators $L_\xi$ bifurcating from the origin at $\xi=0$,
and then track its location. In this case, the 
Evans function $D(\lambda,\xi)$ can be expanded to 
second order as
\[
D(\lambda,\xi)=a_0\lambda+a_1\xi+a_2\lambda^2+a_3\lambda\xi+a_4\xi^2+\mathcal{O}\left(|\lambda|^3+|\xi|^3\right).
\]
%
Assuming that indeed only one critical spectral branch bifurcates from the $(\lambda,\xi)=(0,0)$ state, 
expanding 
%
the
critical spectral branch as $\lambda(\xi)=\alpha\xi+\beta\xi^2+\mathcal{O}(|\xi|^3)$
it follows 
%
that the coefficients $\alpha$ and $\beta$ are related to the above Evans function
expansion via the formulas 
\[
\alpha=-\frac{a_1}{a_0}\quad\textrm{ and } \quad\beta=-\frac{a_2\alpha^2+a_3\alpha+a_4}{a_0}.
\]
Recalling Remark \ref{r:evenoddspec} 
together with the fact that here the spectrum is real, 
we note that $\alpha=0$ so that 
$\beta=-\frac{a_4}{a_0}$, 
and thus there are only two coefficients of the Taylor expansion of the Evans function that need to be computed in order to check
stability for $|(\lambda,\xi)|\ll 1$ with stability in this 
small-Floquet 
regime corresponding to the condition that $\beta<0$.


For $\varepsilon = 0.187$, $\omega = 0.1$, and $X = 6.1678$ we compute the Evans function on a contour 
$\partial (B(0,R)/B(0,r_0)\cap \left\{\lambda\,|\,\Re(\lambda)\geq0\right\})$ 
where $R = \varepsilon^2$ and 
$r_0=10^{-3}$ 
using 1001 Floquet parameters, $\xi$, evenly spaced in the interval $[-\pi/X,\pi/X]$, and an adaptive mesh in the $\lambda$ contour requiring relative error be less than 0.2 between consecutive contour points. We then use the Taylor coefficients method to find $\beta = -3.874+0.000i$ via a convergence study requiring convergence in relative error between successive iterations of $\beta$ be less than 0.01. As a check on the accuracy of our computation, we note that $\alpha = 5.2664e-11+1.9970 e-07i  \approx 0$. Next we determine the maximum modulus root of the Evans function for 
$|\xi| = 0.2$, $\xi\in \mathbb{C}$ using a $\lambda$ contour of radius 2, yielding the bound, $\max(\lambda(\xi)) =1.1099$. 

We find that $k_0 = 1.7420e-4$ suffices for breaking the Floquet parameter interval into small and large modulus values. For 
$|\xi|>k_0$ 
we compute the Evans function on a semicircle of radius 0.1 passing through the origin requiring that the relative error between consecutive contour points not exceed 0.2. Then we compute the Evans function on the same semicircle shifted left by  $ 1e-4$, this time taking 
$|\xi|< k_0$. 
In all computations, we find the winding number is consistent with spectral stability.

We also find the stability boundary, though this time we do not use Taylor coefficients because the algorithm for $\alpha$ and $\beta$ becomes numerically poor conditioned since $a$ becomes small near much of the stability boundary; see Remark \ref{r:degenerate} above.
To circumvent this difficulty, we find the stability boundary as given by computing the 
Evans function on a contour 
$\partial (B(0,\varepsilon^2)/B(0,r_0)\cap \left\{\lambda\,|\,\Re(\lambda)\geq0\right\}), \ r_0< \varepsilon^2,$ 
to determine the presence of spectra as $\omega$ varies, and then take the limit of the location of the 
stability boundary found in this way as 
$r_0\to 0$. 
We take care to choose via a convergence study a 
sufficiently tight tolerance setting in the RKF integration routine to obtain accurate results. 
We plot the stability boundary and a comparison with the analytical bound given by Eckhaus, Schneider, and 
Mielke \cite{Eck,M1,Sc} in Figure \ref{f:SHstabbd}.
%

Finally, we verify numerically the eigenvalue picture given in Figure 1 of \cite{Sc} by 
computing the Evans function on a circle and then solving for the roots using the method of 
moments as described in Section \ref{meth:mom}; see Figure \ref{f:SHeigpic}.

\begin{figure}[htbp]
\begin{center}
(a)\includegraphics[scale=.35]{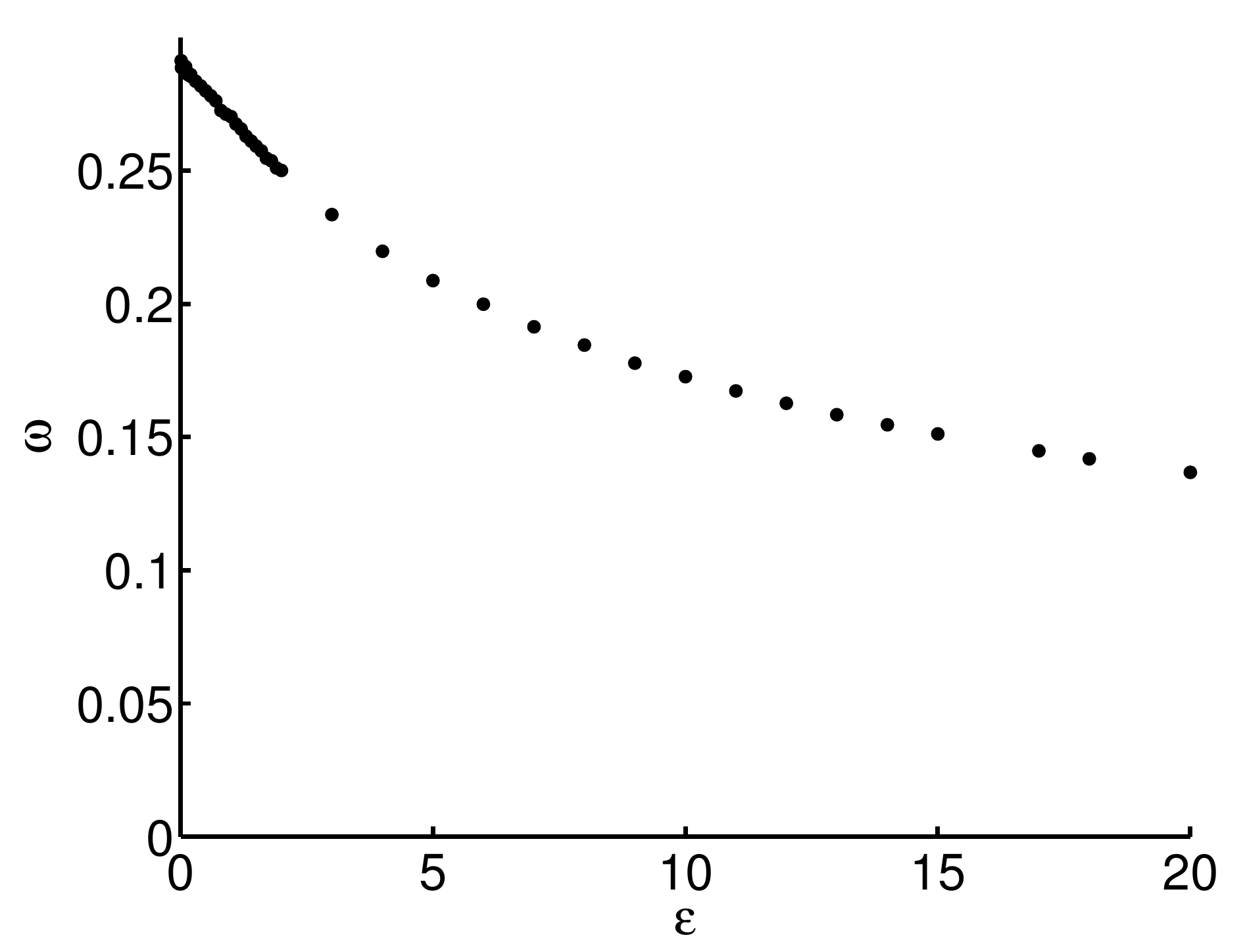}\quad (b)\includegraphics[scale=0.35]{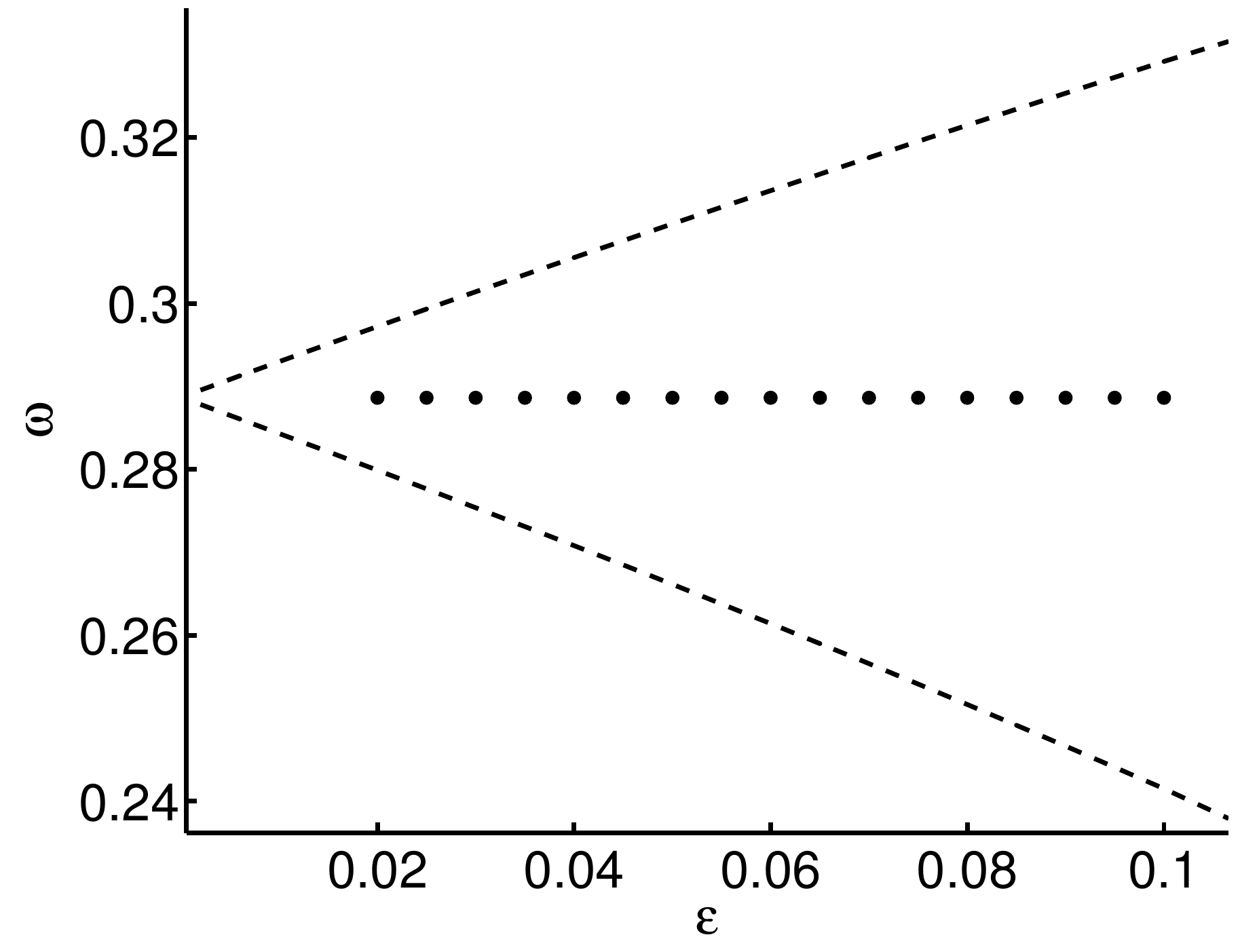}
\caption{In (a) we plot the stability boundary for the Swift-Hohenberg equation, and in (b) we compare the numerical stability boundary with the analytic one in \eqref{eckbdry}, plotted here using dashed lines, found by Eckhaus, Schneider, and Mielke
\cite{M1,MS}.
This provides a check on our numerics against a well known result. 
The analytic curve is found by rigorous perturbation analysis in \cite{Eck,M1,MS}. 
Our numerics  give a nice extension to large amplitudes in the non-perturbative regime. 
We find the stability boundary with relative error tolerance of 0.01. In (a) we do a convergence study as 
$r_0\to 0$
as described previously. In (b) we compute the Evans function on the contour 
$\partial (B(0,\varepsilon^2)/B(0,r_0)\cap\left\{\lambda\,|\,\Re(\lambda)\geq0\right\}), \ r_0 = 1e-8,$ 
with tolerance set at 1e-12 in the RKF routine. 
}\label{f:SHstabbd}
\end{center}
\end{figure}

\begin{figure}[htbp]
\begin{center}
\includegraphics[scale=.35]{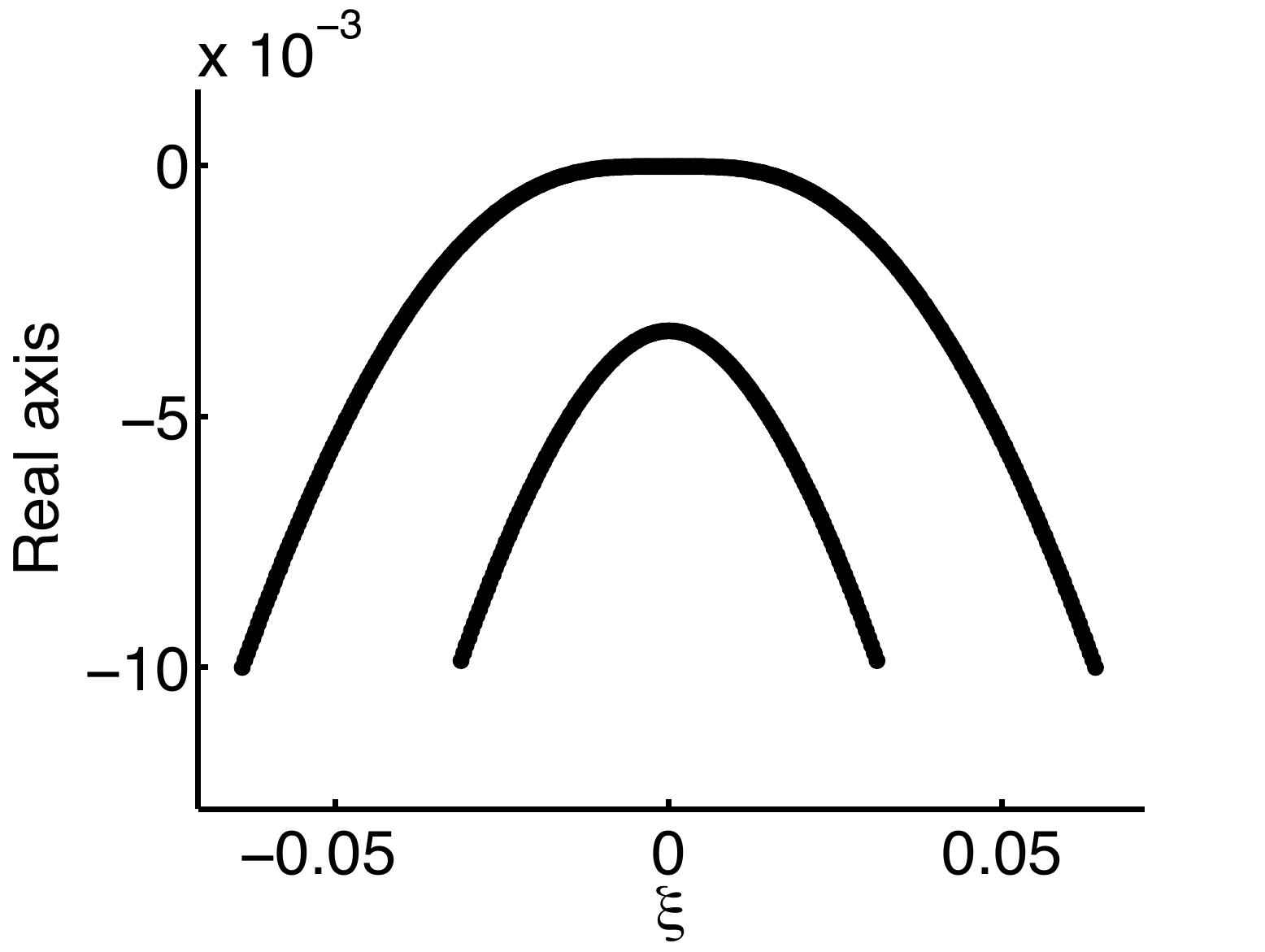}
\caption{Plot of two smallest eigenvalues of the Swift-Hohenberg system against the Floquet parameter 
$\xi$ 
for $\varepsilon = 0.05$, $\omega = 0.28858$; cf. Figure 1 in \cite{Sc}.}\label{f:SHeigpic}  
\end{center}
\end{figure}

\section{Appendix: Computational statistics}\label{s:stats}

In this short appendix we detail computational statistics for some typical values 
in the numerical studies reported in previous sections. 
All numerical computations were performed on either a Mac Pro with 2 Quad-Core Intel Xeon processors with speed 2.26 GHz,
the super computer Quarry\footnote{This is a cluster composed of IBM HS21 Bladeservers and IBM iDataPlex dx340 rack-mounted servers.}
at Indiana University, an Intel(R) Pentium(R) 4 CPU with speed 2.80 GHz running Windows XP, or a MacBook 
with an Intel Core 2 Duo processor with speed 2.0 GHz.

Carrying out the stability study for the classical Kuramoto-Sivashinsky equation with $q = 6$ and $X=6.3$, corresponding
to the $q=6$ row of Table \ref{stabstudy:0eps}, with the MacPro using 8-cores via MATLAB's parallel computing toolbox, it took 11.46 seconds to perform the High Frequency Evans function computation, 159.0 seconds to compute the Taylor coefficients $\{\alpha_j\}_{j=1,2}$ and $\{\beta_j\}_{j=1,2}$, 26.85 seconds
to compute $\max_{\xi} | \lambda_j(\xi)|$, 169.5 seconds to compute the Evans function when\footnote{Here, we are following
Section \ref{sec:lowfreqanal}.  In particular, $k_0$ is the low-frequency cutoff in Step 2(b) of the numerical procedure
outlined in Section \ref{sec:lowfreqanal}.} $k_0<|\xi|< \pi/X$, and 40.89 sec to compute the Evans function for $|\xi|<k_0$. For these same parameters computing SpectrUW on the Windows machine using the Maple kernel took 26 seconds using 5 Fourier modes and 30 Floquet parameters. It took 55 seconds using 10 Fourier modes and 30 Floquet parameters. 

The stability study computation times for the Swifth-Hohenberg equation, those discussed in Appendix \ref{a:shnum} 
using the MacPro with 8-cores are as follows. The High Frequency Evans function computation took 5.638 seconds, the Taylor coefficients took 81.73 seconds to compute,  finding $\max |\lambda_j(\xi)|$ took 1.110 seconds, computing the Evans function for  $k_0<|\xi|<\pi/X$ took 11.30 seconds, and computing the Evans function for  $|\xi| < k_0$ took 8.589 seconds. 

The stability diagram in Figure \ref{f:gKSstabislands} for the full parameter study of the generalized
Kuramoto-Sivashinsky equation took roughly 6 days on 6 nodes with 8 cores each on 
Quarry using MATLAB'S parallel computing capabilities. 
We took steps sizes of 0.25 in $\varepsilon$ and 0.1 in period $X$ to check for stability.
Furthermore, the stability picture in Figure \ref{f:SHstabbd} for the Swift-Hohenberg equation
took roughly 3 days running on the MacPro.  Finally, concerning the time evolution 
studies in Section \ref{num:whitham}, those over time intervals of less than 50 took
between 5-30 minutes each to run while those on larger time intervals took approximately 5 hours each to complete
on the Macbook described at the beginning of this section.

%
%
%

\section{Appendix: Numerical Evans algorithm}\label{s:alg}
Here, we briefly describe the numerical algorithm on which our
Evans function computations are based, and explain the terminology
(scaled lifted polar coordinate method; unscaled lifted polar coordinate 
method) used elsewhere in the paper.
Recall the standard definition \eqref{e:evans} of the periodic Evans function
as
$ D(\xi,\lambda):=\det(\Psi(X)-e^{i\xi X}{\bf I}),$
where $\Psi(x,\lambda)$ is the monodromy matrix of
eigenvalue ODE \eqref{e:system}, that is, a
 matrix-valued solution with initial condition $\Psi(0,\lambda)={\bf I}$.

As pointed out in \cite{OZ2}, this is deceptively simple to code
as compared to the homoclinic or front-type Evans function defined
on the whole line \cite{AGJ,PW,GZ}, in the sense that there are no issues with infinite domains and imposition of asymptotic boundary conditions at 
$x=\pm \infty$, so that one may in a matter of moments develop a naive
algorithm that is 
serviceable 
for moderate values of $\lambda$ and $X$,
using any standard ODE evolution algorithm.

However, as described in \cite{BJNRZ1}, 
for the type of global parameter exploration carried out here,
inherently involving large $\lambda$ and $X$ values, such a naive algorithm
is essentially useless.
This is most easily seen by the asymptotic relation \cite{G2,SS,Z4}
between the periodic and homoclinic Evans functions as period $X\to \infty$,
as we now describe.

\subsection{Balanced and rescaled Evans functions}
Start with the eigenvalue equation
${\bf Y}'(x;\lambda)=\mathbb{H}(x,\lambda){\bf Y}(x;\lambda),$ 
$x\in [0,X]$,
Away from a finite set of curves $\zeta_j(k)$ determined
by the dispersion relation $ik\in \sigma(\mathbb{H}(0,\lambda))$,
$k\in \RR$, where $ \sigma$ denotes spectrum, the eigenvalues
of $\mathbb{H}(0,\lambda)$
have non-vanishing real part.
For each $j\in\NM$, set 
\[
\mathcal{C}_j:=\left\{\lambda\in\CM\ \middle|\ \lambda=\zeta_j(k)\textrm{ for some }k\in\RM\right\}
\]
and let $\mathcal{C}$ denotes the compliment in $\mathbb{C}$ of the set $\cup_j\mathcal{C}_j$.
On $\mathcal{C}$, we can define the real-valued function $n$
such that $n(\lambda)$ equals the number
of eigenvalues of $\mathbb{H}(0,\lambda)$ with negative real part noting, in particular, that
$n$ is locally constant on each connected component of $\mathcal{C}$.  Similarly, we define
the real valued functions $\alpha_\pm$ such that $\alpha_{+}(\lambda)$, respectively $\alpha_-(\lambda)$,
equals the sum of the positive, respectively negative, real part eigenvalues of $\mathbb{H}(0,\lambda)$: as
above, the functions $\alpha_{\pm}$ are locally constant on each connected component of $\mathcal{C}$.
%
Following \cite{G2}, we note that, by Abel's formula,
$D(\lambda, \xi)$ may be written alternatively as 
\be\label{balancrel}
D(\lambda, \xi)=
\tilde D(\lambda, \xi) e^{-\int_0^{X/2} 
\tr\left( \mathbb{H}(y,\lambda)\right) dy},
\ee
where
\ba\label{balancedef}
\tilde D(\lambda, \xi)&:=\det (\Psi(0)\Psi(-X/2)^{-1}
- e^{i\xi X}\Psi(0)\Psi(X
/2)^{-1})
\ea
is a ``balanced'' periodic Evans function defined symmetrically
about $x=0$.

\begin{definition}\label{e:balD}
For $\lambda\in\mathcal{C}$ and $\xi\in[-\pi/X,\pi/X)$, 
we define the {\it rescaled balanced
periodic Evans function} as
\be\label{e:rescaled}
\check D(\lambda, \xi):=
e^{(\alpha_-(\lambda)- \alpha_+(\lambda)) X/2} e^{-n(\lambda) i\xi X} \tilde D(\lambda, \xi).
\ee
\end{definition}

The rescaling in \eqref{e:rescaled} is designed to cancel
the exponential growth in \eqref{balancedef} with respect to the period $X$ of the pieces
$\Psi(0)\Psi(-X/2)^{-1}$ and $ \Psi(0)\Psi(X/2)^{-1}$
in the constant-coefficient case $\mathbb{H}\equiv \const$.
Indeed, suppose that $\bar u^X$ are a sequence of periodic waves
with period $X$, converging as $X\to \infty$ to a solitary wave,
or homoclinic, solution $\bar u^\infty$, and index the associated sequence of periodic
Evans functions by $D^X(\lambda, \xi)$.
Then, under mild assumptions, we have \cite{Z4}:
\be\label{homlim}
\lim_{X\to \infty} \check
D^X(\lambda, \xi)=D_{\rm hom}(\lambda),
\ee
where $D_{\rm hom}(\lambda)$ denotes the associated homoclinic Evans
function defined in \cite{PW,GZ}.

Comparing with \eqref{balancrel} and \eqref{e:rescaled}, we see that
if $\check D^X$ is uniformly bounded
then $D^X$ and $\tilde D^X$ exhibit exponential growth with respect to $X$.
Thus, by \eqref{homlim}, together with the well-known fact 
(see, e.g., \cite{GZ}) that the homoclinic
Evans function $D_{\rm hom}$, hence also $\check D^X$, 
is uniformly bounded on compact $\lambda$-domains, 
that both the usual and balanced Evans functions $D^X$ and $\tilde D^X$
exhibit exponential growth with respect to $X$ as $X\to \infty$, hence
become numerically impractical for large $X$ or
for moderate periods $X$ and large $\lambda$ (leading to large $\alpha_\pm$)
in the sense that even small variations in $\lambda$ will lead to excessive
winding in the image $D(\lambda, \xi)$ that requires fine resolution to track
for purposes of winding number computations.

{\it For our high-frequency winding number estimates, therefore, it is crucial
to use the rescaled Evans function $\check D$ in place of $D$ or $\tilde D$.}
Near $\lambda=0$ on the other hand there are several curves on which
the spectra of $\mathbb{H}(0)$ becomes imaginary, across which $n$,
$\alpha_-$, and $\alpha_+$ undergo discontinuities.
This does not occur on 
$\{\lambda\,|\,\Re \lambda \geq 0 \}\setminus \{0\}$, 
so does not come into play in high-frequency estimates; however, for
our tracking of spectral curves $\zeta_j(\xi)$ near zero using
analyticity/the method of moments, it is evidently a problem.
{\it Thus, for our low-frequency computations, we use the unrescaled
Evans function, to avoid loss of analyticity that would otherwise
occur near $\lambda=0$,}
a subtle but important point.

\subsection{The lifted Evans function}
Up to this point we have only discussed winding with respect to $\lambda$,
ignoring questions of numerical conditioning.
However, these are quite relevant also for large $X$, given the exponential
growth already described.
Indeed, one faces all of the issues described for homoclinic or front-type
Evans functions in \cite{Z1}, arising from variations in growth rate of
various scalar modes in the solution operator $\Psi(\lambda,x)$, in 
which faster-growing modes dominate slower-growing modes through 
cancellation/loss of significant digits.


Similarly as in the homoclinic case \cite{Br,BrZ,HuZ},
these can be avoided by the device introduced in \cite{BJNRZ1} of
working with ``lifted equations''
for which the periodic Evans function appears as a {\it Wronskian}
of certain bases of solutions.
Namely, we may recall Gardner \cite{G1}, and his simple 
but important observation
\be\label{lift}
D(\xi,\lambda):=\det(\Psi(X)-e^{i\xi X}\Id)
=\det \bp \Psi(X)& e^{i\xi X}\Id\\ \Id & \Id \ep
\ee
expressing $D$ as an exterior product of evolving solutions for equations
\be\label{lifteq}
\bp Y\\\alpha\ep'= \bp \mathbb{H}(\lambda,\cdot\,)Y \\0\ep,
\ee
with data $\bp \Id \\\Id\ep$ at $x=0$ and $\bp e^{i\xi X}\Id\\\Id\ep$
at $x=X$.

Working with \eqref{lift} allows us to apply any of the efficient algorithms that have been developed in the homoclinic Evans function setting for evaluation of Wronskians, in particular the {\it exterior product method} of
\cite{Br,BrZ} in which minors of columns in \eqref{lift} are computed
as evolving exterior products, or the {\it polar coordinate method}
of \cite{HuZ}, in which they are computed instead as evolving orthonormal
subspace/scalar ``radius'' pairs; see \cite{HuZ} for further details.
Both of these methods are by now standard and have been implemented in STABLAB,
a Matlab-based platform for numerical Evans function computations; see \cite{BHZ2}.
These methods can thus readily be implemented once we have made the reformulation
\eqref{lift}, and the resulting algorithm is numerically well-conditioned
by the same analysis 
sed in \cite{Z1} to study the homoclinic case.

The results of this apparently simple change are dramatic, {\it extending the
range of $|\lambda|$ values that can be computed by up to three orders
of magnitude}; see \cite{BJNRZ1}. 
Indeed, the efficiency of the lifted balanced
periodic method appears to be quite comparable with that of the well-established
algorithms used to evaluate the homoclinic Evans function, making the 
numerical Evans function approach a practical alternative to the 
Galerkin methods that have been used in past literature \cite{FST,CKTR,CDK}.

\subsection{Optimization across $\xi$}
Finally, we mention for the reader who may wish to perform such computations independently a
final detail that greatly speeds up computations with varying 
%
Bloch frequency $\xi$.
Namely, rather than computing a new solution of \eqref{lifteq} for each
value of $\xi$, we may, by the decoupled nature of the flow in first and
second coordinates, compute the result for a single solution initialized as
$\bp \Id\\\Id\ep$ at $x=X$, then multiply the first entry of the
resulting solution by $e^{i\xi X}$ to obtain the desired suite of solutions
for varying $\xi$.
This means that we need in practice only compute solutions of the
eigenvalue ODE once for each value of $\lambda$, with variation in 
$\xi$ introduced by computationally negligible linear algebraic manipulations,
for a considerable savings.

See \cite{BJNRZ1} for further discussion/comparison of results.


\section{Appendix: Behavior near stability boundaries}\label{s:behavior}

Finally, we consider the difference in behavior expected as
we cross a stability boundary across which hyperbolicity of
the first-order Whitham system is lost vs. behavior expected
as we cross a stability boundary where strict hyperbolicity is
maintained but the second-order diffusion coefficient changes sign.

Introduce a bifurcation parameter $\eta$, with the stability boundary
assumed to correspond to $\eta=0$.  In the second case, by standard
spectral perturbation theory/separation to first order of modes, we
have
$\lambda_j(\xi,\eta)=ia_j(\eta)\xi - b_j(\eta)\xi^2+ ic_j(\eta)\xi^3-d_j(\eta)\xi^4+\mathcal{O}(\xi^5)$, 
where $a_j$, $b_j$, $c_j$, $d_j$ (by complex symmetry of eigenvalues of real-valued operators) are real and 
$$
b_j= -b_* \eta + o(\eta),
\quad
d_j= d_* + o(1),
$$
so that $\Re \lambda_j(\xi,\eta)\sim \eta b_* \xi^2-d_*\xi^4 $
and (after a brief calculation)  
\be\label{2c}
\max_\xi 
\Re \lambda_j(\xi,\eta)\sim 
\eta^2 .
\ee

In the first case, we note, rather, that, by the connection to the
second-order Whitham approximation,
$\lambda_j(\xi,\eta)= \pm \sqrt{\eta} a_*\xi+ b_* \xi^2 +o(\xi^2)$,
whence 
\be\label{3c}
\max \Re \lambda_j(\xi,\eta)\sim  \eta.
\ee
We may conclude that the transition involving loss of hyperbolicity is
more drastic, featuring exponential growth at rate roughly the square root
of the rate expected for a comparable transition across a boundary involving
loss of diffusivity.

\end{document}